\documentclass{amsart}
\usepackage{amsmath,amssymb,amsthm,mathrsfs,mathtools,hyperref}
\usepackage[capitalise,noabbrev]{cleveref}
\usepackage{dsfont}  
\usepackage{enumerate}
\usepackage{esint}   
\numberwithin{equation}{section}

\usepackage{xcolor}

\hypersetup{colorlinks=true, 	linkcolor=blue!75!black, 	citecolor=red!50!black, 	urlcolor=green!50!black, 	linktoc=all }

\def\NN{\mathbb{N}}
\def\RR{\mathbb{R}}
\def\ZZ{\mathbb{Z}}
\def\TT{\mathbb{T}}

\newtheorem{theorem}{Theorem}[section]
\newtheorem{proposition}[theorem]{Proposition}
\newtheorem{corollary}[theorem]{Corollary}

\newtheorem{lemma}[theorem]{Lemma}
\theoremstyle{definition}
\newtheorem*{definition*}{Definition}
\newtheorem{definition}[theorem]{Definition}

\theoremstyle{remark}
\newtheorem{remark}[theorem]{Remark}

\newcommand{\norm}[1]{\left\|#1\right\|}
\newcommand{\abs}[1]{\left\vert#1\right\vert}

\DeclareMathOperator\Div{div}
\DeclareMathOperator\curl{curl}
\DeclareMathOperator\Id{\rm Id}
\newcommand{\id}{{\rm Id}}
\DeclareMathOperator\dist{dist}
\DeclareMathOperator\supp{supp}

\DeclareMathOperator{\SDiff}{SDiff}
\DeclareMathOperator{\SHom}{SHom}

\usepackage{units} 

\DeclarePairedDelimiter\floor{\lfloor}{\rfloor}  
\DeclarePairedDelimiter\ceil{\lceil}{\rceil}

\begin{document}

\title[Steady 3d Euler flows]{Steady 3d Euler flows via a topology-preserving convex integration scheme}

\author[A.~Enciso, J.~Pe\~nafiel-Tom\'as and D.~Peralta-Salas]{Alberto Enciso, Javier Pe\~nafiel-Tom\'as and Daniel Peralta-Salas}

\address{
\newline
\textbf{{\small Alberto Enciso}} \vspace{0.1cm}
\vspace{0.1cm}
\newline \indent Instituto de Ciencias Matem\'aticas, Consejo Superior de Investigaciones Cient\'ficas, 28049 Madrid, Spain}
\email{aenciso@icmat.es}

\address{
\vspace{-0.25cm}
\newline
\textbf{{\small Javier Pe\~nafiel-Tom\'as}} 
\vspace{0.1cm}
\newline \indent Instituto de Ciencias Matem\'aticas, Consejo Superior de Investigaciones Cient\'ficas, 28049 Madrid, Spain}
\email{javier.penafiel@icmat.es}

\address{
\vspace{-0.25cm}
\newline
\textbf{{\small Daniel Peralta-Salas}} \vspace{0.1cm}
\vspace{0.1cm}
\newline \indent Instituto de Ciencias Matem\'aticas, Consejo Superior de Investigaciones Cient\'ficas, 28049 Madrid, Spain}
\email{dperalta@icmat.es}

\begin{abstract}
Given any smooth solenoidal vector field \(v_0\) on \(\mathbb{T}^3\), we show the existence of infinitely many H\"older-continuous steady Euler flows \(v\) with the same topology as~$v_0$, in certain weak sense. In particular, we show that \(v\) possesses a unique flow of the highest H\"older regularity, which is conjugate to the flow of~\(v_0\) via a volume-preserving H\"older homeomorphism of \(\mathbb{T}^3\). This result extends to the case of Euler equations on toroidal domains, which has applications to the study of plasmas. The proof relies on a novel convex integration scheme incorporating the key idea that the velocity field of the subsolutions must remain diffeomorphic to~$v_0$ at each iteration step.
\end{abstract}

\maketitle

\setcounter{tocdepth}{1}
\tableofcontents

\section{Introduction}
The study of steady solutions to the incompressible 3D Euler equations,
\begin{equation}\label{E.Euler}
	v\cdot \nabla v+\nabla p=0\,,\qquad \Div v=0\,,
\end{equation}
can be traced back to Euler's original paper~\cite{Euler}. It has evolved into a major research topic in fluid mechanics, owing to the role of steady Euler flows as equilibrium states and to their central importance in the long-term dynamics of the system.

It has long been known that the space of steady states is infinite dimensional. On~$\TT^3:=(\RR/\ZZ)^3$, all known continuous steady Euler flows are generalized Beltrami fields (that is, solenoidal fields satisfying $\curl v = f v$ for some function~$f$) or exhibit a Euclidean symmetry. This mirrors the dichotomy for analytic steady Euler flows on the plane recently established by Elgindi, Huang, Said and  Xie~\cite{ElgindiHuang}, although  it is not yet clear whether this is merely coincidental. On the other hand, it is known that Beltrami fields can exhibit topologically complex stream lines~\cite{ASENS} and that they typically feature chaotic regions~\cite{FMS}. This can be used to show that Beltrami fields are typically isolated within the class of analytic steady Euler flows on~$\TT^3$, up to a finite-dimensional curl eigenspace~\cite{Willi}. Note that all these results only control the velocity field on a subset of the fluid region.

Concerning the global structure of steady Euler flows, that is, the question of whether there exists a steady state diffeomorphic to some prescribed solenoidal field~$v_0\in C^\infty(\TT^3,\RR^3)$, not much is known in general. A serious obstruction to the possible orbit structure of steady states was first obtained by Cieliebak and Volkov~\cite{CV}, who observed that a $C^1$~steady Euler flow cannot have any Reeb components. Some elaboration on their argument shows~\cite[Appendix~A]{ARMA} that the space of smooth solenoidal fields~$v_0$ which are not diffeomorphic to a $C^1$~steady Euler flow\footnote{Although \cite[Corollary~A.2]{ARMA} is stated for smooth states, the proof works verbatim for steady solutions of class~$C^1$.} is dense in $H^s(\TT^3,\RR^3)$ for any $s<\frac32$. Braided fields on a solid torus, which arise in plasma physics and cannot exhibit any Reeb components, are typically not diffeomorphic to a smooth steady Euler flow either~\cite{ARMA}.

Our objective in this paper is to show that, nevertheless, for any smooth solenoidal field~$v_0$ there exist many weak steady Euler flows ``with the same topology''. These solutions are H\"older continuous but, as a consequence of the Cieliebak--Volkov obstruction, typically they are not continuously differentiable. The proof is based on a novel convex integration scheme in which the topology of the vector field is preserved in each step of the iteration.

\subsection{Main results}

Let us elaborate on this. As we shall consider steady solutions that are not continuously differentiable, let us start by recalling that a vector field $v\in L^2(\TT^3,\RR^3)$ is a {\em (weak) steady Euler flow}\/ if it is divergence-free in the sense of distributions and if
\[
\int_{\TT^3}v_i\,v_j\, \partial_i w_j\, dx=0
\]
for any solenoidal field $w\in C^\infty(\TT^3,\RR^3)$.

Making precise the idea that the smooth field~$v_0$ and the H\"older continuous field~$v$ are topologically equivalent in some weak sense is not immediate. When one has two fields that are at least Lipschitz, one says that they are homeomorphic (or equivalent up to a homeomorphism) if there exists a homeomorphism~$\phi$ mapping the flow of one field into that of the other. However, in the class of H\"older continuous fields, particle trajectories exist but are possibly nonunique, so the field does not a priori define a unique flow. To fix ideas, let us start by recalling the precise definition of a flow\footnote{This notion of flow implies that the
	map $X$ is a regular Lagrangian flow in the sense of Ambrosio~\cite{Am} and in the stronger
	sense of Di Perna–Lions~\cite{DiPer}.}:

\begin{definition*}
	\label{def flujos}
	Let $\alpha\in(0,1)$, and let $v\in C^\alpha(\TT^3,\RR^3)$ be a vector field that is divergence-free in the sense of distributions. A {\em flow}\/ of $v$ is a map $X\in  C^1_{\rm loc}(\RR, C^0(\TT^3,\TT^3))$ such that:
	\begin{enumerate}
		\item For all $x\in \TT^3$ and all $t\in\RR$, $X$ solves the ODE $\partial_t X(x,t)=v(X(x,t))$ with initial condition $X(x,0)=x$.
		
		\item {$X$ is volume preserving, that is, 
			\(\abs{X_t(A)}=\abs{A}\)
			for any measurable set $A\subset \RR^3$ and any $t\in \RR$, with $X_t:=X(t,\cdot)$.}
		
		\item {$X$ satisfies the group property
			$X_s\circ X_t=X_{s+t}$ for all $s,t\in \RR$.
			In particular, $X_t$ is a volume-preserving homeomorphism of $\TT^3$ for all $t\in \RR$.}
	\end{enumerate} 
\end{definition*}

The way we proceed is by approximation: we will show that there is a sequence of volume-preserving diffeomorphisms $\Phi_q$ which converges to a volume-preserving H\"older homeomorphism~$\Phi$, and such that the (smooth) transformed fields $v_q:=(\Phi_q)_*v_0$ also converge to~$v$. Therefore, if $X^0_t$ denotes the flow of~$v_0$, one can then compose it with the homeomorphism~$\Phi$ to obtain a flow of~$v$,
\begin{equation}\label{E.Xt}
	X_t:=\Phi\circ X^0_t\circ \Phi^{-1}\,.
\end{equation}
This does not rule out the existence of other trajectories passing through this point. However, we will show that $X_t$ is the only flow of~$v$ in a certain H\"older class, so we can refer to $X_t$ as {\em the}\/ most regular flow of~$v$.

For each $\alpha\in(0,1)$, let us respectively denote by $\SDiff(\TT^3)$ and $\SHom^\alpha(\TT^3)$ the groups of smooth diffeomorphisms and $C^\alpha$~homeomorphisms of~$\TT^3$ which preserve volume. Our main result can then be stated as follows:

\begin{theorem}\label{T.main}
	Consider a solenoidal vector field $v_0\in C^\infty(\TT^3,\RR^3)$, whose flow we denote by~$X_t^0$, and fix constants $\tau\in(\sqrt{2/3},1)$ and~$\epsilon>0$.
	{For some $\alpha>0$,} there exist a weak steady Euler flow $v\in C^\alpha(\TT^3,\RR^3)$, a sequence of volume-preserving diffeomorphisms $\Phi_q\in \SDiff(\TT^3)$, and a volume-preserving H\"older homeomorphism $\Phi\in \SHom^\tau(\TT^3)$ such that:
	\begin{enumerate}
		\item The diffeomorphisms $\Phi_q$ converge to~$\Phi$ in~$C^{\tau}(\TT^3,\TT^3)$ as $q\to\infty$.
		\item The sequence of smooth pushed forward vector fields $v_q:=(\Phi_q)_*v_0$ converges to~$v$ in~$C^\alpha(\TT^3,\RR^3)$.
		\item $\|v-v_0\|_{H^{-1}(\TT^3)}+ \|\Phi-\id\|_{C^0(\TT^3)}<\epsilon$		\item $X_t:= \Phi\circ X^0_t\circ \Phi^{-1}$ is the only flow of~$v$ of class $C^0_{\rm loc}(\RR,C^{\tau^2}(\TT^3,\TT^3))\cap C^1_{\rm loc}(\RR,C^{0}(\TT^3,\TT^3))$.
		\item The zero sets $\{v_0=0\}$ and $\{v=0\}$ coincide.
	\end{enumerate}
	Here $\alpha$ depends on~$\tau$ but not on~$v_0$.
\end{theorem}

Several remarks are in order. First, note that property~{(v)} ensures that the supports of~$v$ and~$v_0$ coincide. Therefore, this theorem immediately yields the existence of H\"older continuous steady Euler flows with compact support on~$\RR^3$ that are topologically equivalent (in the sense described above) to any solenoidal field $v_0\in C^\infty_c(\RR^3,\RR^3)$. The study of compactly supported steady Euler flows is a vibrant research topic which has attracted significant recent attention~\cite{Gavrilov,CLV,EDP,EFR,EFRS,JGS,gomez,Hamel2}.

Second, since steady Euler flows and magnetohydrostatic (MHS) equilibria are described by the same equations, this theorem turns out to have a fairly immediate interpretation in the language of plasma physics. It is related to the concept of \emph{topological accessibility} introduced by Moffatt~\cite{Mo85} (see also~\cite{Ko22}). Roughly speaking, our main theorem ensures that H\"older continuous MHS equilibria are ``topologically accessible'' from any smooth solenoidal field on the torus, with the caveat that the $L^2$ norms of the sequence of smooth fields~$v_q$ we use to obtain~$v$ are not decreasing as demanded in~\cite{Mo85}. We will soon say more about this sequence, and about the iteration scheme we use to construct it.

To conclude, an important observation is that our approach remains valid if instead of~$\TT^3$, we consider a toroidal\footnote{That is, $\partial\Omega$ is diffeomorphic to~$\TT^2$.} domain $\Omega\subset\RR^3$, provided that~$v_0$ does not vanish on~$\partial\Omega$. This has a particularly interesting application to the mathematical theory of plasmas. There, an important problem is to identify MHS equilibria (which, as mentioned above, are described by the steady Euler equations~\eqref{E.Euler} in~$\Omega$) that can effectively confine ions during a nuclear fusion reaction. To first approximation, ions move along the integral
curves of~$v$, which play the role of magnetic lines. Therefore, the most basic requirement used for confinement is that the domain~$\Omega$ should be foliated by compact surfaces that are invariant under the flow of~$v$. In physics, this foliation is typically described by the level sets of a ``good flux function'', which is a first integral of~$v$.

When~$\partial\Omega$ is a revolution torus, it is not hard to construct families of axisymmetric smooth steady Euler flows that are tangent to foliations of~$\Omega$ by nested toroidal surfaces, and each foliation becomes singular along a closed embedded circle~$\mathcal C\subset\Omega$. In 1967, Grad famously conjectured~\cite{Grad} that this should be the only possibility: on a toroidal domain, there should be no families of non-axisymmetric smooth steady solutions to~\eqref{E.Euler}, each being tangent to a (singular) foliation by closed surfaces.

By letting $v_0$ range over a family of not topologically equivalent solenoidal fields tangent to a foliation of a toroidal domain~$\Omega$ by tori, a modification of Theorem~\ref{T.main} readily gives nontrivial families of H\"older continuous Euler flows satisfying the requirements of Grad's conjecture in a weak sense. Somewhat informally, what one can prove is the following (for a precise, more general statement, see Theorem~\ref{T.dominios} in the main text):

\begin{theorem}\label{T.Grad}
	Let $\Omega\subset\RR^3$ be a smooth toroidal domain. There exist some $\alpha>0$ and families of (non-axisymmetric) weak steady Euler flows~$v\in C^\alpha(\overline{\Omega},\RR^3)$ on~$\Omega$ such that:
	\begin{enumerate}
		\item Each $v$ admits a unique most regular flow~$X_t$, as in Theorem~\ref{T.main}.
		\item For each~$v$, there exists a foliation of $\Omega$ by tori, singular along a closed embedded circle~$\mathcal C\subset\Omega$ (a periodic orbit of $v$), whose leaves are invariant under the flow~$X_t$.
	\end{enumerate}
\end{theorem}

To put it differently, Theorem~\ref{T.Grad} shows that Grad's conjecture is not true in the class of H\"older solutions, and that the weak steady states that one obtains from this theorem are not isolated. Grad  had already discussed that the low-regularity case should be very different. In particular, he made comments on the case of  stepped-pressure solutions, a class of bounded yet discontinuous weak steady solutions that was rigorously analyzed in~\cite{BL,Alex}. However, Theorem~\ref{T.Grad} is the first result in this context where one can actually control the topology of unforced steady plasmas.

\subsection{Strategy of the proof and overview of existing results}

The proof relies on convex integration, a technique introduced by Nash~\cite{Nash} in the study of the $C^1$ isometric embedding problem. His idea was vastly extended in Gromov's work on the h-principle~\cite{Gromov}, which has had a major impact on a variety of geometric contexts. Some years later, ideas of a similar flavour found striking applications in the calculus of variations and in elliptic regularity theory~\cite{MullerSverak}.  Since the groundbreaking work of De Lellis and Sz\'ekelyhidi~\cite{Bounded,Continuous}, convex integration has made a major impact in the study of nonuniqueness and energy dissipation for non-smooth solutions of the incompressible Euler equations, culminating in the proof of Onsager's conjecture~\cite{Isett,Onsager_final}. For detailed expositions of these and other results on various models in fluid mechanics, we refer the reader to the surveys~\cite{ VicolBAMS, LellisBAMS17} and the papers~\cite{ Colombo, VicolBook,   SQG, CCF, CheskiShy14, Cheskidov23, ChesShv, Faraco, Faraco2,NovackVicol, ShvJAMS}.

To apply convex integration methods to the problem at hand, there are two main challenges one needs to overcome:
\begin{enumerate}
	\item In the known convex integration schemes, the topology of the vector field is not preserved.
	\item {H\"older continuous volume preserving fields do not necessarily have a unique most regular flow.}
	\item The stationary Euler flows that have been constructed using convex integration so far~\cite{Choffrut} are just in~$L^\infty$.
\end{enumerate}
Among these difficulties, the {third} one is mostly technical, while the first {two are}  fundamental. Indeed, the basic strategy in a convex integration scheme for the Euler equations is to obtain the stationary state $v$ as the limit of a sequence of vector fields~$\{v_q\}_{q=0}^\infty$, constructed so that $v_{q+1}$ is obtained from~$v_q$ through the {\em addition}\/ of suitably chosen rapidly oscillating corrections:
\[
v_{q+1}:= v_q+ \text{(highly oscillatory terms)}\,.
\]
These corrections are small in~$C^0$ but large in~$C^1$, so $v_{q+1}$ will generally not be diffeomorphic to~$v_q$ even if~$v_q$ is strongly structurally stable.

Instead, the key ingredient in our scheme is a sequence of volume-preserving diffeomorphisms~$\{\phi_q\}_{q=0}^\infty$, which are indeed constructed essentially as
\[
\phi_{q+1}:= \phi_q+ \text{(highly oscillatory terms)}\,,
\]
with $\phi_0:=\id$. What makes this approach nontrivial is that the corresponding sequence of velocities is 
\[
v_{q+1}(x):=(\phi_{q+1})_*v_q(x)= (D\phi_{q+1}v_q)( \phi_{q+1}^{-1}x)\,.
\]
Thus the transformation~$\phi_q$ also appears in the argument of the oscillatory terms, so we need to analyze the {\em composition}\/ of rapidly oscillating functions. Note that the diffeomorphism which appears in the statement of the theorem can then be written as
\[
\Phi_q:=\phi_q\circ \phi_{q-1}\circ\cdots\circ\phi_1\,.
\]

As mentioned above, a second, less obvious source of complications is that we want to ensure that the weak Euler flow~$v$, which we eventually obtain as the limit of~$v_q$ as $q\to\infty$, has a unique flow in the H\"older regularity class $C^0_{\rm loc}(\RR,C^{\tau^2}(\TT^3,\TT^3))\cap C^1_{\rm loc}(\RR,C^{0}(\TT^3,\TT^3))$. Since individual trajectories can a priori be non-unique, this uniqueness statement hinges on showing that the real $\tau\in(0,1)$ can be chosen sufficiently close to~1. This is not obvious, and even less so because the H\"older exponent of the weak steady state~$v$ will be considerably lower. Achieving this high level of Hölder regularity for the flow requires controlling several additional error terms, all while ensuring that the terms arising from compositions do not cause problems.

Designing a scheme that can absorb all the resulting error terms requires several nonstandard cancellations and the consideration of multiple scales. These cancellations differ significantly from the typical ones. A simple way to see this is that, as the scheme must preserve the orbit structure of the field~$v_0$ in some weak sense, it inherently depends on the structure of the zero set~$\{v_0=0\}$. This is why the zero set of the steady Euler flow~$v$ we eventually construct coincides with that of~$v_0$.

It is interesting to look at this result on the existence of rough steady Euler flows diffeomorphic to~$v_0$ in the light of the traditional approaches to smooth solutions with prescribed topology. By a classical variational formulation for steady states (see e.g.~\cite[Section~2.A]{AK}), a necessary and sufficient condition for the existence of a smooth steady Euler flow diffeomorphic to~$v_0$ is that the restriction of the energy functional $E(v):=\int_{\TT^3}|v|^2\, dx$ to the set
\[
\mathcal O_{v_0}:=\big\{\phi_*v_0: \phi\in\SDiff(\TT^3)\big\}
\]
has a critical point. Note that $\mathcal{O}_{v_0}$ is simply the adjoint orbit of the group of volume-preserving diffeomorphisms passing through~$v_0$. While this characterization remains valid with less stringent regularity assumptions, proving the existence of (smooth or non-smooth) critical points is a very subtle question because there is a serious lack of compactness, and very few results are available. Theorem~\ref{T.main} asserts that, although for most~$v_0$ there cannot be any steady states on the orbit~$\mathcal{O}_{v_0}$ by the Cieliebak--Volkov obstruction, its closure in~$C^\alpha$ always contains infinitely many steady states (in fact, they can be taken arbitrarily close to any given field in~$\mathcal O_{v_0}$  in the $H^{-1}$ norm).

The paper is organized as follows. In~\cref{S.overview} we present the convex integration scheme that our results are based on, and state our induction hypotheses. The proof of these hypotheses, which takes most of the paper, starts in~\cref{S.construction} with the construction of the diffeomorphisms and subsolutions used in the proof. Sections~\ref{section estimates diffeo}--\ref{section prueba prop steps} are respectively devoted to estimates for the diffeomorphism, for the velocity and for the Reynolds stress. With these estimates in hand, we proceed to the proof of Theorem~\ref{T.main}, which is presented in~\cref{S.proof}. The case of toroidal domains requires  some modifications that are discussed in~\cref{S.proofGrad}. The paper concludes with an Appendix where we recall some auxiliary results that are used throughout the article.

\section{Overview of our convex integration scheme}
\label{S.overview}
In this subsection we will discuss the basic ideas of our construction. Our goal is to make this explanation as illustrative and understandable as possible, so our exposition will sometimes be imprecise or overly simplified. All of the details will be discussed in later sections. 

The proof of \cref{T.main} is achieved through an iteration scheme. At each step $q\in \NN$ we will construct a volume-preserving diffeomorphism $\phi_{q+1}$ and we will obtain a new velocity field $v_{q+1}$ as the pushforward of the previous velocity $v_q$ through this diffeomorphism, that is,
\begin{equation}
	v_{q+1}\coloneqq (D\phi_{q+1}v_q)\circ \phi_{q+1}^{-1}.
	\label{pushforward intro}
\end{equation}
Therefore, the velocity $v_q$ is the pushforward of $v_0$ through the diffeomorphism $\Phi_q\coloneqq \phi_q\circ\cdots\circ \phi_1$. Hence, the flow $X^q_t(x)$ of $v_q$ is smoothly conjugate to the flow of the initial field, that is,
\[X^q_t=\Phi_q\circ X^0_t\circ \Phi_q^{-1}.\]

The goal of this process is to make the new velocity $v_{q+1}$ closer to being a solution of the Euler equations. This is quantified by the following standard concept:
\begin{definition}
	A subsolution is a smooth triplet $(v,p,R)$ consisting of a vector field $v$, a scalar $p$ and a symmetric matrix $R$, which is usually called \textit{Reynolds stress}, satisfying
	\begin{equation}
		\label{def subsolution}
		\begin{cases}
			\Div(v\otimes v)+\nabla p=\Div R, \\
			\Div v=0.
		\end{cases}
	\end{equation}
\end{definition}
Thus, our objective is to construct a sequence of subsolutions $\{(v_q,p_q,R_q)\}_{q=0}^\infty$ that converges uniformly to certain triplet $(v,p,0)$. We see that $v$ will then be a weak solution of the Euler equations.

In traditional convex integration schemes, the new velocity is constructed through the addition of a highly oscillatory perturbation $w_{q+1}$, that is,
\[v_{q+1}=v_q+w_{q+1}.\]
Inserting this ansazt in \cref{def subsolution}, we see that the new Reynolds stress must satisfy
\begin{equation}
	\Div R_{q+1}=w_{q+1}\cdot\nabla v_q+v_q\cdot\nabla w_{q+1}+\Div\left[R_q+(p_{q+1}-p_q)\Id+w_{q+1}\otimes w_{q+1}\right].
	\label{ec nueva R}
\end{equation}
As we will see later, we can expect a smoothing effect when solving for $R_{q+1}$, so highly oscillatory terms are attenuated. Hence, even if a term is large, it will have a small contribution to the Reynolds stress as long as its size is small compared to the frequency at which it oscillates.

Since $w_{q+1}$ oscillates much faster than $v_q$, the first term in (\ref{ec nueva R}) is harmless. However, the second term is potentially problematic. We will need some cancellations to ensure that its size is much smaller than the frequency of oscillation of $w_{q+1}$. The third term in (\ref{ec nueva R}) is quadratic in $w_{q+1}$, so it will contain high-frequency terms (for which we will need some cancellations) and low-frequency terms. The correction $w_{q+1}$ is chosen so that the low-frequency modes of $w_{q+1}\otimes w_{q+1}$ partially cancel $R_q$, up to a multiple of the identity. The key is to use a decomposition of the form
\begin{equation}
	\Id-sR_q(x)=\sum_{j=1}^6 \gamma_j(x)^2\zeta_j\otimes \zeta_j
	\label{descomposición R intro}
\end{equation}
for a sufficiently small number $s>0$. See \cref{geometric lemma}.

In its simplest form, the perturbation $w_{q+1}$ will (locally) look roughly like this:
\begin{equation}
	w_{q+1}(x)\approx b(x)\,\zeta \cos(\lambda \hspace{0.5pt}k\cdot x),
	\label{approx wq+1}
\end{equation}
where $b$ is a slowly-varying amplitude, $\lambda$ is a very large number and $\zeta$ and $k$ are unitary vectors. This field actually needs a small correction to make it divergence-free. Since we want this correction to be fairly small, we need the previous expression to be almost divergence-free already. Therefore, we must impose
\begin{equation}
	k\cdot \zeta=0.
	\label{prod k zeta}
\end{equation}
It can be seen that this also provides the necessary cancellations in the quadratic term in (\ref{ec nueva R}). As one would guess, $\zeta$ is one of the vectors that appear in decomposition (\ref{descomposición R intro}) and the amplitude $b$ is chosen accordingly. This serves to cancel the lower-frequency terms. Meanwhile, to control the second term in (\ref{ec nueva R}), we also require
\begin{equation}
	\label{k almost perpendicular}
	k\cdot \frac{v_q}{\abs{v_q}}\ll 1.
\end{equation}

These are the basic ingredients of a traditional convex integration scheme adapted to the steady Euler equations. However, we want to achieve this perturbation through a diffeomorphism, as in (\ref{pushforward intro}). If we forget about the composition with $\phi_{q+1}^{-1}$ for a moment, to obtain a perturbation like (\ref{approx wq+1}) we could try the following ansatz:
\begin{equation}
	\label{approx phiq+1}
	\phi_{q+1}(x)\approx x+\frac{b(x)}{\lambda v_q(x)\cdot k}\,\zeta \sin(\lambda\hspace{0.5pt}k\cdot x).
\end{equation}
There are many difficulties with this. The most visible problem is probably the denominator. By (\ref{k almost perpendicular}), we require $k$ to be almost perpendicular to $v_q$, but it is clear that it cannot be completely so. We must ensure that
\begin{equation}
	\label{ángulo no demasiado pequeño}
	k\cdot \frac{v_q}{\abs{v_q}}\gg \frac{1}{\lambda}.
\end{equation}
Of course, controlling the angle is not enough, the zeros $\mathcal{Z}$ of $v_q$ are problematic, as well. At each step $q$ we will not perturb the velocity in a neighborhood of its zeros, that is, $\phi_{q+1}=\Id$ in that region. We will progressively perturb closer to $\mathcal{Z}$, but always avoiding it by some distance. Since we are not perturbing close to $\mathcal{Z}$, we are not correcting the Reynolds stress, so it must be very small in that region to begin with. The size of the $w_{q+1}$ and $R_{q+1}$ as we approach the zeros of the field must be handled carefully.

Another major difficulty is that $\phi_{q+1}$ oscillates very quickly. Hence, in the pushforward $v_{q+1}$ we compose two rapidly-oscillating functions, which can be quite difficult to deal with. Our approach is to use building blocks that are as simple as possible, as in (\ref{approx wq+1}) and (\ref{approx phiq+1}). We will see that, in this case, the argument of the trigonometric function is almost unchanged by the composition with $\phi_{q+1}^{-1}$. Since this is the most critical part of $D\phi_{q+1}$, this property greatly simplifies controlling the composition with $\phi_{q+1}^{-1}$. The price that we have to pay is that it is highly nontrivial to obtain the final perturbation that we need using these too simple building blocks. In contrast, this part is usually straightforward in traditional convex integration schemes.

As the reader may have noticed, with a perturbation like (\ref{approx wq+1}) we only cancel one of the terms in (\ref{descomposición R intro}). Therefore, the step to construct $v_{q+1}$ is actually divided into smaller stages in which we successively cancel each of the terms in (\ref{descomposición R intro}). The problem is that the perturbation may be much larger than the field $v_q$ at certain points, as we will see. As a result, even if the vectors $\zeta_j$ that appear in (\ref{descomposición R intro}) are chosen to make certain angle with $v_q$, they may align with the perturbed velocity in the intermediate stages. Thus, the analogues of (\ref{prod k zeta}) and (\ref{ángulo no demasiado pequeño}) would be in contradiction. Solving this, as well as other difficulties of geometrical nature, is the main challenge posed by using such simple building blocks. 

As a final consideration, notice that $D\phi_{q+1}$ is expected to be quite large because of (\ref{k almost perpendicular}). However, it remains controlled in the direction $v_q$. This explains why the pushforward $v_{q+1}$ will converge uniformly as $q\to \infty$, even though $D\phi_{q+1}$ will blow up.

\subsection{Inductive hypotheses} \label{subsect inductive hyp}
In this subsection we will describe in detail the iterative process. The notation of the different norms we use is presented in the Appendix. We begin by constructing the initial subsolution. Given the initial divergence-free field $v_0$, we define
\begin{align}
	p_0&\coloneqq -\abs{v_0}^2, \label{def p0}\\
	R_0&\coloneqq v_0\otimes v_0-\abs{v_0}^2\Id. \label{def R0}
\end{align}
It is clear that $(v_0,p_0,R_0)$ is a subsolution and it satisfies the additional property $R_0v_0=0$. Without loss of generality, after multiplying by a constant we may assume that the $C^1$~norm of these functions is bounded as
\begin{equation}
	\label{assumption tamaño v0}
	\norm{v_0}_1+\norm{R_0}_1\leq 1.
\end{equation}

The frequency of the oscillations will be controlled by a parameter $\lambda_q$ and their size will be controlled by a parameter $\delta_q$, which are given by
\begin{align}
	\lambda_{q}&\coloneqq a^{b^{q}-1}, \label{def lambdaq}\\
	\delta_q&\coloneqq\lambda_q^{-4\alpha}, \label{def deltaq}
\end{align}
where $b:=\frac{3}{2}$ and $a>1$ is a very large parameter that will be chosen later. Meanwhile, $\alpha>0$ is the parameter in the statement of \cref{T.main} and it is assumed to be sufficiently small.  In addition, it will be convenient to set
\[\beta\coloneqq 2^{-11}\]
to make our expressions more compact.

In order to have better control close to the zeros of the field $v_0$, we define the sets
\begin{align}
	\Sigma_{q}&\coloneqq \left\{x\in \TT^3:\;\abs{v_0(x)}=\frac{1}{2}\delta_q^{1/2}\right\}, \label{def Sigmaq}\\ 
	\Omega_{q}&\coloneqq \left\{x\in \TT^3:\;\abs{v_0(x)}>\frac{1}{2}\delta_q^{1/2}\right\}. \label{def Omegaq}
\end{align}
Given $b>1$, the sets $\Sigma_q$ will be smooth surfaces for almost any $a>1$, by Sard's theorem. We will later see that the specific value of $a$ is not relevant in our construction, it just needs to be sufficiently large. Since the countable intersection of full measure sets still has full measure, we may assume that $\Sigma_q$ is a smooth surface for any $q\geq 0$, although this plays no role in our scheme. In contrast, this bound will be very useful:
\begin{equation}
	\label{crecimiento R0}
	\abs{R_0(x)}\leq\frac{1}{4}\delta_q \qquad \forall x\in \TT^3\backslash \Omega_q,
\end{equation}
which follows from (\ref{def R0}) and (\ref{def Omegaq}). Next, we choose $x\in \Sigma_q$ and $y\in \Sigma_{q+1}$ and compute
\[\frac{1}{2}\delta_q^{1/2}-\frac{1}{2}\delta_{q+1}^{1/2}=\abs{v_0(x)}-\abs{v_0(y)}\leq \abs{v_0(x)-v_0(y)}\leq \norm{v_0}_1\abs{x-y}\leq \abs{x-y}.\]
Note that, if $a>1$ is sufficiently large (depending on $\alpha>0$), we have
\begin{equation}
	\label{cociente deltas}
	\frac{\delta_{q+1}^{1/2}}{\delta_q^{1/2}}=a^{-2\alpha b^q(b-1)}\leq a^{-\alpha}\leq \frac{1}{2}.
\end{equation}
Taking the infimum in $x\in \Sigma_q$ and $y\in \Sigma_{q+1}$, we conclude
\begin{equation}
	\label{distance sigmas}
	\dist(\Sigma_q,\Sigma_{q+1})\geq \frac{1}{4}\delta_q^{1/2}.
\end{equation}

Once we have introduced these preliminary definitions, we are ready to state the inductive hypotheses. The volume-preserving diffeomorphism $\Phi_q$ will satisfy the following inductive hypotheses for $q\geq 0$:
\begin{align}
	\Phi_q&= \Id \quad \text{in }\TT^3\backslash\Omega_{q+1}, \label{inductive support perturbation diffeo}\\
	\norm{\Phi_q-\Id}_0+\|\Phi_q^{-1}-\Id\|_0&\leq a^{-\beta}(1-2^{-64q}), \label{diffeo-Id}\\
	\norm{\Phi_q}_1+\|\Phi_q^{-1}\|_1&\leq \delta_{q+1}^{-1280}, \label{derivative Phiq} \\ \norm{\Phi_q}_2+\|\Phi_q^{-1}\|_2&\leq \delta_{q+1}^{-2560}\lambda_q. \label{second derivative Phiq}
\end{align}
These conditions trivially hold for $\Phi_0\coloneqq \Id$. Meanwhile, the subsolution $(v_q,p_q,R_q)$ will satisfy the following hypotheses:
\begin{align}
	(v_q,p_q,R_q)&=(v_0,p_0,R_0) \hspace{20pt} \text{in }\TT^3\backslash\Omega_{q+1},\label{inductive support perturbation}\\
	v_q\circ \Phi_q&=D\Phi_q\,v_0,  \label{inductive pushforward}\\ 
	\abs{v_q(x)}&\geq 2^{-2}\delta_{q+1}^{1/2} \qquad \forall x\in\Omega_{q+1}, \label{inductive growth} \\
	\norm{v_q}_0&\leq 1+2^{13}(1-2^{-q}), \label{uniform bound vq}\\
	\norm{v_{q}}_1&\leq 2^6\delta_{q-1}^{1/2}\lambda_{q}, \label{inductive derivative}\\ 
	\norm{R_{q}}_0&\leq \delta_{q}, \label{inductive Rq C0}\\
	\norm{R_{q}}_1&\leq \delta_{q}\lambda_{q}.  \label{inductive Rq C1}
\end{align}
If we set $\delta_{-1}\coloneqq 1$, these conditions hold for $q=0$, by (\ref{assumption tamaño v0}) and (\ref{def Omegaq}). 

We further assume that for $q>0$ we may decompose $R_q$ as
\begin{equation}
	\label{descomposición Rq}
	R_q=\rho_q R_0+S_q,
\end{equation}
where $\rho_q\in C^\infty(\TT^3,[0,1])$ is a cutoff function that vanishes in a neighborhood of $\overline{\Omega}_q$ and equals 1 in a neighborhood of $\TT^3\backslash\Omega_{q+1}$. We may assume that $\sqrt{\rho_q}$ is smooth and
\begin{equation}
	\norm{\sqrt{\rho_q}}_N\leq C_N \delta_{q}^{-N/2} \qquad \forall N\geq 0,
	\label{cotas rhoq}
\end{equation}
where the constants $C_N$ depend only on $N$. Meanwhile, the support of $S_q$ is contained in $\Omega_{q+1}$ and it satisfies
\begin{equation}
	\label{inductive decomposition error}
	\norm{S_q}_0+\lambda_q^{-1}\norm{S_q}_1\leq \delta_{q+3}.
\end{equation}
For convenience, we set $\rho_0\equiv 1$ and $S_0\equiv 0$, so that (\ref{descomposición Rq}) holds for $q=0$, too. This decomposition and the definition (\ref{def R0}) imply that, for any $q\geq 0$, we have
\begin{equation}
	\norm{R_qv_0}_0\leq \delta_{q+3}.
\end{equation}
As a result, the perturbations will be mostly perpendicular to $v_0$, which will simplify the construction. 

The following proposition (which is proved in Section~\ref{section prueba prop steps}) proves the induction hypotheses:

\begin{proposition}
	\label{prop steps}
	Let $v_0\in C^\infty(\TT^3,\RR^3)$ be a divergence-free vector field satisfying (\ref{assumption tamaño v0}). Let $\Phi_0:=\Id$ and let $p_0,R_0$ be given by (\ref{def p0}) and (\ref{def R0}). If $\alpha>0$ is sufficiently small and $a>1$ is sufficiently large (depending on $\alpha$), the following holds:
	
	Let $q\geq 0$. Suppose that for all $q'\leq q$ there is a volume-preserving diffeomorphism $\Phi_{q'}$ and a subsolution $(v_{q'},p_{q'},R_{q'})$ satisfying (\ref{inductive support perturbation diffeo})-(\ref{inductive decomposition error}). Then, there exist a volume-preserving diffeomorphism $\Phi_{q+1}$ and a subsolution $(v_{q+1},p_{q+1},R_{q+1})$ satisfying (\ref{inductive support perturbation diffeo})-(\ref{inductive decomposition error}) with $q$ replaced by $q+1$. Furthermore, we have
	\begin{align}
		\norm{v_{q+1}-v_q}_{H^{-1}}&\leq 2^6\delta_{q+1}^{3000}, \label{cota H-1 q} \\ 
		\norm{v_{q+1}-v_q}_0+\lambda_{q+1}^{-1}\norm{v_{q+1}-v_q}_1&\leq 2^{12}\delta_{q}^{1/2}, \label{conclusión 1}\\
		\norm{\Phi_{q+1}-\Phi_q}_0+\norm{\Phi_{q+1}^{-1}-\Phi_q^{-1}}_0&\leq 2^6\delta_{q+1}^{-10}\lambda_q^{-1}, \label{conclusión 2} \\
		\norm{\Phi_{q+1}-\Phi_q}_1+\norm{\Phi_{q+1}^{-1}-\Phi_q^{-1}}_1&\leq 2^6\delta_{q+1}^{-1280}. \label{conclusión 3}
	\end{align}
\end{proposition}

Each of these steps will be achieved in a series of stages. To simplify the notation, we introduce the following convention:

\begin{definition}
	\label{indices convention}
	We define the space of indices
	\[\mathcal{J}\coloneqq \NN\times \{0, \dots, 7\}\times \{0, \dots, 7\},\]
	where we consider $0\in\NN$. We equip $\mathcal{J}$ with the lexicographic order and we denote by $J+1$ the element that follows $J\in\mathcal{J}$ in this order. Given $J=(q,j,l)\in\mathcal{J}$, we define
	\begin{align}
		q^\ast(J)&\coloneqq \begin{cases}
			q-1 \qquad \text{if }j=l=0, \\ q \hspace{37pt} \text{otherwise,} 	\end{cases} \label{def q*}\\
		\abs{J}&\coloneqq q+\frac{1}{8}j+\frac{1}{64}l,
		\label{def abs J} \\ s(J)&\coloneqq 1280+80j+10l, \label{def sJ} \\
		M(J)&\coloneqq 1+2^{13}(1-2^{-q})+(8j+l)2^{6-q}. \label{def M(J)}
	\end{align}
	We also denote by $q(J), j(J), l(J)$ the projection onto the first, second and third coordinates, respectively.
\end{definition}

To prove \cref{prop steps}, we will go through a series of intermediate stages, starting with $(q,0,0)$ and reaching $(q+1,0,0)$. In each stage $J\in\mathcal{J}$, the perturbation will oscillate with frequency
\begin{equation}
	\lambda_{J}\coloneqq a^{b^{\abs{J}}-1}. \label{def lambdaJ}
\end{equation}
The volume-preserving diffeomorphism $\Phi_{J}$ will satisfy the following inductive hypotheses:
\begin{align}
	\Phi_J&= \Id \quad \text{in }\TT^3\backslash\Omega_{q(J)+2}, \label{support perturbation diffeo J}\\
	\norm{\Phi_{J}-\Id}_0+\|\Phi_{J}^{-1}-\Id\|_0&\leq a^{-\beta}\left(1-2^{-64\abs{J}}\right),  \label{inductive Phiqj C0}\\
	\norm{\Phi_{J}}_1+\|\Phi_{J}^{-1}\|_1&\leq \delta_{q(J)+1}^{-s(J)}, \label{inductive Phiqj C1} \\
	\norm{\Phi_{J}}_2+\|\Phi_{J}^{-1}\|_2&\leq \delta_{q(J)+1}^{-2s(J)}\lambda_{J+1}. \label{inductive Phiqj C2}
\end{align}
Meanwhile, the subsolution $(v_{J},p_{J},R_{J})$ will satisfy the following hypotheses:
\begin{align}
	\label{inductive J igual fuera de omegaq+2}
	(v_{J},p_{J},R_{J})&=(v_0,p_0,R_0) \quad  \text{in }\TT^3\backslash\Omega_{q^\ast(J)+2}, \\
	v_{J}\circ \Phi_{J}&=D\Phi_{J}\,v_0, \label{inductive j pushforward} \\
	\abs{v_J(x)}&\geq 2^{-2}\delta_{q^\ast(J)+2}^{1/2} \qquad \forall x\in \Omega_{q^\ast(J)+2},
	\label{crecimiento vJ}	\\
	\norm{v_J}_0&\leq M(J), \label{inductive j cota C0}\\
	\norm{v_{J}}_1&\leq 2^6\delta_{q^\ast(J)}^{1/2}\lambda_{J}. \label{derivative vqj}
\end{align}
Furthermore, in these intermediate stages the perturbation will be almost perpendicular to $v_{(q(J),0,0)}\equiv v_{q(J)}$:
\begin{equation}
	\norm{[v_J-v_{q(J)}]\cdot \frac{v_{q(J)}}{\abs{v_{q(J)}}}}_0\leq 16[8j(J)+l(J)]\delta_{q(J)+3}^{1/2}.
	\label{perturbation almost perpendicular}
\end{equation}

Each intermediate stage is achieved combining a finite number of simpler building blocks, which are constructed in the following auxiliary proposition (whose proof is contained in Sections~\ref{S.construction}-\ref{section estimates R}): 

\begin{proposition}
	\label{prop stages}  
	Let $v_0\in C^\infty(\TT^3,\RR^3)$ be a divergence-free vector field satisfying (\ref{assumption tamaño v0}). Let $\Phi_0=\Id$ and let $p_0,R_0$ be given by (\ref{def p0}) and (\ref{def R0}). If $\alpha>0$ is sufficiently small and $a>1$ is sufficiently large (depending on $\alpha$), the following holds:
	
	Let $J\in \mathcal{J}$. Suppose that for all $J'\leq J$ there exists a volume-preserving diffeomorphism $\Phi_{J'}$ and a subsolution $(v_{J'},p_{J'},R_{J'})$ satisfying (\ref{support perturbation diffeo J})-(\ref{perturbation almost perpendicular}). Let 
	\begin{equation}
		\label{stages tamaño ell}
		1\gg\ell\geq \frac{\delta_{q(J)+4}}{2^4\delta_{q^\ast(J)}^{1/2}\lambda_J}.
	\end{equation}
	Let $Q$ be a cube of side $3\ell/2$ and let $U$ be a convex open set with smooth boundary and such that $Q\subset U\subset \Omega_{q(J)+2}$. In addition, suppose that
	\begin{equation}
		\dist(Q,\partial U)\geq \ell/8.
		\label{distance Q,U}
	\end{equation}
	Let $\zeta\in \RR^3$ be a unitary vector and let $\gamma\in C^\infty_c(Q)$ such that
	\begin{align}
		\norm{\gamma}_0&\leq \delta_{q(J)}^{1/2}, \label{cota gamma stages C0}\\
		\norm{\gamma}_N&\leq C_N\delta_{q(J)}^{1/2}\ell^{-N} \qquad \forall N\geq 0, \label{cota gamma stages CN} \\ \label{perpendicularidad o tamaño} \abs{\gamma(x)\,\zeta\cdot\frac{v_{q(J)}(x)}{\abs{v_{q(J)}(x)}}}&\leq \delta_{q(J)+3}^{1/2} \qquad \forall x\in U,\\ \label{angulo vJ enunciado prop}
		\abs{v_J(x)\times \zeta}&\geq 2^{-3}\delta_{q(J)+2}^{1/2} \qquad \forall x\in U
	\end{align}
	for some universal constants $C_N$. Then, there exist a volume-preserving diffeomorphism $\Phi_{J+1}$ and a subsolution $(v_{J+1},p_{J+1},R_{J+1})$ satisfying (\ref{support perturbation diffeo J})-(\ref{perturbation almost perpendicular}) and such that
	\begin{equation}
		\label{resultado igual fuera de U}
		(\Phi_{J+1}, v_{J+1},p_{J+1},R_{J+1})=(\Phi_{J}, v_{J},p_{J},R_{J})
	\end{equation}
	in a neighborhood of $\TT^3\backslash U$. Furthermore, we have
	\begin{align} 
		\norm{v_{J+1}-v_{J}}_0+\lambda_{J+1}^{-1}\norm{v_{J+1}-v_{J}}_1&\leq 2^{6}\delta_{q(J)}^{1/2},
		\label{cambio iteración j} \\
		\norm{\Phi_{J+1}-\Phi_J}_0+\norm{\Phi_{J+1}^{-1}-\Phi_J^{-1}}_0&\leq \delta_{q(J)+1}^{-10}\lambda_{J+1}^{-1}, \label{conclusión stages 2} \\
		\norm{\Phi_{J+1}-\Phi_J}_1+\norm{\Phi_{J+1}^{-1}-\Phi_J^{-1}}_1&\leq \delta_{q(J)+1}^{-s(J+1)}\label{conclusión stages 3}
	\end{align}
	and the error $E_{J+1}\coloneqq R_{J+1}-(R_{J}+\gamma^2\zeta\otimes\zeta)$ satisfies
	\begin{equation}
		\norm{E_{J+1}}_0+\lambda_{J+1}^{-1}\norm{E_{J+1}}_1\leq 2^{-8}\delta_{q(J)+4}.
		\label{error J+1}
	\end{equation}
	In addition, there exist $A\in C^\infty(\TT^3,\RR^3)$ and $B\in C^\infty(\TT^3,\RR^{3\times 3})$ with
	\begin{equation}
		\norm{A}_0+\lambda_{J+1}\norm{B}_0\leq \delta_{q(J)+1}^{3000}.
		\label{cotas A y B}
	\end{equation}
	and such that
	\begin{equation}
		\int_{\TT^3}f\cdot(v_{J+1}-v_J)=\int_{\TT^3}(f\cdot A+Df:B)
		\label{integral norma H-1}
	\end{equation}
	for any test-function $f\in C^\infty(\TT^3,\RR^3)$. In particular,
	\begin{equation}
		\norm{v_{J+1}-v_{J}}_{H^{-1}}\leq \delta_{q(J)+1}^{3000}. \label{cota H-1 J}
	\end{equation}
\end{proposition}

\begin{remark}
	\label{remark igual si varias disjuntas}
	To avoid introducing even more indices, we have presented this result for a single cube $Q$, an open set $U$, a coefficient $\gamma$ and a vector $\zeta$. However, the same result holds if we have a collection $\{Q_m,U_m,\gamma_m,\zeta_m\}$, provided that the open sets $U_m$ are pairwise disjoint. Indeed, by (\ref{resultado igual fuera de U}) we can simply glue the tuples $(\Phi_{J+1}, v_{J+1},p_{J+1},R_{J+1})_m$ to obtain a new tuple $(\Phi_{J+1}, v_{J+1},p_{J+1},R_{J+1})$ satisfying the inductive hypotheses (\ref{support perturbation diffeo J})-(\ref{inductive J igual fuera de omegaq+2}) and (\ref{cambio iteración j})-(\ref{error J+1}).
\end{remark}

\section{Construction of the diffeomorphism and the new subsolution}
\label{S.construction}
We begin the proof of \cref{prop stages} by defining the diffeomorphism $\Phi_{J+1}$ and the new subsolution $(v_{J+1},p_{J+1},R_{J+1})$. The new diffeomorphism $\Phi_{J+1}$ will be of the form $\Phi_{J+1}=\phi_{J+1}\circ\Phi_J$ for certain diffeomorphism $\phi_{J+1}$ that we will construct. The new velocity $v_{J+1}$ will be the pushforward of $v_J$ by $\phi_{J+1}$. In this section we will identify the main term $w_0$ of the perturbation $w_{J+1}\coloneqq v_{J+1}-v_J$ and in \cref{section estimates v} we will show that it is, indeed, the most important term. At the end of this section we will explain how to construct $p_{J+1}$ and a new Reynolds stress $R_{J+1}$ to obtain a subsolution $(v_{J+1},p_{J+1},R_{J+1})$.

Since we will be working with a fixed $J\in \mathcal{J}$, we will write $q$ and $q^\ast$ instead of $q(J)$ and $q^\ast(J)$ to simplify the notation. Keep this in mind when comparing the expressions with the ones in \cref{prop stages}.


We will use $A\lesssim B$ to denote $A\leq C\hspace{0.5pt}B$ for some constant $C>0$ that is independent of the iteration index $J$ or the parameter $a$. However, the implicit constants are allowed to depend on $\alpha$. In addition, when estimating a $C^N$-norm, they may also depend on $N$. This is not an issue because we will only use a finite number of derivatives.

The goal of this notation is to focus on the dependence on the iteration parameters $\lambda_J$ and $\delta_q$. The implicit constants will be irrelevant in most inequalities because we will have an extra factor that can be made arbitrarily small by increasing $a$. Therefore, once we have chosen a suitable value of $\alpha$, choosing $a$ sufficiently large will compensate the implicit constants, provided that they do not depend on $a$ and that there is a finite number of them. This is ensured by excluding dependence on $J$ and by using only a finite number of derivatives.

\subsection{Definition of the diffeomorphism}
The diffeomorphism $\phi_{J+1}$ will be of the form
\[\phi_{J+1} \coloneqq  \phi_c\circ \phi_0,\]
where $\phi_0$ is the main correction term, which will yield the main perturbation $w_0$, while $\phi_c$ is very close to the identity and its only role is to ensure that $\phi_{J+1}$ is volume-preserving.

We will not define the diffeomorphism $\phi_0$ directly. Instead, it will be convenient to define an auxiliary diffeomorphism $\psi$ and then let $\phi_0\coloneqq \psi^{-1}$. The reason is that having an explicit formula for $\phi_0^{-1}$ will simplify some estimates, especially when constructing $\phi_c$. Nevertheless, both $\phi_0$ and $\phi_0^{-1}$ appear in the expression of the pushforward, so we will need certain control on both of them.

The corrections $\phi_0$ and $w_0$ will oscillate with frequency $\lambda_{J+1}$ in a direction that must be perpendicular to $\zeta$ and $v_J$ to first order. The projection of this direction onto $v_J$ must be small but nonzero; we will control it with a parameter $\eta\ll\lambda_{J+1}$. In addition, we introduce a parameter $\mu^{-1}$ that will control the size of certain cutoffs that we will use. The hierarchy of parameters is as follows:
\begin{equation}
	\mu\ll \eta \ll \lambda_{J+1}.
	\label{relación tamaño entre parámetros}
\end{equation}
They are given by
\begin{align}
	\eta&\coloneqq \delta_{q+4}\lambda_{J+1}^{1-2\alpha}, \label{def eta}\\
	\mu &\coloneqq \lambda_{J+1}^{1-4\beta}, \label{def mu}
\end{align}
where the small parameter $\beta$ was defined as $\beta=2^{-11}$. The reason behind these definitions will become clear as we progress in the proof, but the motivation is the following: we want the ratio $\lambda_{J+1}\eta^{-1}$ to be as small as possible so that $D\Phi_{q+1}$ is bounded by some power of $\delta_{q+1}$. On the other hand, we want $\lambda_{J+1}\mu^{-1}$ to be greater than some power of $\lambda_{J+1}$ (independent of $\alpha$) so that (\ref{cotas A y B}) and (\ref{integral norma H-1}) hold. This will be essential to prove uniqueness of the flow in a suitable regularity class. 

In view of these definitions, it is not evident that $\mu\ll \eta$. We will check this in \cref{lema técnico relaciones parámetros}. We will also see that we may assume that
\begin{equation}
	\ell^{-1}+\norm{v_J}_1\leq \mu,
	\label{relación mu ell}
\end{equation}
where $\ell$ is the parameter in the statement of \cref{prop stages}.

Once we have defined the various parameters that we will need, we fix a smooth cutoff function $\chi\in C^\infty_c\left(\left(-\frac{3}{4},\frac{3}{4}\right)\right)$ such that 
\[\sum_{m\in \ZZ^3}\chi(x-m)^2=1.\]
For $m\in \ZZ^3$ we define 
\begin{equation}
	\chi_m(x)\coloneqq \chi(\mu x-m).
	\label{def chim}
\end{equation}
Note that at most $8$ of the cutoffs are nonzero at any given point. Hence, we can fix coefficients 
\[l_m\in \left\{1,2^{-1}, \dots, 2^{-7}\right\}\]
so that for any two distinct $\chi_{m_1}$ and $\chi_{m_2}$ whose supports are not disjoint, we have $l_{m_1}\neq l_{m_2}$. This will later help us avoid unwanted destructive interference. Taking into account that the size of the support of these cutoffs is comparable to $\mu^{-1}\ll \ell$, it follows from (\ref{distance Q,U}) that we can find a subset $\Lambda\subset \ZZ^3$ such that
\[\begin{cases}
	\sum_{m\in\Lambda}\chi_m^2=1 \qquad\text{on }Q, \\ \supp\chi_m\subset U \hspace{26pt} \forall m\in \Lambda.
\end{cases}\]
We choose $\Lambda$ as the minimal set satisfying this property (which clearly exists). Since the support of $\gamma$ is contained in $Q$, we only need to use the cubes in $\Lambda$ and we can work in $U\subset \Omega_{q+2}$. 


Next, for each $m\in \Lambda$ we define an orthogonal reference frame
\begin{align*}
	u_{J,m}&\coloneqq v_J(\mu^{-1}m), \\
	k_{J,m}&\coloneqq \frac{u_{J,m}\times\zeta}{\abs{u_{J,m}\times\zeta}},\\
	\xi_{J,m}&\coloneqq \frac{\zeta\times k_{J,m}}{\abs{\zeta\times k_{J,m}}},
\end{align*}
which are well-defined because of (\ref{angulo vJ enunciado prop}). With all these ingredients, we can define the auxiliary diffeomorphism
\begin{equation}
	\psi(x)\coloneqq  x-\sum_{m\in\Lambda} \frac{a_m(x)}{l_m\eta}\zeta\sin[\theta_m(x)],
\end{equation}
where
\begin{align}
	\theta_m(x)&\coloneqq l_m(\eta\,\xi_{J,m}\cdot x+\lambda_{J+1}\,k_{J,m}\cdot x), \label{def thetam} \\[1pt]
	a_m(x)&\coloneqq \frac{\sqrt{2}\;\chi_m(x)}{u_{J,m}\cdot \xi_{J,m}} \gamma(x). \label{def am}
\end{align}
To ensure that this is well-defined, we compute the denominator.
Fix $m\in\Lambda$. Since $u_{J,m}, \zeta$ and $\xi_{J,m}$ are contained in the plane perpendicular to $k_{J,m}$, we have
\begin{equation}
	\label{tamaño denominador}
	\abs{u_{J,m}\cdot \xi_{J,m}}=\abs{u_{J,m}\times\zeta} \overset{(\ref{angulo vJ enunciado prop})}{\geq} 2^{-3}\delta_{q+2}^{1/2}.
\end{equation}
In particular, the denominator in (\ref{def am}) is nonzero, so $\psi$ is well-defined. Since the support of $\gamma$ is contained in $Q$, we see that $\psi=\Id$ in a neighborhood of $\TT^3\backslash U$.

\begin{remark}
	\label{fase invariante}
	For any $m\in \Lambda$, the phase $\theta_m$ satisfies $\theta_m\circ \psi=\theta_m$ because $\zeta$ is orthogonal to $\xi_{J,m}$ and $k_{J,m}$. This will greatly simplify the construction because we avoid the difficulties of composing two rapidly oscillating functions.
\end{remark} 

We will see that the Jacobian of $\psi$ is very close to 1, which is quite remarkable, since $D\psi$ is very large. Again, orthogonality is at play here. In particular, since $\det(D\psi)$ does not vanish, we deduce that $\psi$ is a local diffeomorphism. To prove that it is a global diffeomorphism, as we claim, we will apply Banach's fixed point theorem to the map
\[T:C^0(\TT^3,\TT^3)\to C^0(\TT^3,\TT^3), \quad f\mapsto \Id+\sum_{m\in\Lambda}\frac{a_m\circ f}{l_m\eta}\zeta\sin(\theta_m).\]
The fixed point $\varphi\in C^0(\TT^3,\TT^3)$ will turn out to be the inverse of $\psi$, which makes it automatically smooth. Once we known that $\psi$ is invertible, we define
\[\phi_0\coloneqq \psi^{-1}.\]
It is not possible to obtain an explicit formula for $\phi_0$, but the identity $T\phi_0=\phi_0$ will be enough to derive all the estimates that we will need. Note that $\phi_0=\Id$ in a neighborhood of $\TT^3\backslash U$ because so does $\psi$. 

We will see that $\phi_0$ is not volume-preserving, so we need to introduce a correction $\phi_c$ in order to obtain the final diffeomorphism $\phi_{J+1}$. Let 
\begin{equation}
	\label{def fc}
	f_c\coloneqq \det(D\psi)-1.
\end{equation}
Since $\psi=\Id$ in a neighborhood of $\TT^3\backslash U$, we see that the support of $f_c$ is contained in $U$. It is easy to check that
\[\det(D \phi_c)=1+f_c \quad \Rightarrow \quad \det[D(\phi_c\circ\phi_0)]=1.\]
Note that if we tried to express $f_c$ in terms of $\phi_0$ instead of $\psi$, we would obtain a more complicated expression. This is another advantage of having an explict formula for $\phi_0^{-1}$ instead of $\phi_0$. 

To construct a diffeomorphism $\phi_c$ such that $\det(D \phi_c)=1+f_c$, we will use the following lemma, which is just a simplified version of the results in \cite{DM}: 
\begin{proposition}
	\label{lema dacorogna}
	Let $U\subset \RR^3$ be a bounded domain with smooth boundary and let $f\in C^\infty_c(U,\RR^3)$ such that
	$\int_Uf=0$. In addition, let us assume that
	\begin{align*}
		&\norm{f}_0\leq 1/2, \\
		&\norm{f}_\beta+\norm{f}_{B^{-1+\beta}_{\infty,\infty}}\norm{f}_1 \leq 1.
	\end{align*}
	Then, there exists a diffeomorphism $\phi\in C^\infty(\RR^3,\RR^3)$ such that $\phi\equiv\Id$ in a neighborhood of $\RR^3\backslash U$ and such that \[\det(D \phi)=1+f.\] Furthermore, $\phi$ may be chosen so that 
	\begin{align}
		\norm{\phi-\Id}_0&\lesssim \norm{f}_{B^{-1+\beta}_{\infty,\infty}}, \label{estimate phi dacorogna C0}\\
		\norm{\phi-\Id}_1&\lesssim \norm{f}_\beta+\norm{f}_{B^{-1+\beta}_{\infty,\infty}}\norm{f}_1,  \label{estimate phi dacorogna C1} \\
		\norm{\phi}_2&\lesssim \norm{f}_{1+\beta}+\norm{f}_\beta\norm{f}_1+\norm{f}_{B^{-1+\beta}_{\infty,\infty}}\norm{f}_2+\norm{f}_{B^{-1+\beta}_{\infty,\infty}}\norm{f}_1^2, \label{estimate phi dacorogna C2}
	\end{align}
	where the implicit constants depend on $\beta$ and on $U$.
\end{proposition}
While the construction in the previous lemma is standard, we also need to derive the appropriate bounds for the diffeomorphism. We delay the proof until \cref{section estimates diffeo}. Note that the previous lemma is stated in $\RR^3$, but we can simply work locally and forget about periodicity because our set $U$ covers a very small region of~$\TT^3$. 

The final diffeomorphism will then be $\phi_{J+1}=\phi_c\circ\phi_0$. The correction $\phi_c$ is essential to ensure that $v_{J+1}$ is divergence-free, but we will see that its effects are negligible regarding the Reynolds stress. Therefore, the most significant part of the perturbation to the velocity is determined by $\phi_0$, which we constructed to yield the necessary correction $w_0$.

The diffeomorphism linking to the original state is
\begin{equation}
	\Phi_{J+1}\coloneqq \phi_{J+1}\circ\Phi_{J}.
	\label{def Phiqj}
\end{equation}
Since $\phi_0=\phi_c=\Id$ in a neighborhood of $\TT^3\backslash U$, we see that $\Phi_{J+1}=\Phi_J$ in a neighborhood of $\TT^3\backslash U$. In particular, the inductive hypothesis (\ref{support perturbation diffeo J}) then implies that $\Phi_{J+1}=\Id$ in $\TT^3\backslash \Omega_{q+2}$.

\subsection{The new velocity field}
We define the new velocity field as the pushforward of $v_{J}$ by the diffeomorphism $\phi_{J+1}$, that is,
\begin{equation}
	v_{J+1}\coloneqq (D\phi_{J+1}\,v_{J})\circ \phi_{J+1}^{-1}.
	\label{def vqj}
\end{equation}
Note that $v_{J+1}$ is divergence-free because it is the pushforward of a divergence-free field by a volume preserving diffeomorphism. Since $\phi_{J+1}=\Id$ in a neighborhood of $\TT^3\backslash U$, we see that $v_{J+1}=v_J$ in a neighborhood of $\TT^3\backslash U$. In particular, the inductive hypothesis (\ref{inductive J igual fuera de omegaq+2}) then implies that $v_{J+1}=v_0$ in $\TT^3\backslash \Omega_{q(J)+2}$. In addition, it follows from the inductive hypothesis (\ref{inductive j pushforward}) for $v_{J}$ and the definition (\ref{def Phiqj}) that $v_{J+1}$ is the pushforward of $v_0$ by the diffeomorphism $\Phi_{J+1}$, so (\ref{inductive j pushforward}) also holds for~$J+1$.

It will be convenient to decompose the perturbation $w_{J+1}\coloneqq v_{J+1}-v_J$ as a sum $w_{J+1}=w_0+w_c$, where $w_0$ is the main correction term, while $w_c$ is much smaller. We define
\begin{equation}
	w_0(x)\coloneqq \sum_{m\in \Lambda}b_m(x)\,\zeta\cos[\theta_m(x)],
	\label{def w0}
\end{equation}
where
\begin{equation}
	b_m(x)=\sqrt{2}\chi_m(x)\gamma(x)
	\label{def bm}
\end{equation}
and $\theta_m$ was defined in (\ref{def thetam}). The remaining term is just
\begin{equation}
	w_c\coloneqq v_{J+1}-v_{J}-w_0,
\end{equation}
which will be shown to be much smaller than $w_0$. In fact, $w_c$ will turn out to be negligible regarding the Reynolds stress, as we will see.

\subsection{Definition of the new Reynolds stress}
Since we want to construct a new subsolution, we still have to define a new pressure $p_{J+1}$ and a new Reynolds stress $R_{J+1}$. We set $p_{J+1}\coloneqq p_{J}$, so we need to find a symmetric matrix~$R_{J+1}$ such that
\begin{align}
	\begin{split}
		\label{eq new Reynolds}
		\Div R_{J+1}&=\Div(v_{J+1}\otimes v_{J+1}+p_{J+1}\Id) \\
		&=\Div(R_J+w_{J+1}\otimes v_J+v_J\otimes w_{J+1}+w_{J+1}\otimes w_{J+1}).
	\end{split}
\end{align}

Let us manipulate this expression a bit. Using a well-known trigonometric formula and the fact that $\sum b_m^2=2\gamma^2$, we see that we may write
\begin{align*}
	w_0\otimes w_0=&\gamma^2\zeta\otimes \zeta+ \sum_{m\in \Lambda}\frac{1}{2}b_m^2\zeta\otimes \zeta\cos(2\theta_m)\\&+ \sum_{m\neq m'}\frac{1}{2}b_mb_{m'}\zeta\otimes \zeta[\cos(\theta_m-\theta_{m'})+\cos(\theta_m+\theta_{m'})].
\end{align*}
Let
\begin{align}
	\begin{split}
		\rho_1\coloneqq&\Div\left(\sum_{m\neq m'}\frac{1}{2}b_mb_{m'}\zeta\otimes \zeta[\cos(\theta_m-\theta_{m'})+\cos(\theta_m+\theta_{m'})]\right) \\ &+\Div\left(\sum_{m\in \Lambda}\frac{1}{2}b_m^2\zeta\otimes \zeta\cos(2\theta_m)\right).
		\label{def rho1}
	\end{split}
\end{align}

In \cref{section estimates R} we will construct a symmetric matrix $M_1$ such that $\Div M_1=\rho_1$. Hence, 
\[\Div(w_0\otimes w_0)=\Div(\gamma^2\zeta\otimes \zeta+M_1).\]
In \cref{section estimates R} we will also construct a symmetric matrix $M_2$ such that $\Div M_2=\rho_2$, where
\begin{equation}
	\label{def rho2}
	\rho_2=\Div\left(\sum_{m\in \Lambda}b_m[\zeta\otimes u_{J,m}+u_{J,m}\otimes \zeta]\cos(\theta_m)\right).
\end{equation}
Defining
\begin{equation}
	\label{def M3}
	M_3\coloneqq \sum_{m\in \Lambda}b_m[\zeta\otimes (v_J-u_{J,m})+(v_J-u_{J,m})\otimes \zeta]\cos(\theta_m),
\end{equation}
we see that
\[\Div(v_J\otimes w_0+w_0\otimes v_J)=\Div(M_2+M_3).\]

Therefore, if we define
\begin{equation}
	E_{J+1}\coloneqq w_c\otimes v_{J+1}+v_{J+1}\otimes w_c-w_c\otimes w_c+M_1+M_2+M_3,
	\label{def EJ+1}
\end{equation}
it is easy to check that the smooth symmetric matrix
\begin{equation}
	R_{J+1}\coloneqq R_J+\gamma^2\zeta\otimes\zeta+E_{J+1}
	\label{def RJ+1}
\end{equation}
satisfies (\ref{eq new Reynolds}). If the support of the matrices $M_1$ and $M_2$ is contained in $U$, we see that (\ref{resultado igual fuera de U}) will hold, because the support of $\gamma$ is contained in $U$, as well. Combining this with (\ref{inductive J igual fuera de omegaq+2}), we see that (\ref{inductive J igual fuera de omegaq+2}) will hold for $J+1$, too. 

In conclusion, we must construct symmetric matrices $M_1, M_2$ whose divergence equals $\rho_1, \rho_2$ and whose support is contained in $U$. Then, our definition of $R_{J+1}$ will ensure that $(v_{J+1},p_{J+1},R_{J+1})$ is a subsolution satisfying (\ref{inductive J igual fuera de omegaq+2}) and (\ref{resultado igual fuera de U}). After constructing $E_{J+1}$, it is crucial to guarantee that it is sufficiently small, namely that (\ref{error J+1}) holds. All this will be addressed in \cref{section estimates R}.

\section{Estimates on the diffeomorphism} \label{section estimates diffeo}
We begin by studying the relationship between the different parameters. This will allow us to simplify many expressions later on.
\begin{lemma}
	\label{lema técnico relaciones parámetros}
	If $\alpha>0$ is sufficiently small, we have
	\begin{align}
		\lambda_{J+1}^{-\beta}\leq  \delta_{q+1}^{3000}&\leq \delta_{q+4}\lambda_{J+1}^{-\alpha}, \label{aux lambda negativa}\\
		\lambda_J&\leq \delta_{q+4}\mu\lambda_{J+1}^{-\beta}, \label{aux rel parámetros} \\
		\label{relación mu eta}
		\left(\frac{\delta_{q}}{\delta_{q+2}}\right)^{1/2}\frac{\mu}{\eta}&\leq \lambda_{J+1}^{-3\beta}, \\ 
		\label{cociente de parámetros}
		\left(\frac{\delta_{q}\delta_{q^\ast}}{\delta_{q+2}}\right)^{1/2}\frac{\lambda_{J}\lambda_{J+1}}{\mu\eta}&\leq \lambda_{J+1}^{-2\beta}. 
	\end{align}
	In particular, if $a>1$ is sufficiently large, we have
	\begin{align}
		\lambda_J&\leq \delta_{q}^{1/2}\lambda_{J+1}^{1-2\beta}, \label{otra relación entre las lambdas} \\
		\norm{v_J}_1+\ell^{-1}&\leq \mu. \label{comparación ell mu}
	\end{align}
\end{lemma}
\begin{proof}
	Taking into account that $\abs{J+1}\geq 2^{-6}$, we have
	\[\lambda_{J+1}=a^{\abs{J+1}-1}\geq a^{b^{1/64}-1}\geq a^{\ln(b)/64} \qquad \Rightarrow \qquad a\leq \lambda_{J+1}^{64/\ln(b)}.\]
	Therefore, for any $n\geq 0$ we have
	\begin{align}
		\begin{split}
			\delta_{q+n}^{-1}&=a^{4\alpha b^n(b^q-1)+4\alpha(b^n-1)}\leq a^{4\alpha b^n(b^{\abs{J+1}}-1)+4\alpha(b^n-1)}\\ &\leq \lambda_{J+1}^{4\alpha b^n+2^8\alpha (b^n-1)/\ln(b)}\equiv \lambda_{J+1}^{c_n\alpha}.
			\label{cota delta en función de lambdaJ+1}
		\end{split}
	\end{align}
	On the other hand, since $\abs{J+1}=\abs{J}+2^{-6}$, it follows from (\ref{def lambdaJ}) that
	\[\frac{\lambda_{J+1}}{\lambda_J}=a^{b^{\abs{J+1}}-b^{\abs{J}}}\geq a^{2^{-6}\ln(b)b^{\abs{J}}}>\lambda_J^{2^{-6}\ln(b)}.\]
	Taking into account that $b=\frac{3}{2}$ and $\beta=2^{-10}$,  numerical evaluation shows
	\[
	[1+2^{-6}\ln(b)]^{-1}>1-7\beta\,.
	\]
	Hence,
	\[\lambda_J\leq \lambda_{J+1}^{[1+2^{-6}\ln(b)]^{-1}}\leq \lambda_{J+1}^{1-7\beta}. \]
	
	Using these estimates along with the definition of the various parameters, we may write
	\begin{align}
		\frac{\lambda_{J+1}^{-\beta}}{\delta_{q+1}^{3000}}&\leq \lambda_{J+1}^{-\beta+3000c_1\alpha}, \label{aux param 1} \\
		\frac{\lambda_J}{\delta_{q+4}(\mu\lambda_{J+1}^{-\beta})}&\leq \lambda_{J+1}^{-2\beta+c_4\alpha}, \label{aux param 2} \\
		\frac{\delta_q^{1/2}\mu}{\delta_{q+2}^{1/2}\eta\lambda_{J+1}^{-3\beta}}&\leq \lambda_{J+1}^{-\beta+\alpha(2+\frac{1}{2}c_2+c_4)}, \label{aux param 3} \\
		\left(\frac{\delta_{q}\delta_{q^\ast}}{\delta_{q+2}}\right)^{1/2}\frac{\lambda_{J}\lambda_{J+1}}{\mu\eta\lambda_{J+1}^{-2\beta}}&\leq \lambda_{J+1}^{-\beta+\alpha(1+\frac{1}{2}c_2+c_4)}. \label{aux param 4}
	\end{align}
	If $\alpha>0$ is sufficiently small, the exponent on the right-hand side of the inequalities will be negative. Hence, (\ref{aux param 1})-(\ref{aux param 4}) imply (\ref{aux lambda negativa})-(\ref{cociente de parámetros}). Regarding the second inequality in (\ref{aux lambda negativa}), one can check that
	\[\log_a\left(\lambda_{q+1}^9\lambda_{q+4}^{-1}\right)=9(b^{q+1}-1)-(b^{q+4}-1)>0\]
	for all $q\geq 0$, so $\lambda_{q+1}^9\geq \lambda_{q+4}$. The second inequality in (\ref{aux lambda negativa}) then follows easily.
	
	On the other hand, it is clear that (\ref{otra relación entre las lambdas}) follows from (\ref{aux rel parámetros}). Furthermore, (\ref{comparación ell mu}) also follows from (\ref{aux param 3}) for sufficiently large $a>1$ due to (\ref{derivative vqj}) and (\ref{stages tamaño ell}).
\end{proof}
From now on, we will assume that $\widetilde{\alpha}_0\leq \widetilde{\alpha}_1$, so that $\alpha\leq \widetilde{\alpha}_0\leq \widetilde{\alpha}_1$ and the previous lemma holds. Let us move on to estimating the diffeomorphism. The first step is controlling the coefficients $a_m$:
\begin{lemma}
	For any $m\in \Lambda$ and any $N\geq 0$ we have
	\begin{equation}
		\label{tamaño am CN}
		\norm{a_m}_N\lesssim \left(\frac{\delta_{q}}{\delta_{q+2}}\right)^{1/2}\mu^N.
	\end{equation}
\end{lemma}
\begin{proof}
	The $C^0$ bound is a direct consequence of (\ref{cota gamma stages C0}) and (\ref{tamaño denominador}). To derive the bounds for $N>0$, we use (\ref{cota gamma stages CN}), (\ref{def chim}) and also the fact that $\mu \geq \ell^{-1}$, which means that the derivatives of the cutoff $\chi_m$ dominate.
\end{proof}
These bounds on the coefficients $a_m$ directly translate into bounds for the diffeomorphism:
\begin{lemma}
	For any $N\geq1$, the map $\psi$ satisfies
	\begin{equation}
		\norm{\psi-\Id}_0 +\lambda_{J+1}^{-N}\norm{\psi}_N \lesssim \left(\frac{\delta_{q}}{\delta_{q+2}}\right)^{1/2}\frac{1}{\eta}. \label{estimaciones psi}
	\end{equation}
\end{lemma}
\begin{proof}
	The bound for $\norm{\psi-\Id}_0$ is a direct consequence of (\ref{tamaño am CN}) with $N=0$. Concerning the first derivative, we compute
	\begin{equation} 
		\label{nabla psi}
		D\psi=\Id\, -\;\zeta\otimes\sum_{m\in\Lambda}\left[\frac{1}{l_m\eta} \nabla a_m\sin\theta_m+a_m \left(\frac{\lambda_{J+1}}{\eta}k_{J,m}+\xi_{J,m}\right)\cos\theta_m\right].
	\end{equation}
	Since $\mu\ll \lambda_{J+1}$, the claimed bound for $D\psi$ then follows from (\ref{tamaño am CN}). More generally, we infer from (\ref{tamaño am CN}) that when we compute successive derivatives, the derivative of the trigonometric function dominates, yielding an extra $\lambda_{J+1}$ factor each time, so (\ref{estimaciones psi}) holds.
\end{proof}

Next, we will study the Jacobian of $\psi$. As we have seen, $D\psi$ can be quite large. Fortunately, $\det(D\psi)$ will turn out to be very close to 1 due to orthogonality. This is extremely important because, otherwise, the correction $\phi_c$ needed to obtain a volume-preserving map would not be small. This would mean that the perturbation to the velocity would not be what we want and we would not be able to reduce the Reynolds stress.
\begin{lemma}
	\label{lema fc}
	The function $f_c\coloneqq \det(D\psi)-1$ satisfies the estimates
	\begin{align} 
		\norm{f_c}_N&\lesssim \lambda_{J+1}^{N-3\beta} \qquad \forall N\geq 0, \label{estimates fc CN}\\
		\norm{f_c}_{B_{\infty,\infty}^{-1+\beta}}&\lesssim \lambda_{J+1}^{-1-2\beta}. \label{estimates fc Besov}
	\end{align}
	In addition, we have
	\begin{equation}
		\int_Uf_c=0. 
		\label{mean fc}
	\end{equation}
\end{lemma}
\begin{proof}
	We fix unitary vectors $u_2, u_3\in \RR^3$ so that $\{\zeta,u_2,u_3\}$ is an orthonormal basis of $\RR^3$. Taking into account that $\xi_{J,m}$ and $k_{J,m}$ are orthogonal to $\zeta$ for any $m\in\Lambda$, it follows from (\ref{nabla psi}) that the expression of $D\psi$ in this basis is of the form
	\[D\psi=\begin{pmatrix}
		1+f & \ast & \ast \\ 0 & 1 & 0 \\ 0 & 0 & 1
	\end{pmatrix},\]
	where
	\[f:=\sum_{m\in \Lambda}\frac{1}{l_m\eta}\zeta\cdot\nabla a_m\,\sin\theta_m\]
	and where we denote by $*$ a function whose precise expression we will not need.
	
	Since $\{\zeta,u_2,u_3\}$ is an orthonormal basis, we see that $\det(D\psi)=1+f$, so we have obtained an expression for $f_c$. Using (\ref{tamaño am CN}), we conclude that for any $N\geq 0$ we have
	\[\norm{f_c}_N\lesssim \left(\frac{\delta_{q}}{\delta_{q+2}}\right)^{1/2}\frac{\mu\lambda_{J+1}^{N}}{\eta}=\lambda_{J+1}^{N-3\beta},\]
	where we have used (\ref{relación mu eta}). In particular, by taking $a>1$ sufficiently large we can ensure that $\norm{f_c}_0\leq 1/2$. On the other hand, by (\ref{tamaño am CN}) and (\ref{stationary phase lemma}) we have
	\begin{equation}
		\norm{f_c}_{B_{\infty,\infty}^{-1+\beta}}\lesssim \left(\frac{\delta_{q}}{\delta_{q+2}}\right)^{1/2}\frac{\mu\lambda_{J+1}^{-1+\beta}}{\eta}\left(1+\frac{\mu^N}{\lambda_{J+1}^{N-1}}\right)
		\label{fc aux cota Besov}
	\end{equation}
	for any $N\in \NN$. By (\ref{def mu}), the term in parenthesis equals 2 for $N=(4\beta)^{-1}$. Hence, substituting (\ref{relación mu eta}) into (\ref{fc aux cota Besov}), we obtain (\ref{estimates fc Besov}). Finally, we prove (\ref{mean fc}). Since $\psi=\Id$ in a neighborhood of $\RR^3\backslash U$, the restriction $\psi|_U$ is a diffeomorhphism of $U$. By the change of variables formula, we have 	
	\[\int_U1=\int_U|\det(D\psi)|=\int_U(1+f_c),\]
	where we have used that $\norm{f_c}_0\leq 1/2$ for sufficiently large $a$. We conclude (\ref{mean fc}).
\end{proof}
As a consequence of the previous lemma and the inverse function theorem, each point has a sufficiently small neighborhood where $\psi$ has a smooth inverse, that is, $\psi$ is a local diffeomorphism. We will now check that it is, in fact, a global diffeomorphism:
\begin{proposition}
	The map $\psi$ is invertible and its inverse $\phi_0\in C^\infty(\TT^3,\TT^3)$ satisfies the identity
	\begin{equation}
		\phi_0=\Id+\sum_{m\in\Lambda}\frac{a_m\circ \phi_0}{l_m\eta}\zeta\sin(\theta_m).
		\label{identity phi0}
	\end{equation}
\end{proposition}
\begin{proof}
	Let us consider the map
	\[T:C^0(\TT^3,\TT^3)\to C^0(\TT^3,\TT^3), \quad f\mapsto \Id+\sum_{m\in\Lambda}\frac{a_m\circ f}{l_m\eta}\zeta\sin(\theta_m).\]
	By (\ref{tamaño am CN}) with $N=1$, for any $f,g\in C^0(\TT^3)$ we have
	\begin{align*}
		\norm{Tf-Tg}_0\lesssim \frac{1}{\eta} \max_{m\in\Lambda}\norm{a_m\circ f-a_m\circ g}_0&\lesssim \left(\frac{\delta_{q}}{\delta_{q+2}}\right)^{1/2}\frac{\mu}{\eta}\norm{f-g}_0 \\ &\lesssim \lambda_{J+1}^{-3\beta}\norm{f-g}_0,
	\end{align*}
	where we have also used (\ref{relación mu eta}). Therefore, $\norm{Tf-Tg}_0\leq \frac{1}{2}\norm{f-g}_0$ for sufficiently large $a$. By Banach's fixed point theorem, we conclude that there exists a unique $\varphi\in C^0(\TT^3,\TT^3)$ such that $T\varphi=\varphi$. We will now check that $\varphi$ is the inverse of $\psi$. We compute
	\[\varphi\circ \psi-\Id=\sum_{m\in \Lambda}\frac{1}{l_m\eta}[a_m\circ(\varphi\circ\psi)-a_m]\sin\theta_m,\]
	where we have used \cref{fase invariante}. Again, by (\ref{tamaño am CN}) we have
	\begin{align*}
		\norm{\varphi\circ \psi-\Id}_0\lesssim \frac{1}{\eta} \max_{m\in\Lambda}\norm{a_m\circ (\varphi\circ \psi)-a_m}_0&\lesssim \left(\frac{\delta_{q}}{\delta_{q+2}}\right)^{1/2}\frac{\mu}{\eta}\norm{\varphi\circ \psi-\Id}_0 \\
		&\lesssim \delta_{q+4}\norm{\varphi\circ \psi-\Id}_0.
	\end{align*}
	If $a>1$ is sufficiently large, we obtain $\norm{\varphi\circ \psi-\Id}_0\leq \frac{1}{2}\norm{\varphi\circ \psi-\Id}_0$. Therefore, $\varphi\circ \psi=\Id$. 
	
	Reasoning as in \cref{fase invariante}, we see that $\theta_m\circ\varphi=\theta_m$ for any $m\in \Lambda$. Hence, repeating the previous argument with $\psi$ and $\varphi$ exchanged leads to $\psi\circ \varphi=\Id$. We conclude that $\psi$ is invertible and $\varphi$ is its inverse, which we rename to $\phi_0$. Since the local inverses were smooth, $\phi_0$ is automatically smooth.
\end{proof}
\begin{corollary}
	The diffeomorphism $\phi_0$ satisfies
	\begin{equation}
		\label{estimates phi0}
		\norm{\phi_0-\Id}_0+\lambda_{J+1}^{-1}\norm{\phi_0}_1+\lambda_{J+1}^{-2}\norm{\phi_0}_2\lesssim \left(\frac{\delta_{q}}{\delta_{q+2}}\right)^{1/2}\frac{1}{\eta}.
	\end{equation}
\end{corollary}
\begin{proof}
	The bound for $\norm{\phi_0-\Id}_0$ is a direct consequence of (\ref{tamaño am CN}) with $N=0$. Concerning the first derivative, using (\ref{identity phi0}) we compute
	\begin{align}
		\label{nabla phi0}
		D\phi_0=&\Id+\sum_{m\in \Lambda}\frac{1}{l_m\eta}\zeta\otimes[(\nabla a_m\circ\phi_0)D\phi_0]\sin(\theta_m) \\&+\sum_{m\in \Lambda}a_m\circ\phi_0\;\zeta\otimes\left(\frac{\lambda_{J+1}}{\eta}k_{J,m}+\xi_{J,m}\right)\cos(\theta_m). \nonumber
	\end{align}
	By (\ref{tamaño am CN}), we have
	\[\norm{D \phi_0}_0\lesssim 1+\left(\frac{\delta_{q}}{\delta_{q+2}}\right)^{1/2}\frac{\mu}{\eta}\norm{D\phi_0}_0+\left(\frac{\delta_{q}}{\delta_{q+2}}\right)^{1/2}\frac{\lambda_{J+1}}{\eta}.\]
	By (\ref{relación mu eta}), the coefficient multiplying $\norm{D\phi_0}_0$ on the right-hand side will be smaller than $1/2$ for $a>1$ sufficiently large, so 
	\[\norm{D \phi_0}_0\lesssim\left(\frac{\delta_{q}}{\delta_{q+2}}\right)^{1/2}\frac{\lambda_{J+1}}{\eta}.\]
	Regarding the remaining estimate, differentiating (\ref{nabla phi0}) yields:
	\begin{align*}
		\norm{D\phi_0}_1&\lesssim \frac{1}{\eta}\max_{m\in\Lambda}\left(\norm{a_m}_2\norm{D\phi_0}_0^2+\norm{a_m}_1\norm{D\phi_0}_1+\norm{a_m}_1\norm{D\phi_0}_0\lambda_{J+1}+\norm{a_m}_0\lambda_{J+1}^2\right) \\
		&\lesssim \left(\frac{\delta_{q}}{\delta_{q+2}}\right)^{3/2}\frac{\mu^2\lambda_{J+1}}{\eta^2}+\left(\frac{\delta_{q}}{\delta_{q+2}}\right)^{1/2}\frac{\mu}{\eta}\norm{D\phi_0}_1\\ &\hspace{105pt}+\frac{\delta_{q}}{\delta_{q+2}}\,\frac{\mu\lambda_{J+2}}{\eta^2}+\left(\frac{\delta_{q}}{\delta_{q+2}}\right)^{1/2}\frac{\lambda_{J+1}^2}{\eta}.
	\end{align*}
	Again, by (\ref{relación mu eta}) the coefficient multiplying $\norm{D\phi_0}_1$ on the right-hand side will be smaller than $1/2$ for sufficiently large $a$. The first and third term are smaller than the last term because
	\[\left(\frac{\delta_{q}}{\delta_{q+2}}\,\frac{\mu^2}{\eta\lambda_{J+1}}\right)^{1/2}\leq\left(\frac{\delta_{q}}{\delta_{q+2}}\right)^{1/2}\frac{\mu}{\eta}\lesssim \lambda_{J+1}^{-3\beta}\ll 1\]
	by (\ref{relación mu eta}). We conclude (\ref{estimates phi0}).
\end{proof}

Next, we construct and estimate the correction $\phi_c$: 
\begin{proof}[Proof of \cref{lema dacorogna}]
	By \cref{invertir divergencia vectores} there exists $z\in C^\infty_c(U,\RR^3)$ such that $\Div z=f$ and satisfies the estimate
	\begin{equation}
		\label{estimates z}
		\norm{z}_{N}\lesssim \norm{f}_{B^{N-1+\beta}_{\infty,\infty}}.
	\end{equation}
	Let us define a map $\Phi\in C^\infty(\RR^3\times[0,1],\RR^3)$ as the unique solution of the following initial value problem:
	\[\begin{cases}
		\frac{\partial}{\partial t}\Phi(x,t)=\frac{z(\Phi(x,t))}{1+(1-t)f(\Phi(x,t))}, \\ \Phi(x,0)=x.
	\end{cases}\]
	Since $\norm{f}_0\leq 1/2$, by assumption, the right-hand side of the ODE is bounded by $2\norm{z}_0$ for all $t\in[0,1]$. Hence, the solution for a fixed $x\in \RR^3$ is, indeed, defined in the considered time interval. In addition, it can be proved (see \cite[proof of Lemma 3]{DM}) that 
	\[\Div_x \Phi(x,1)=f(x),\]
	so it suffices to define $\phi(x)=\Phi(x,1)$. Note that $\phi=\Id$ outside of the support of $z$ because
	\begin{equation}
		\label{edo diferencia}
		\begin{cases}
			\frac{\partial}{\partial t}(\Phi-\Id)=\frac{z\circ \Phi}{1+(1-t)f\circ \Phi}, \\ \Phi(x,0)-x=0.
		\end{cases}
	\end{equation}
	
	Therefore, all we have to do is to prove the bounds for $\phi$. Estimate (\ref{estimate phi dacorogna C0}) follows at once from (\ref{estimates z}) and (\ref{edo diferencia}). Concerning the first derivative:
	\[\frac{\partial}{\partial t}\partial_i\Phi=\sum_k\frac{(\partial_k z\circ\Phi)\partial_i\Phi_k}{1+(1-t)f\circ\Phi}-\frac{(z\circ\Phi)(\partial_k f\circ \Phi)\partial_i\Phi_k}{[1+(1-t)f\circ\Phi]^2}.\]
	Hence,
	\begin{align*}
		\frac{d}{dt}\norm{\Phi(\cdot,t)-\Id}_1&\lesssim \norm{z}_0+(\norm{z}_1+\norm{z}_0\norm{f}_1)\norm{\Phi(\cdot,t)}_1 \\
		&\lesssim (\norm{z}_1+\norm{z}_0\norm{f}_1)+(\norm{z}_1+\norm{z}_0\norm{f}_1)\norm{\Phi(\cdot,t)-\Id}_1.
	\end{align*}
	Applying Grönwall's inequality, we conclude
	\[\norm{\phi-\Id}_1=\norm{\Phi(\cdot,1)-\Id}_1\lesssim \exp(\norm{z}_1+\norm{z}_0\norm{f}_1)-1\lesssim \norm{z}_1+\norm{z}_0\norm{f}_1,\]
	where we have used that $\norm{z}_1+\norm{z}_0\norm{f}_1\leq 1$. Substituting (\ref{estimates z}) yields (\ref{estimate phi dacorogna C1}). An analogous argument allows us to estimate the norm $\norm{\Phi(\cdot,t)}_1$ to obtain
	\begin{equation}\label{eq.otra}
		\norm{\Phi(\cdot,t)}_1 \lesssim 1\,.  
	\end{equation}
	for all $t\in[0,1]$.
	
	Concerning the second derivative, we compute
	\begin{align*}
		\frac{\partial}{\partial t}\partial_{ij}\Phi=\sum_{k,l}&\frac{(\partial_{kl}z\circ\Phi)\partial_i\Phi_k\partial_j\Phi_l}{1+(1-t)f\circ\Phi}-2\frac{(\partial_kz\circ\Phi)(\partial_lf\circ \Phi)\partial_i\Phi_k\partial_j\Phi_l}{[1+(1-t)f\circ\Phi]^2}\\ &+\frac{(\partial_{k}z\circ\Phi)\partial_{ij}\Phi_k}{1+(1-t)f\circ\Phi}-\frac{(z\circ\Phi)(\partial_{kl}f\circ \Phi)\partial_i\Phi_k\partial_j\Phi_l}{[1+(1-t)f\circ\Phi]^2} \\ &-\frac{(z\circ\Phi)(\partial_{k}f\circ \Phi)\partial_{ij}\Phi_k}{[1+(1-t)f\circ\Phi]^2}+2\frac{(z\circ\Phi)(\partial_{k}f\circ \Phi)(\partial_{l}f\circ \Phi)\partial_i\Phi_k\partial_j\Phi_l}{[1+(1-t)f\circ\Phi]^3}.
	\end{align*}
	Thus, using~\eqref{eq.otra} we have
	\begin{align*}
		\frac{d}{dt}\norm{\Phi(\cdot,t)}_2\lesssim &\left(\norm{z}_2+\norm{z}_1\norm{f}_1+\norm{z}_0\norm{f}_2+\norm{z}_0\norm{f}_1^2\right)\\&+(\norm{z}_1+\norm{z}_0\norm{f}_1)\norm{\Phi(\cdot,t)}_2.
	\end{align*}
	We conclude, again from Grönwall's inequality
	\begin{align*}
		\norm{\Phi(\cdot,1)}_2&\lesssim \left(\norm{z}_2+\norm{z}_1\norm{f}_1+\norm{z}_0\norm{f}_2+\norm{z}_0\norm{f}_1^2\right)\frac{e^{\norm{z}_1+\norm{z}_0\norm{f}_1}-1}{\norm{z}_1+\norm{z}_0\norm{f}_1}\\&\lesssim \norm{z}_2+\norm{z}_1\norm{f}_1+\norm{z}_0\norm{f}_2+\norm{z}_0\norm{f}_1^2,
	\end{align*}
	where we have used that $\norm{z}_1+\norm{z}_0\norm{f}_1\leq 1$. Substituting (\ref{estimates z}) yields (\ref{estimate phi dacorogna C2}).
\end{proof}

Constructing the correction $\phi_c$ with the desired properties is a simple consequence of these results:
\begin{corollary}
	There exists a diffeomorphism $\phi_c\in C^\infty(\TT^3,\TT^3)$ such that $\phi_c\equiv\Id$ in a neighborhood of $\TT^3\backslash U$, with \[\det[D(\phi_c\circ\phi_0)]=1\]
	and satisfying the following estimates for $N=0,1,2$:
	\begin{equation}
		\norm{\phi_c-\Id}_N\lesssim \lambda_{J+1}^{N-1-2\beta}. \label{estimates phic}
	\end{equation}
	In addition, its inverse satisfies
	\begin{equation}
		\norm{\phi_c^{-1}-\Id}_N\lesssim \lambda_{J+1}^{N-1-2\beta} \qquad N=0,1,2. \label{estimates phic-1}
	\end{equation}
\end{corollary}
\begin{proof}
	Since $\psi=\Id$ in a neighborhood of $\TT^3\backslash U$, the support of $f_c$ is contained in~$U$. In addition, we have $\int_Uf_c=0$ by \cref{lema fc}. We can then apply \cref{lema dacorogna} to $f_c$,  obtaining a diffeomorphism $\phi_c\in C^\infty(\TT^3,\TT^3)$ such that $\phi_c=\Id$ in a neighborhood of $\TT^3\backslash U$ and
	\[\det(D\phi_c)=1+f_c=\det(D\psi)=\det[D(\phi_0^{-1})],\]
	which leads to $\det[D(\phi_c\circ \phi_0)]=1$. Regarding the estimates, we have
	\[\norm{f_c}_{N+\beta}\lesssim \lambda_{J+1}^{N-2\beta}\]
	by interpolation of (\ref{estimates fc CN}). Substituting our estimates for $f_c$ in (\ref{estimate phi dacorogna C0}), (\ref{estimate phi dacorogna C1}) and (\ref{estimate phi dacorogna C2}) yield (\ref{estimates phic}). 
	
	We will now estimate $\phi_c^{-1}$. The bound for $\norm{\phi_c^{-1}-\Id}_0$ follows immediately from the bound for $\phi_c$. Regarding the $C^1$-norm, we have \[D(\phi_c^{-1})=(D\phi_c)^{-1}\circ\phi_c^{-1},\]
	so
	\begin{align*}
		\norm{D(\phi_c^{-1})-\Id}_0=\norm{(D\phi_c)^{-1}-\Id}_0&\lesssim \norm{(D\phi_c)^{-1}}_0\norm{\Id-D\phi_c}_0 \\ &\lesssim \norm{\Id-D\phi_c}_0\norm{(D\phi_c)^{-1}-\Id}_0+\norm{\Id-D\phi_c}_0.
	\end{align*}
	Since $\norm{\Id-D\phi_c}_0\ll 1$, we readily obtain
	\[\norm{D(\phi_c^{-1})-\Id}_0=\norm{(D\phi_c)^{-1}-\Id}_0\lesssim\norm{\Id-D\phi_c}_0,\]
	so we deduce the bound for $\phi_c^{-1}$ from the bound for $\phi_c$. Concerning the $C^2$-norm, if we differentiate twice the identity $\phi_c\circ \phi_c^{-1}=\Id$, we obtain
	\[\partial_{lm}(\phi_c)_i\circ\phi_c^{-1}\,\partial_j(\phi_c^{-1})_l\,\partial_k(\phi_c^{-1})_m+\partial_l(\phi_c)_i\circ\phi_c^{-1}\,\partial_{jk}(\phi_c^{-1})_l=0.\]
	Multiplying by $\partial_i(\phi_c^{-1})_n$ and summing over $i$ leads to
	\[\partial_{jk}(\phi_c^{-1})_n=-\partial_{lm}(\phi_c)_i\circ\phi_c^{-1}\,\partial_j(\phi_c^{-1})_l\,\partial_k(\phi_c^{-1})_m\,\partial_i(\phi_c^{-1})_n.\]
	Therefore,
	\[\norm{D^2(\phi_c^{-1})}_0\lesssim \norm{D^2(\phi_c)}_0\norm{D(\phi_c^{-1})}_0^3\lesssim \norm{D^2(\phi_c)}_0.\]
	Hence, we deduce the bound for $\phi_c^{-1}$ from the bound for $\phi_c$.
\end{proof}

\begin{remark}
	The implicit constants in (\ref{estimates phic}) depend on $U$ through the use of \cref{invertir divergencia vectores}. We must ensure that they remain uniformly bounded throughout the iterative process, with a bound independent of the parameter $a$. This means that we are safe to dismiss them. In \cref{section prueba prop steps} we will see that one can take $U=s U_1+p$, where $U_1$ is a fixed smooth domain, $s\in (0,1)$ and $p\in\RR^3$. Therefore,  \cref{invertir divergencia vectores} ensures that the constants remain uniformly bounded as we shrink the domain in subsequent steps, by letting $s$ become small.
\end{remark}

Once that we have estimated $\phi_0$ and $\phi_c$, we will check that $\Phi_{J+1}$ satisfies the required inductive hypotheses:
\begin{lemma}
	\label{lemma induct hip diffeo C0}
	If $\alpha>0$ is sufficiently small and $a>1$ is sufficiently large (depending on $\alpha$), the new diffeomorphism $\Phi_{J+1}$ satisfies (\ref{inductive Phiqj C0}) and (\ref{conclusión stages 2}).
\end{lemma}
\begin{proof}
	Using (\ref{estimates phi0}) and (\ref{estimates phic}) we obtain
	\begin{align*}
		\norm{\Phi_{J+1}-\Phi_J}_0&=\norm{\phi_{J+1}-\Id}_0\leq \norm{\phi_c\circ \phi_0-\phi_0}_0+\norm{\phi_0-\Id}_0\\&\lesssim \lambda_{J+1}^{-1-2\beta}+ \delta_{q}^{1/2}\delta_{q+2}^{-1/2}\eta^{-1}\lesssim \delta_q\delta_{q+2}^{-1/2}\delta_{q+4}^{-1}\lambda_{J+1}^{-1+2\alpha} \\ &\lesssim \delta_q^{1/2}\delta_{q+1}^{-1/2}\delta_{q+2}^{-1/2}\delta_{q+4}^{-1}\lambda_{J+1}^{-1}
	\end{align*}
	where we have substituted (\ref{def eta}). One can check that the following inequality holds for any $q\geq 0$:
	\begin{align*}
		\log_a\left(\frac{\delta_q^{1/2}}{\delta_{q+1}^{1/2}\delta_{q+2}^{1/2}\delta_{q+4}}\right)&=-2\alpha(b^q-1)+2\alpha(b^{q+1}-1)+2\alpha(b^{q+2}-1)+4\alpha(b^{q+4}-1) \\ &=2\alpha[b^q(2b^4+b^2+b-1)-3]\leq 2\alpha[ 20(b^{q+1}-1)-2^{-3}],
	\end{align*}
	where $b=\frac{3}{2}$. Since $a>1$, we deduce
	\begin{equation}
		\frac{\delta_q^{1/2}}{\delta_{q+1}^{1/2}\delta_{q+2}^{1/2}\delta_{q+4}}\leq a^{-\alpha/4}\delta_{q+1}^{-10}.
		\label{aux deltas diffeo}
	\end{equation}
	Hence, taking $a>1$ sufficiently large (depending on $\alpha$) so as to compensate the numerical constants, we finally obtain
	\[\norm{\Phi_{J+1}-\Phi_J}_0\leq \frac{1}{4}\delta_{q+1}^{-10}\lambda_{J+1}^{-1}.\]
	Since $\phi_0^{-1}$ and $\phi_c^{-1}$ satisfy the same estimates as $\phi_0$ and $\phi_c$, an analogous bound holds for $\Phi_{J+1}^{-1}$. Taking into account that $\delta_{q(J+1)+1}\leq \delta_{q(J)+1}$, we conclude (\ref{conclusión stages 2}). 
	
	On the other hand, by further reducing $\alpha$, if necessary, we also have
	\begin{align*}
		\norm{\Phi_{J+1}-\Phi_J}_0+\norm{\Phi_{J+1}^{-1}-\Phi_J^{-1}}_0&\stackrel{(\ref{cota delta en función de lambdaJ+1})}{\leq} \frac{1}{2}\lambda_{J+1}^{-1/2}=\frac{1}{2}a^{-(b^{1/64}-1)/2}a^{-(b^{\abs{J+1}}-b^{1/64})/2}\\&\leq \frac{1}{2}a^{-\ln(b)/128}a^{-b^{1/64}\ln(b)\abs{J}/2}\leq \frac{1}{2}a^{-\beta}\hspace{0.5pt}2^{-64\abs{J}}
	\end{align*}
	where we have also taken $a>1$ to be sufficiently large. Combining this with (\ref{inductive Phiqj C0}) yields
	\begin{align*}
		\norm{\Phi_{J+1}-\Id}_0+\|\Phi_{J+1}^{-1}-\Id\|_0&\leq \frac{1}{2}a^{-\beta}\hspace{0.5pt}2^{-64\abs{J}}+ a^{-\beta}\left(1-2^{-64\abs{J}}\right)\\ &=a^{-\beta}\left(1-2^{-(64\abs{J}+1)}\right)=a^{-\beta}\left(1-2^{-64\abs{J+1}}\right),
	\end{align*}
	so (\ref{inductive Phiqj C0}) holds for $J+1$, too.
\end{proof}

Finally, we study the higher-order norms of the new diffeomorphism: 
\begin{lemma}
	If $a>1$ is sufficiently large (depending on $\alpha$), the diffeomorphism $\Phi_{J+1}$ satisfies (\ref{inductive Phiqj C1}), (\ref{inductive Phiqj C2}) and (\ref{conclusión stages 3}).
\end{lemma}
\begin{proof}
	Using (\ref{estimates phi0}), we have 
	\[\norm{\phi_0}_1\lesssim \left(\frac{\delta_{q}}{\delta_{q+2}}\right)^{1/2}\frac{\lambda_{J+1}}{\eta} = \frac{\delta_q^{1/2}\lambda_{J+1}^{2\alpha}}{\delta_{q+2}^{1/2}\delta_{q+4}}\leq \frac{\delta_q^{1/2}}{\delta_{q+1}^{1/2}\delta_{q+2}^{1/2}\delta_{q+4}} \stackrel{(\ref{aux deltas diffeo})}{\leq} a^{-\alpha/4}\delta_{q+1}^{-10}.\]
	Meanwhile, by (\ref{estimates phic}) we have $\norm{\phi_c}_1\lesssim 1$. Taking $a>1$ sufficiently large (depending on $\alpha$) so as to compensate the numerical constants, we obtain
	\begin{equation}
		\norm{\phi_{J+1}}_1\leq \norm{\phi_c}_1\norm{\phi_0}_1\leq \frac{1}{4} \delta_{q+1}^{-10}. \label{aux norma 1 phi_J+1}
	\end{equation}
	Using (\ref{inductive Phiqj C1}), we have
	\[\norm{\Phi_{J+1}}_1\leq \norm{\phi_{J+1}}_1\norm{\Phi_J}_1\leq \frac{1}{4}\delta_{q(J)+1}^{-[10+s(J)]}.\]
	Unless $j(J)=l(J)=7$, we have $q(J+1)=q(J)$ and $s(J+1)=s(J)+10$. Thus, in that case we deduce
	\begin{equation}
		\delta_{q(J)+1}^{-[10+s(J)]}=\delta_{q(J+1)+1}^{-s(J+1)}.
		\label{deltas e índices}
	\end{equation}
	On the other hand, if $j(J)=l(J)=7$, we have
	\[\begin{cases}
		q(J)+1=q(J+1), \\ s(J)+10=1920, \\ s(J+1)=1280.
	\end{cases}\]
	Hence,
	\[\delta_{q(J)+1}^{-[10+s(J)]}=\delta_{q(J+1)}^{-1920}\leq \delta_{q(J+1)+1}^{-1280}=\delta_{q(J+1)+1}^{-s(J+1)}\]
	because
	\begin{align*}
		\log_a\left(\frac{\delta_{q(J+1)}^{3/2}}{\delta_{q(J+1)+1}}\right)&=4\alpha(b^{q(J+1)+1}-1)-6\alpha(b^{q(J+1)}-1)=2\alpha+4\alpha b^{q(J+1)}\left(b-\frac{3}{2}\right)\\&=2\alpha>0.
	\end{align*}
	In either case, (\ref{deltas e índices}) holds and we conclude
	\[\norm{\Phi_{J+1}}_1\leq\frac{1}{4}\delta_{q(J+1)+1}^{-s(J+1)}.\]
	
	Since $\phi_0^{-1}$ and $\phi_c^{-1}$ satisfy the same estimates as $\phi_0$ and $\phi_c$, an analogous bound holds for $\Phi_{J+1}^{-1}$. Therefore, we conclude 
	\[\norm{\Phi_{J+1}}_1+\norm{\Phi_{J+1}^{-1}}_1\leq\frac{1}{2}\delta_{q(J+1)+1}^{-s(J+1)},\]
	so (\ref{inductive Phiqj C1}) holds for $J+1$. On the other hand, if we separate in cases and argue as before, we deduce from (\ref{inductive Phiqj C1}) for $\Phi_J$  that
	\[\norm{\Phi_{J}}_1+\norm{\Phi_{J}^{-1}}_1\leq\delta_{q(J)+1}^{-s(J)}\leq \frac{1}{2}\delta_{q(J+1)+1}^{-s(J+1)}\]
	if $a>1$ is sufficiently large. Applying the triangle inequality yields (\ref{conclusión stages 3}). Concerning the $C^2$-norm, we have
	\[D\phi_{J+1}=D\phi_c\circ\phi_0\,D\phi_0,\]
	so
	\begin{align*}
		\norm{D\phi_{J+1}}_1\lesssim \norm{D\phi_c}_1\norm{D\phi_0}_0^2+\norm{D\phi_c}_0\norm{D\phi_0}_1&\lesssim \left(\frac{\delta_{q}}{\delta_{q+2}}\right)^{1/2}\frac{\lambda_{J+1}^2}{\eta}\\&\stackrel{(\ref{aux deltas diffeo})}{\leq} a^{-\alpha/4}\delta_{q+1}^{-10}\lambda_{J+1},
	\end{align*}
	where we have used (\ref{estimates phi0}) and (\ref{estimates phic}). Using this estimate along with (\ref{aux norma 1 phi_J+1}) and the inductive hypotheses (\ref{inductive Phiqj C1}) and (\ref{inductive Phiqj C2}) yields
	\begin{align*}
		\norm{D\Phi_{J+1}}_1&\lesssim \norm{D\phi_{J+1}}_1\norm{D\Phi_J}_0^2+\norm{D\phi_{J+1}}_0\norm{D\phi_J}_1\\&\lesssim a^{-\alpha/4}\delta_{q(J)+1}^{-[10+2s(J)]}\lambda_{J+1}+\delta_{q(J)+1}^{-[10+2s(J)]}\lambda_J\\&\stackrel{(\ref{otra relación entre las lambdas})}{\lesssim} (a^{-\alpha/4}+\lambda_{J+1}^{-2\beta})\delta_{q(J)+1}^{-2[10+s(J)]}\lambda_{J+1} \stackrel{(\ref{deltas e índices})}{\lesssim} (a^{-\alpha/4}+\lambda_{J+1}^{-2\beta})\delta_{q(J+1)+1}^{-2s(J+1)}\lambda_{J+1}.
	\end{align*}
	Choosing $a>1$ sufficiently large (depending on $\alpha$) so as to compensate the numerical constants, we obtain
	\[\norm{\Phi_{J+1}}_2\leq \frac{1}{2}\delta_{q(J+1)+1}^{-2s(J+1)}\lambda_{J+1}.\]
	Since $\phi_0^{-1}$, $\phi_c^{-1}$ and $\Phi_J^{-1}$ satisfy the same estimates as $\phi_0$, $\phi_c$ and $\Phi_J$, an analogous bound holds for $\Phi_{J+1}^{-1}$. We conclude (\ref{inductive Phiqj C2}).
\end{proof}

\section{Estimates on the velocity} \label{section estimates v}
We begin with some basic estimates:
\begin{lemma}
	For any $m\in \Lambda$ and any $N\geq 0$ we have
	\begin{align}
		\norm{b_m}_0&\leq 2\delta_q^{1/2}, \label{cota bm C0}\\
		\norm{b_m}_N&\lesssim \delta_q^{1/2}\mu^N. \label{cota bm CN}
	\end{align}
\end{lemma}
\begin{proof}
	The bound (\ref{cota bm C0}) is a direct consequence of (\ref{cota gamma stages C0}). The bounds for the derivatives follow from (\ref{cota gamma stages CN}) and (\ref{def chim}). Remember that $\mu \geq \ell^{-1}$, so the derivatives of the cutoff $\chi_m$ dominate.
\end{proof}
This allows us to estimate $w_0$:
\begin{lemma}
	The correction $w_0$ satisfies
	\begin{equation}
		\norm{w_0}_0+\lambda_{J+1}^{-1}\norm{w_0}_1\leq  48\delta_{q}^{1/2}. \label{cota w0 exacta}
	\end{equation}
	In addition, for any $N\geq0$ we have
	\begin{equation}
		\norm{w_0}_N\lesssim \delta_{q}^{1/2}\lambda_{J+1}^N.
		\label{cota w0 con constantes}
	\end{equation}
\end{lemma}
\begin{proof}
	Since at most 8 of the cutoffs $\chi_m$ are nonzero at any given point, it follows from (\ref{cota bm C0}) that $\norm{w_0}_0\leq 16\delta_{q}^{1/2}$. Regarding the higher norms, it follows from (\ref{cota bm CN}) that the derivatives of $\cos(\theta_m)$ dominate because $\lambda_{J+1}\gg \mu$. Thus, we deduce (\ref{cota w0 con constantes}). Furthermore, it follows from the definition of the parameters that $\mu^{-1}\lambda_{J+1}$ can be made arbitrarily small by taking $a>1$ sufficiently large. Hence, it can compensate the implicit constant that appears in \cref{cota bm CN} for $N=1$. Therefore, the term that comes from differentiating $\cos(\theta_m)$ dominates and we conclude (\ref{cota w0 exacta}).
\end{proof}
Combining this with (\ref{otra relación entre las lambdas}), we see that the $C^1$-norm of the perturbation $w_0$ is much larger than the $C^1$-norm of $v_J$. In particular, we have
\begin{equation}
	\norm{v_J+w_0}_1\lesssim \delta_q^{1/2}\lambda_{J+1}.
	\label{cota C1 suma}
\end{equation}

So far, we have estimated the \textit{ad hoc} field $w_0$. However, to have the complete field $v_{J+1}$ under control, we must check that the correction $w_c$ is, indeed, very small. For this, we need to understand how the diffeomorphism $\phi_{J+1}$ acts on $v_J$. 
\begin{lemma}
	The correction $w_c$ satisfies
	\begin{equation}
		\label{estimates wc}
		\norm{w_c}_0+\lambda_{J+1}^{-1}\norm{w_c}_1\leq \lambda_{J+1}^{-\beta}.
	\end{equation}
\end{lemma}
\begin{proof}
	Let $\widetilde{v}_{J+1}$ be the pushforward of $v_J$ by $\phi_0$, that is,
	\[\widetilde{v}_{J+1}\coloneqq (D\phi_0\,v_J)\circ \phi_0^{-1}.\]
	Since $\phi_c$ is very close to the identity, we expect this to be the dominating term in $v_{J+1}$. It follows from (\ref{identity phi0}) that
	\begin{align*}
		D\phi_0\,v_J=v_J&+\sum_{m\in\Lambda}\left[(b_m\circ \phi_0)\,\zeta\cos(\theta_m)+\frac{a_m\circ\phi_0}{l_m\eta}[(v_J-u_m)\cdot\nabla\theta_m]\,\zeta \cos(\theta_m)\right]\\&+\sum_{m\in\Lambda}\frac{1}{l_m\eta}[v_J\cdot\nabla(a_m\circ \phi_0)]\,\zeta\sin(\theta_m).
	\end{align*}
	Taking into account \cref{fase invariante}, we obtain
	\begin{align*}
		\widetilde{v}_{J+1}=v_J\circ\phi_0^{-1}+w_0&+\sum_{m\in \Lambda}\frac{a_m}{l_m\eta}[(v_J\circ\phi_0^{-1}-u_m)\cdot\nabla\theta_m]\,\zeta \cos(\theta_m)\\&+\sum_{m\in\Lambda}\frac{1}{l_m\eta}(\nabla a_m\cdot \widetilde{v}_{J+1})\,\zeta\sin(\theta_m).
	\end{align*}
	
	We will prove the desired estimates using the following decomposition:
	\begin{equation}
		\norm{w_c}_N=\norm{v_{J+1}-(v_0+w_0)}_N\leq \norm{v_{J+1}-\widetilde{v}_{J+1}}_N+\norm{\widetilde{v}_{J+1}-(v_J+w_0)}_N. \label{decomposition wc}
	\end{equation}
	We begin by studying the difference
	\begin{align}
		\widetilde{v}_{J+1}-(v_J+w_0)=(v_J\circ\phi_0^{-1}-v_J)&+\sum_{m\in \Lambda}\frac{a_m}{l_m\eta}[(v_J\circ\phi_0^{-1}-u_m)\cdot\nabla\theta_m]\,\zeta \cos(\theta_m) \nonumber\\&+\sum_{m\in\Lambda}\frac{1}{l_m\eta}(\nabla a_m\cdot \widetilde{v}_{J+1})\,\zeta\sin(\theta_m). \label{diferencia widetildevJ+1}
	\end{align}
	By (\ref{derivative vqj}) and (\ref{estimates phi0}), for any $x\in \supp a_m$ we have
	\begin{align}
		\begin{split}
			\label{nabla theta aprovechando ortogonalidad}
			\abs{(v_J(\phi_0^{-1}(x))-u_m)\cdot\nabla\theta_m}&\lesssim \lambda_{J+1}\abs{v_J(\phi_0^{-1}(x))-v_J(\mu^{-1}m)}\\&\lesssim \delta_{q^\ast}^{1/2}\lambda_J\lambda_{J+1}\left(\abs{\phi_0^{-1}(x)-x}+\abs{x-\mu^{-1}m}\right) \\ &\lesssim \delta_{q^\ast}^{1/2}\lambda_J\lambda_{J+1}\left(\delta_{q}^{1/2}\delta_{q+2}^{-1/2}\eta^{-1}+\mu^{-1}\right) \\
			&\lesssim \delta_{q^\ast}^{1/2}\lambda_J\lambda_{J+1}\mu^{-1}.
		\end{split}
	\end{align}
	Substituting this estimate in (\ref{diferencia widetildevJ+1}), along with (\ref{derivative vqj}), (\ref{estimates phi0}) and (\ref{tamaño am CN}), yields
	\begin{align*}
		\norm{\widetilde{v}_{J+1}-(v_J+w_0)}_0\lesssim \frac{\delta_{q}}{\delta_{q+2}^{1/2}}\frac{\lambda_J}{\eta}+\left(\frac{\delta_{q}\delta_{q^\ast}}{\delta_{q+2}}\right)^{1/2}\frac{\lambda_J\lambda_{J+1}}{\mu\eta}+\left(\frac{\delta_{q}}{\delta_{q+2}}\right)^{1/2}\frac{\mu}{\eta}\norm{\widetilde{v}_{J+1}}_0.
	\end{align*}
	Since $\delta_{q^\ast}\geq \delta_q$ and $\lambda_{J+1}\gg\eta$, the second term is larger than the first. By (\ref{relación mu eta}), the factor multiplying $\norm{\widetilde{v}_{J+1}}_0$ can be made arbitrarily small by choosing $a>1$ sufficiently large. Therefore, it follows from (\ref{inductive j cota C0}) and (\ref{cota w0 exacta}) that
	\begin{equation}
		\norm{\widetilde{v}_{J+1}-(v_J+w_0)}_0\lesssim\left(\frac{\delta_{q}\delta_{q^\ast}}{\delta_{q+2}}\right)^{1/2}\frac{\lambda_J\lambda_{J+1}}{\mu\eta}\stackrel{(\ref{cociente de parámetros})}{\lesssim}\lambda_{J+1}^{-2\beta}.
		\label{difference with widetildev C0}
	\end{equation}
	In particular, by (\ref{inductive j cota C0}) and (\ref{cota w0 exacta}), this means that
	\begin{equation}
		\norm{\widetilde{v}_{J+1}}_0\lesssim 1.
		\label{tamaño widetildevJ+1 C0}
	\end{equation}
	
	Next, the $C^1$-norm of the first term on the right-hand side of (\ref{diferencia widetildevJ+1}) is bounded by
	\[\norm{v_J}_1\norm{\phi_0^{-1}}_1\lesssim \left(\frac{\delta_{q}\delta_{q^\ast}}{\delta_{q+2}}\right)^{1/2}\frac{\lambda_J\lambda_{J+1}}{\eta}\stackrel{(\ref{cociente de parámetros})}{\lesssim}\lambda_{J+1}^{1-2\beta}.\]
	Up to a multiplicative constant, the $C^1$-norm of the second term on the right-hand side of (\ref{diferencia widetildevJ+1}) can be estimated by
	\begin{align*}
		&\frac{1}{\eta}\max_{m\in \Lambda}\left[\left(\norm{a_m}_1+\lambda_{J+1}\norm{a_m}_0\right)\frac{\delta_{q^\ast}^{1/2}\lambda_J\lambda_{J+1}}{\mu}+\norm{a_m}_0\norm{v_J}_1\norm{\phi_0^{-1}}_1\lambda_{J+1}\right]\lesssim \\[5pt]&\hspace{\linewidth-200pt}\lesssim \left(\frac{\delta_{q}\delta_{q^\ast}}{\delta_{q+2}}\right)^{1/2}\frac{\lambda_J\lambda_{J+1}^2}{\mu\eta}\stackrel{(\ref{cociente de parámetros})}{\lesssim}\lambda_{J+1}^{1-2\beta},
	\end{align*}
	where we have used (\ref{nabla theta aprovechando ortogonalidad}). We have also used (\ref{relación mu eta}) to simplify the parameters and combine everything into a single term. Up to a multiplicative constant, the $C^1$-norm of the last term in (\ref{diferencia widetildevJ+1}) can be estimated by
	\begin{align*}
		\frac{1}{\eta}&\max_{m\in \Lambda}\left(\norm{a_m}_2\norm{\widetilde{v}_{J+1}}_0+\norm{a_m}_1\norm{\widetilde{v}_{J+1}}_0\lambda_{J+1}+\norm{a_m}_1\norm{\widetilde{v}_{J+1}}_1\right)\stackrel{(\ref{tamaño widetildevJ+1 C0})}{\lesssim} \\ &\lesssim\left(\frac{\delta_q}{\delta_{q+2}}\right)^{1/2}\frac{\mu\lambda_{J+1}}{\eta}+\left(\frac{\delta_{q}}{\delta_{q+2}}\right)^{1/2}\frac{\mu}{\eta}(\norm{v_J+w_0}_1+\norm{\widetilde{v}_{J+1}-(v_J+w_0)}_1)\\ &\stackrel{(\ref{cota C1 suma})}{\lesssim} \left(\frac{\delta_q}{\delta_{q+2}}\right)^{1/2}\frac{\mu\lambda_{J+1}}{\eta}+\left(\frac{\delta_{q}}{\delta_{q+2}}\right)^{1/2}\frac{\mu}{\eta}\norm{\widetilde{v}_{J+1}-(v_J+w_0)}_1 \\ &\stackrel{(\ref{relación mu eta})}{\lesssim} \lambda_{J+1}^{1-3\beta}+\lambda_{J+1}^{-3\beta}\norm{\widetilde{v}_{J+1}-(v_J+w_0)}_1.
	\end{align*}
	Since $\lambda_{J+1}^{-3\beta}$ can be made arbitrarily small by choosing $a>1$ sufficiently large, after substituting the estimate for each term in (\ref{diferencia widetildevJ+1}), we conclude
	\begin{equation}
		\norm{\widetilde{v}_{J+1}-(v_J+w_0)}_1\lesssim\lambda_{J+1}^{1-2\beta}.
		\label{difference with widetildev C1}
	\end{equation}
	In particular, combining this with (\ref{cota C1 suma}) we have
	\[\norm{\widetilde{v}_{J+1}}_1\lesssim \delta_q^{1/2}\lambda_{J+1}.\]
	
	It will be also useful to estimate
	\begin{align*}
		\norm{D\widetilde{v}_{J+1}\circ\phi_c^{-1} -D\widetilde{v}_{J+1}}_0&\leq 2\norm{\widetilde{v}_{J+1}-v_J-w_0}_1+2\norm{v_J}_1+\norm{Dw_0\circ\phi_c^{-1}-Dw_0}_0 \\&\lesssim  \lambda_{J+1}^{1-2\beta}+\delta_{q^\ast}^{1/2}\lambda_J+(\delta_q^{1/2}\lambda_{J+1}^2)\lambda_{J+1}^{-1-2\beta} \stackrel{(\ref{otra relación entre las lambdas})}{\lesssim}\lambda_{J+1}^{1-2\beta},
	\end{align*}
	where we have used (\ref{estimates phic}) with $N=0$ and (\ref{cota w0 con constantes}) with $N=2$.
	
	We will now focus on estimating the difference $v_{J+1}-\widetilde{v}_{J+1}$. Since $\phi_{J+1}=\phi_c\circ\phi_0$, we see that $v_{J+1}$ is the pushforward of $\widetilde{v}_{J+1}$ by the diffeomorphism $\phi_c$, that is,
	\begin{equation}
		v_{J+1}=(D\phi_c\,\widetilde{v}_{J+1})\circ \phi_c^{-1}.
		\label{expresión vJ+1 en función de widetilde}
	\end{equation}
	Hence
	\begin{align*}
		\norm{v_{J+1}-\widetilde{v}_{J+1}}_0&\leq \norm{[(D\phi_c-\Id)\widetilde{v}_{J+1}]\circ\phi_c^{-1}}_0+\norm{\widetilde{v}_{J+1}\circ\phi_c^{-1}-\widetilde{v}_{J+1}}_0 \\ &\lesssim \norm{\phi_c-\Id}_1+\norm{v_J+w_0}_1\norm{\phi_c-\Id}_0+\norm{\widetilde{v}_{J+1}-v_J-w_0}_0 \\&\lesssim \lambda_{J+1}^{-2\beta}+(\delta_q^{1/2}\lambda_{J+1})\lambda_{J+1}^{-1-2\beta}+\lambda_{J+1}^{-2\beta}\lesssim \lambda_{J+1}^{-2\beta},
	\end{align*}
	where we have used (\ref{estimates phic}), (\ref{cota C1 suma}) and (\ref{difference with widetildev C0}). Substituting this bound and (\ref{difference with widetildev C0}) into (\ref{decomposition wc}), we obtain
	\begin{equation}
		\norm{w_c}_0\lesssim \lambda_{J+1}^{-2\beta}.
		\label{estimate wc 2beta}
	\end{equation}
	Choosing $a>1$ sufficiently large so as to compensate the numerical constants, we obtain \[\norm{w_c}_0\leq \frac{1}{2}\lambda_{J+1}^{-\beta}.\]
	
	Concerning the $C^1$-norm, differentiating (\ref{expresión vJ+1 en función de widetilde}) yields
	\begin{align*}
		Dv_{J+1}-D\widetilde{v}_{J+1}=&[(D^2\phi_c\,\widetilde{v}_{J+1})\circ\phi_c^{-1}]\,D(\phi_c^{-1})\\&+[((D\phi_c-\Id)D\widetilde{v}_{J+1})\circ\phi_c^{-1}+D\widetilde{v}_{J+1}\circ\phi_c^{-1}-D\widetilde{v}_{J+1}]\,D(\phi_c^{-1})\\&+D\widetilde{v}_{J+1}[D(\phi_c^{-1})-\Id].
	\end{align*}
	Using (\ref{estimates phic-1}), we have
	\begin{align*}
		\norm{Dv_{J+1}-D\widetilde{v}_{J+1}}_0&\lesssim \norm{\phi_c}_2+\norm{\phi_c-\Id}_1\norm{\widetilde{v}_{J+1}}_1+\norm{D\widetilde{v}_{J+1}\circ\phi_c^{-1}-D\widetilde{v}_{J+1}}_0\\&\hspace{10pt}+\norm{\widetilde{v}_{J+1}}_1\norm{\phi_c^{-1}-\Id}_1 \\ &\lesssim \lambda_{J+1}^{1-2\beta}+\lambda_{J+1}^{-2\beta}(\delta_q^{1/2}\lambda_{J+1})+\lambda_{J+1}^{1-2\beta}+(\delta_q^{1/2}\lambda_{J+1})\lambda_{J+1}^{-2\beta}\\ &\lesssim \lambda_{J+1}^{1-2\beta}
	\end{align*}
	by (\ref{relación mu eta}) and (\ref{cociente de parámetros}). Substituting this and (\ref{difference with widetildev C1}) in (\ref{decomposition wc}) leads to
	\[\norm{w_c}_1\lesssim \lambda_{J+1}^{1-2\beta}.\]
	Again, taking $a>1$ large leads to the desired bound for $w_c$.
\end{proof}

Once we have estimated $w_0$ and $w_c$, it is easy to check that the new velocity $v_{J+1}$ satisfies all the desired properties. We begin with some simple estimates:
\begin{lemma}
	If $\alpha>0$ is sufficiently small and $a>1$ is sufficiently large (depending on $\alpha$), the new velocity $v_{J+1}$ satisfies the estimates (\ref{inductive j cota C0}) and (\ref{derivative vqj}) for $J+1$. In addition, (\ref{cambio iteración j}) holds.
\end{lemma}
\begin{proof}
	By (\ref{aux lambda negativa}) and (\ref{estimates wc}), we have
	\[\norm{w_c}_0+\lambda_{J+1}^{-1}\norm{w_c}_1\leq \lambda_{J+1}^{-\beta}\leq \delta_{q+4}\leq \delta_q^{1/2},\]
	which, combined with (\ref{cota w0 exacta}), yields
	\[\norm{v_{J+1}-v_J}_0+\lambda_{J+1}^{-1}\norm{v_{J+1}-v_J}_1\leq 50\delta_q^{1/2},\] 
	so (\ref{cambio iteración j}) holds. Using this, along with (\ref{inductive j cota C0}), leads to
	\[\norm{v_{J+1}}_0\leq \norm{v_J}_0+\norm{v_{J+1}-v_{J}}_0\leq M(J)+2^6\delta_{q}^{1/2}.\]
	Note that
	\[\delta_{q}^{1/2}=\lambda_{q}^{-2\alpha}=a^{-2\alpha(b^{q}-1)}\leq a^{-2\alpha\ln(b)\,q},\]
	which can be made smaller than $2^{-q}$ by choosing $a>1$ sufficiently large. Hence,
	\begin{equation}
		\norm{v_{J+1}}_0\leq M(J)+2^6\delta_{q}^{1/2}\leq M(J)+2^{6-q}=M(J+1)
		\label{cota C0 vJ+1}
	\end{equation}
	by definition (\ref{def M(J)}). Hence, (\ref{inductive j cota C0}) holds for $J+1$, too. Next, we estimate the $C^1$-norm:
	\begin{align*}
		\norm{v_{J+1}}_1&\leq \norm{v_J}_1+\norm{v_{J+1}-v_J}_1\leq 2^6\delta_{q^\ast}^{1/2}\lambda_J+50\delta_{q}^{1/2}\lambda_{J+1}\\&\stackrel{(\ref{otra relación entre las lambdas})}{\leq} 2^6\delta_{q}^{1/2}\lambda_{J+1}^{1-2\beta}+50\delta_q^{1/2}\lambda_{J+1}\leq 2^6\delta_q^{1/2}\lambda_{J+1}
	\end{align*}
	for sufficiently large $a>1$. It follows from (\ref{def q*}) that $q^\ast(J+1)=q(J)$, so (\ref{derivative vqj}) holds for $J+1$. 
\end{proof}
\begin{remark}
	If $\gamma$ is smaller, we obtain better bounds:
	\begin{equation}
		\norm{v_{J+1}-v_J}_0\leq 8\max_{m\in\Lambda}\norm{b_m}_0+\norm{w_c}_0\leq16\norm{\gamma}_0+\lambda_{J+1}^{-\beta}\stackrel{(\ref{aux lambda negativa})}{\leq} 16\norm{\gamma}_0+\delta_{q+4}.
		\label{cota mejor si gamma pequeña}
	\end{equation}
\end{remark} 

Next, we check two conditions that are deeply related to the geometry of the construction:
\begin{lemma}
	If $\alpha>0$ is sufficiently small and $a>1$ is sufficiently large (depending on $\alpha$), the new velocity $v_{J+1}$ satisfies (\ref{crecimiento vJ}) and (\ref{perturbation almost perpendicular}).
\end{lemma}
\begin{proof}
	Let us write:
	\[\norm{[v_{J+1}-v_{q(J)}]\cdot \frac{v_{q(J)}}{\abs{v_{q(J)}}}}_0\leq \norm{w_0\cdot \frac{v_{q(J)}}{\abs{v_{q(J)}}}}_0+\norm{w_c}_0+\norm{[v_J-v_{q(J)}]\cdot \frac{v_{q(J)}}{\abs{v_{q(J)}}}}_0.\]
	Due to (\ref{perpendicularidad o tamaño}), we have
	\[\norm{w_0\cdot \frac{v_{q(J)}}{\abs{v_{q(J)}}}}_0\leq 8\max_{m\in\Lambda}\norm{\sqrt{2}\chi_m\gamma\,\zeta\cdot\frac{v_{q}}{\abs{v_{q}}}}_0\leq 2^{7/2}\delta_{q+3}^{1/2}.\]
	Combining this with (\ref{estimates wc}) and the inductive hypothesis (\ref{perturbation almost perpendicular}) for $v_J$ leads to
	\begin{align*}
		\norm{[v_{J+1}-v_{q(J)}]\cdot \frac{v_{q(J)}}{\abs{v_{q(J)}}}}_0&\leq 2^{7/2}\delta_{q+3}^{1/2}+\lambda_{J+1}^{-\beta}+16[8j(J)+l(J)]\delta_{q+3}^{1/2}\\&\stackrel{(\ref{aux lambda negativa})}{\leq} 16[8j(J)+l(J)+1]\delta_{q+3}^{1/2}\leq 2^{10}\delta_{q+3}^{1/2}.
	\end{align*}
	If $(j(J),l(J))\neq (7,7)$, then $q(J+1)=q(J)$, so the previous inequality implies (\ref{perturbation almost perpendicular}) for $J+1$. If $j(J)=l(J)=7$, then $v_{q(J+1)}=v_{J+1}$ (in $U$, see \cref{remark igual si varias disjuntas}), so (\ref{perturbation almost perpendicular}) holds trivially. 
	
	Next, we prove (\ref{crecimiento vJ}). If $a>1$ is sufficiently large, in $\Omega_{q+2}$ we have
	\begin{align*}
		\abs{v_{J+1}}&\geq \abs{v_{J+1}\cdot\frac{v_{q(J)}}{\abs{v_{q(J)}}}} \geq \abs{v_{q(J)}}- \norm{[v_{J+1}-v_{q(J)}]\cdot \frac{v_{q(J)}}{\abs{v_{q(J)}}}}_0\\&\geq 2^{-1}\delta_{q+2}^{1/2}-2^{10}\delta_{q+3}^{1/2} 
	\end{align*}
	by (\ref{inductive J igual fuera de omegaq+2}) and (\ref{crecimiento vJ}). Note that
	\[\frac{\delta_{q+2}}{\delta_{q+3}}=\left(\frac{\lambda_{q+3}}{\lambda_{q+2}}\right)^{4\alpha}=a^{4\alpha(b^{q+3}-b^{q+2})}\geq a^{4\alpha(b-1)},\]
	which can be made arbitrarily small by taking $a>1$ sufficiently large (not depending on $q$). Hence, we see that (\ref{crecimiento vJ}) holds for $J+1$.
\end{proof}

We finish studying the $H^{-1}$-norm of the perturbation:
\begin{lemma}
	If $\alpha>0$ is sufficiently small and $a>1$ is sufficiently large (depending on $\alpha$), there exist $A\in C^\infty(\TT^3,\RR^3)$ and $B\in C^\infty(\TT^3,\RR^{3\times 3})$ such that (\ref{cotas A y B}), (\ref{integral norma H-1}) and (\ref{cota H-1 J}) hold.
\end{lemma}
\begin{proof}
	Since $\nabla\theta_m$ is a constant vector for any $m\in \Lambda$, we can write
	\[w_0=\sum_{m\in \Lambda}\left\{\Div\left[\frac{b_m}{\abs{\theta_m}^2}(\zeta\otimes\nabla\theta_m)\sin(\theta_m)\right]-\frac{\nabla\theta_m\cdot\nabla b_m}{\abs{\nabla\theta_m}^2}\,\zeta \sin(\theta_m)\right\}.\]
	Defining
	\begin{align*}
		A&\coloneqq w_c-\sum_{m\in\Lambda}\frac{\nabla\theta_m\cdot\nabla b_m}{\abs{\nabla\theta_m}^2}\,\zeta \sin(\theta_m), \\ B&\coloneqq -\sum_{m\in\Lambda}\frac{b_m}{\abs{\theta_m}^2}(\zeta\otimes\nabla\theta_m)\sin(\theta_m),
	\end{align*}
	we have
	\[\int_{\TT^3}f\cdot(v_{J+1}-v_J)=\int_{\TT^3}f\cdot(w_0+w_c)=\int_{\TT^3}f\cdot[A-\Div B]=\int_{\TT^3}(f\cdot A+Df:B),\]
	so (\ref{integral norma H-1}) holds. Let us check (\ref{cotas A y B}). By (\ref{cota bm CN}) and (\ref{estimate wc 2beta}), we have
	\[\norm{A}_0\lesssim \norm{w_c}_0+\max_{m\in \Lambda} \lambda_{J+1}^{-1}\norm{b_m}_1\lesssim \lambda_{J+1}^{-2\beta}+\lambda_{J+1}^{-1}\delta_q^{1/2}\mu\lesssim \lambda_{J+1}^{-2\beta}\stackrel{(\ref{aux lambda negativa})}{\leq} \delta_{q+1}^{3000}\lambda_{J+1}^{-\beta}.\]
	Regarding $B$, by (\ref{cota bm CN}) we have
	\[\norm{B}_0\lesssim \max_{m\in \Lambda} \lambda_{J+1}^{-1}\norm{b_m}_0\lesssim \lambda_{J+1}^{-1}\stackrel{(\ref{otra relación entre las lambdas})}{\leq}\lambda_{J}^{-1}\lambda_{J+1}^{-2\beta}\stackrel{(\ref{aux lambda negativa})}{\leq}\delta_{q+1}^{3000}\lambda_q^{-1}\lambda_{J+1}^{-\beta}.\]
	Choosing $a>1$ sufficiently large to compensate the numerical constants, we conclude $(\ref{cotas A y B})$. Finally, to deduce (\ref{cota H-1 J}), it suffices to use the following characterization of the $H^{-1}$-norm:
	\[\norm{u}_{H^{-1}(\TT^3)}=\inf\left\{\norm{A}_{L^2(\TT^3)}+\norm{B}_{L^2(\TT^3)}: \langle u, f \rangle = \int_{\TT^3}(f\cdot A+Df:B) \quad \forall f\in H^1(\TT^3) \right\},\]
	where $\langle \cdot, \cdot\rangle$ is the $H^{-1}$-pairing. A proof (of the scalar-valued version) of this characterization can be found in \cite[Chapter 5]{Evans}. The proof is for bounded domains, but it works on the torus, as well.
\end{proof}

In conclusion, the new velocity $v_{J+1}$ has all the desired properties. 

\section{Estimates on the Reynolds stress} \label{section estimates R}
We begin by constructing and estimating $M_1$ and $M_2$. First of all, taking into account that $\zeta\cdot\nabla\theta_m=0$ for any $m\in \Lambda$, we may write the function $\rho_1$ defined in (\ref{def rho1}) as
\begin{align*}
	\rho_1=&\sum_{m\neq m'}b_{m'}(\zeta\cdot \nabla b_m)\, \zeta\,[\cos(\theta_m-\theta_{m'})+\cos(\theta_m+\theta_{m'})] \\ &+\sum_{m\in \Lambda}b_m(\zeta\cdot\nabla b_m)\,\zeta\cos(2\theta_m).
\end{align*}
Since we also have $k_{J,m}\cdot u_{J,m}$ for any $m\in \Lambda$, we may write the function $\rho_2$ defined in (\ref{def rho2}) as
\[\rho_2=\sum_{m\in \Lambda}[(\zeta\cdot\nabla b_m)u_{J,m}+(u_{J,m}\cdot\nabla b_m)\zeta]\cos(\theta_m)-\sum_{m\in \Lambda}l_m\eta\, b_m(\zeta\cdot\xi_{J,m})\zeta \sin(\theta_m).\]
It follows from (\ref{cota bm CN}) that
\[\norm{\rho_1}_0+\lambda_{J+1}^{-1}\norm{\rho_1}_1+\norm{\rho_2}_0+\lambda_{J+1}^{-1}\norm{\rho_2}_1\lesssim \delta_{q}^{1/2}\eta.\]
By interpolation, we have
\begin{equation}
	\norm{\rho_1}_\alpha+\norm{\rho_2}_\alpha\lesssim \delta_{q}^{1/2}\eta\lambda_{J+1}^{\alpha}\stackrel{(\ref{def eta})}{=}\delta_{q}^{1/2}\delta_{q+4}\lambda_{J+1}^{1-\alpha}\leq \delta_{q+4}\lambda_{J+1}^{1-\alpha}.
	\label{norma beta rhos}
\end{equation}

On the other hand, by \cref{stationary phase lemma} we have
\[\norm{\rho_1}_{B^{-1+\alpha}_{\infty,\infty}}+\norm{\rho_1}_{B^{-1+\alpha}_{\infty,\infty}}\lesssim \frac{\delta_{q}^{1/2}\eta}{\lambda_{J+1}^{1-\alpha}}\left(1+\frac{\mu^N}{\lambda_{J+1}^{N-1}}\right)\]
for any $N\in \NN$. Choosing $N=(4\beta)^{-1}$ so that the term in parenthesis equals $2$ and using (\ref{def eta}) yields
\begin{equation}
	\norm{\rho_1}_{B^{-1+\beta}_{\infty,\infty}}+\norm{\rho_1}_{B^{-1+\beta}_{\infty,\infty}}\lesssim\delta_{q}^{1/2}\delta_{q+4}\lambda_{J+1}^{-\alpha}\leq \delta_{q+4}\lambda_{J+1}^{-\alpha}.
	\label{norma besov rhos}
\end{equation}
Since $f_1$ and $f_2$ are the divergence of a smooth symmetric matrix whose support is contained in $U$, it follows from (\ref{identidad de Green matrices}) that
\[\int_U f_i\cdot \xi=0 \qquad \forall \xi\in \ker D^s\]
for $i=1,2$. Therefore, by \cref{invertir divergencia matrices} there exist $M_1, M_2 \in C^\infty_c(U,\mathcal{S}^3)$ such that $\Div M_1=f_1$ and $\Div M_2=f_2$. We can then define $E_{J+1}$ as (\ref{def EJ+1}) and $R_{J+1}$ as (\ref{def RJ+1}). 

We have completed the definition of the new subsolution $(v_{J+1},p_{J+1},R_{J+1})$ and it only remains to estimate $E_{J+1}$. By \cref{invertir divergencia matrices}, $M_1$ and $M_2$ can be chosen so that
\begin{align}
	\norm{M_1}_0+\norm{M_2}_0&\lesssim \norm{\rho_1}_{B^{-1+\alpha}_{\infty\infty}}+\norm{\rho_2}_{B^{-1+\alpha}_{\infty\infty}}\lesssim \delta_{q+4}\lambda_{J+1}^{-\alpha},  \label{estimaciones R primera}\\
	\norm{M_1}_1+\norm{M_2}_1&\lesssim \norm{\rho_1}_\alpha+\norm{\rho_2}_\alpha\lesssim \delta_{q+4}\lambda_{J+1}^{1-\alpha},
\end{align}
where we have used (\ref{norma beta rhos}) and (\ref{norma besov rhos}). Note that the implicit constants depend on $U$. However, in \cref{section prueba prop steps} we will see that one can take $U=s U_1+p$, where $U_1$ is a fixed smooth domain, $s\in (0,1)$ and $p\in\RR^3$. Therefore,  \cref{invertir divergencia vectores} ensures that the constants remain uniformly bounded as we shrink the domain in subsequent steps, by letting $s$ become small. 

Next, we estimate the matrix $M_3$, which was defined in (\ref{def M3}). Since the support of each $\chi_m$ (and hence $b_m$) is contained in a cube of side $2\mu^{-1}$, we see that
\begin{equation}
	\norm{M_3}_0\lesssim \max_{m\in\Lambda}\norm{b_m}_0\norm{v_J}_1\mu^{-1}\lesssim \delta_q^{1/2}\delta_{q^\ast}^{1/2} \frac{\lambda_J}{\mu}\stackrel{(\ref{aux rel parámetros})}{\leq} \delta_{q+4}\lambda_{J+1}^{-\beta}.
\end{equation}
Regarding the $C^1$-norm, the dominating term is the derivative of the trigonometric function, as usual:
\begin{align}
	\begin{split}
		\norm{M_3}_1&\lesssim \max_{m\in\Lambda}\left(\norm{b_m}_1\norm{v_J}_1\mu^{-1}+\norm{b_m}_0\norm{v_J}_1+\norm{b_m}_0\norm{v_J}_1\mu^{-1}\lambda_{J+1}\right) \\ &\lesssim \delta_q^{1/2}\delta_{q^\ast}^{1/2} \lambda_J\mu^{-1}\lambda_{J+1}\stackrel{(\ref{aux rel parámetros})}{\leq} \delta_{q+4}\lambda_{J+1}^{1-\beta}.
		\label{cota M3 C1}
	\end{split}
\end{align}

Finally, it follows from (\ref{estimates wc}) and the inductive hypotheses (\ref{inductive j cota C0}) and (\ref{derivative vqj}), which we have already proved for $J+1$, that
\begin{align}
	\norm{w_c\otimes v_{J+1}+v_{J+1}\otimes w_c-w_c\otimes w_c}_0&\lesssim\lambda_{J+1}^{-\beta} \stackrel{(\ref{aux lambda negativa})}{\leq} \delta_{q+4}\lambda_{J+1}^{-\alpha}, \\
	\norm{w_c\otimes v_{J+1}+v_{J+1}\otimes w_c-w_c\otimes w_c}_1&\lesssim\lambda_{J+1}^{1-\beta}\stackrel{(\ref{aux lambda negativa})}{\leq} \delta_{q+4}\lambda_{J+1}^{1-\alpha}. \label{estimaciones R final}
\end{align}
Combining (\ref{estimaciones R primera})-(\ref{estimaciones R final}), we obtain
\[\norm{E_{J+1}}_0+\lambda_{J+1}^{-1}\norm{E_{J+1}}_1\lesssim \delta_{q+4}\lambda_{J+1}^{-\alpha}.\]
Taking $a>1$ sufficiently large to compensate the numerical constants leads to
\[\norm{E_{J+1}}_0+\lambda_{J+1}^{-1}\norm{E_{J+1}}_1\leq 2^{-8}\delta_{q+4}.\]
Therefore, the diffeomorphism $\Phi_{J+1}$ and the subsolution $(v_{J+1},p_{J+1},R_{J+1})$ satisfy all the desired properties. This concludes the proof of \cref{prop stages}.

\section{Proof of Proposition 2.2} \label{section prueba prop steps}
\subsection{Preliminaries}
Since at the present iteration we only wish to perturb the field in $\Omega_{q+2}$, we fix a cutoff function $\rho_{q+1}\in C^\infty(\TT^3,[0,1])$ that equals 1 on the set
\[\left(\TT^3\backslash \Omega_{q+2}\right)+\overline{B}\left(0,\,\delta_{q+1}^{1/2}/12\hspace{0.5pt}\right)\]
and whose support is contained in
\[\left(\TT^3\backslash \Omega_{q+2}\right)+B\left(0,\,\delta_{q+1}^{1/2}/6\hspace{0.5pt}\right).\]
Hence, the support of $(1-\rho_{q+1})$ is contained in $\Omega_{q+2}$ and the distance to $\Sigma_{q+2}$ is, at least, $\delta_{q+1}^{1/2}/12$. In addition, $\rho_{q+1}$ vanishes in a neighborhood of $\overline{\Omega}_{q+1}$ for sufficiently large $a>1$, by (\ref{distance sigmas}). On the other hand, by \cref{lemma cutoff} we may assume that
\begin{equation}
	\norm{\sqrt{\rho_{q+1}}}_N+\norm{\sqrt{1-\rho_{q+1}}\,}_N\leq C_N\delta_{q+1}^{-N/2} \qquad \forall N\geq 0
	\label{cotas rhoq+1}
\end{equation}
for some constants $C_N$ depending only on $N$.

To write the Reynolds stress in a form that we can correct, we will exploit the degree of freedom in the definition of subsolution. We define
\begin{align*}
	\widetilde{p}_q&\coloneqq p_q-2r\delta_{q+2}(1-\rho_{q+1}), \\
	\widetilde{R}_q&\coloneqq R_q-2r\delta_{q+2}(1-\rho_{q+1})\Id,
\end{align*}
where $r>0$ is the radius in \cref{geometric lemma}. We see that $(v_q,\widetilde{p}_q,\widetilde{R}_q)$ is also a subsolution. By (\ref{descomposición Rq}), we may write
\begin{equation}
	\label{descomposición Rtilde}
	\widetilde{R}_q=\rho_{q+1}R_0+\rho_q(1-\rho_{q+1}) R_0-2r\delta_{q+3}(1-\rho_{q+1})\left(\Id-\frac{S_q}{2r\delta_{q+3}}\right).
\end{equation}
Regarding each of the terms:
\begin{itemize}
	\item According to the inductive hypothesis (\ref{def Omegaq}), the first term can be ignored until the iteration $q+1$. Note that $\norm{\rho_{q+1}R_0}_0\leq \delta_{q+1}/4$ because of (\ref{crecimiento R0}) and the fact that $\rho_{q+1}$ vanishes on $\Omega_{q+1}$.
	\item The second term is large, of order $\delta_q$, but $v_0$ is in its kernel, by (\ref{def R0}). This will allow us to partially cancel this error while keeping control of the geometry.
	\item The third term does not have any orthogonality property, but it is very small, by (\ref{inductive decomposition error}). Thus, we will be able to partially cancel it while keeping control of the geometry.
\end{itemize}

\subsection{First stage} \label{first stage}
Let us focus on the second term in the decomposition (\ref{descomposición Rtilde}). To express $R_0$ in a suitable manner and to carry out the construction, we will need an orthogonal basis adapted to the field $v_q$. However, in general there does not exist a nonvanishing vector field $\zeta_0\in C^\infty(\Omega_{q+2},\RR^3)$ that is orthogonal to $v_q$, due to topological obstructions. Nevertheless, we can work locally.

We will restrict ourselves to small regions by means of cutoffs. We fix a smooth cutoff function $\sigma\in C^\infty_c\big(\left(-\frac{3}{4},\frac{3}{4}\right)^3\big)$ such that 
\[\sum_{m\in \ZZ^3}\sigma(x-m)^2=1.\]
For $m\in \ZZ^3$ we define $\sigma_m$ as the periodic extension of 
\[\sigma_m(x)\coloneqq \sigma(\ell^{-1}x-m),\]
where
\begin{equation}
	\ell\coloneqq \floor*{2^{4}\delta_{q+4}^{-1}\delta_{q-1}^{1/2} \lambda_q}^{-1}.
	\label{def ell}
\end{equation}
Note that $\delta_{q-1}^{1/2}\lambda_q\geq \delta_q^{1/2}\lambda_q\geq 1$. We also define the cubes $Q\coloneqq (-\frac34,\frac34)^3$ and
\[Q_m\coloneqq \ell m+\ell Q.\]
We choose a convex open set $Q\subset U_1\subset Q+B(0,\frac14)$ with smooth boundary and such that $\dist(Q,\partial U_1)\geq \frac18$ and then set
\[U_m\coloneqq \ell U_1+\ell m,\]
so that $Q_m\subset U_m\subset Q_m+ B(0,\ell/4)$.

By construction, $\supp \sigma_m\subset Q_m\subset U_m$. Since $\ell\leq \delta_{q+4}\ll \delta_{q+1}^{1/2}$ for sufficiently large $a$, we can find a subset $\Lambda\subset \ZZ^3$ such that
\[\sum_{m\in \Lambda}\sigma_m^2\equiv 1\qquad \text{on }\supp (1-\rho_{q+1})\]
and the closure of $U_m$ is contained in $\Omega_{q+2}$ for any $m\in \Lambda$. For each $m\in \Lambda$ we choose two unitary vectors $\zeta_{0,m}$ and $\zeta_{1,m}$ such that
\[\{v_q(\ell m)\,,\; \zeta_{0,m}\,,\; \zeta_{1,m}\}\]
is an orthogonal reference frame. We will check that we can control the angle at other points:
\begin{lemma}
	If $a>1$ is sufficiently large, for $i\in\{0,1\}$, any $m\in \Lambda$ and any $x\in U_m$ we have
	\begin{align}
		\abs{v_q(x)\times \zeta_{i,m}}&\geq 2^{-2}\delta_{q+2}^{1/2}, \label{angulo con zetaim} \\ \abs{\zeta_{i,m}\cdot \frac{v_q(x)}{\abs{v_q(x)}}}&\leq 2^4\delta_{q+4}\delta_{q+2}^{-1/2}. \label{perpendicularidad con zetaim}
	\end{align}
\end{lemma}
\begin{proof}
	First of all, it follows from (\ref{def Omegaq}) that 
	\[\abs{v_0(x)}\geq 2^{-1}\delta_{q+2}^{1/2} \qquad \forall x\in \Omega_{q+2}.\]
	Combining this with (\ref{inductive support perturbation}) and (\ref{inductive growth}) yields
	\begin{equation}
		\abs{v_q(x)}\geq 2^{-1}\delta_{q+2}^{1/2} \qquad \forall x\in \Omega_{q+2}
		\label{tamaño vq en Omegaq+2}
	\end{equation}
	for sufficiently large $a>1$, by (\ref{cociente deltas}). Hence, for any $m\in \Lambda$ and $i=0,1$ we have
	\[\abs{v_q(\ell m)\times \zeta_{i,m}}=\abs{v_q(\ell m)}\geq 2^{-1}\delta_{q+2}^{1/2}\]
	because $\ell m\in \Omega_{q+2}$. On the other hand, for $x\in U_m$ we have
	\[\abs{v_q(x)-v_q(\ell m)}\leq \norm{v_q}_1\,\sqrt{3}\hspace{1pt}\ell\leq 2^3\delta_{q+4} \leq 2^{-2}\delta_{q+2}^{1/2}\]
	for sufficiently large $a>1$. We conclude
	\[\abs{v_q(x)\times \zeta_{i,m}}\geq \abs{v_q(\ell m)\times \zeta_{i,m}}-\abs{v_q(x)-v_q(\ell m)}\geq 2^{-2}\delta_{q+2}^{1/2}\]
	for $i\in\{0,1\}$, any $m\in \Lambda$ and any $x\in U_m$. Also, since $\zeta_{i,m}$ is perpendicular to $v_q(\ell m)$, we have
	\[\abs{\zeta_{i,m}\cdot \frac{v_q(x)}{\abs{v_q(x)}}}\leq \abs{\frac{v_q(x)}{\abs{v_q(x)}}-\frac{v_q(\ell m)}{\abs{v_q(\ell m)}}} \leq \frac{\norm{v_q}_1\sqrt{3}\,\ell}{\inf_{x\in U_m}\abs{v_q(x)}}\leq 2^4\delta_{q+4}\delta_{q+2}^{-1/2}\]
	by (\ref{inductive derivative}), (\ref{tamaño vq en Omegaq+2}) and the definition (\ref{def ell}).
\end{proof}

Concerning the matrix $R_0$, given $m\in \Lambda$, we use the fact that $\{v_q(\ell m)\,,\; \zeta_{0,m}\,,\; \zeta_{1,m}\}$ is an orthogonal basis to write
\begin{equation}
	R_0(\ell m)=-\abs{v_0(\ell m)}^2\left(\zeta_{0,m}\otimes\zeta_{0,m}+\zeta_{1,m}\otimes\zeta_{1,m}\right).
	\label{R0 en l-1m}
\end{equation}
On the other hand, for $x\in U_m$ we have
\begin{equation}
	\label{diferencia R0 con valor cte}
	\abs{R_0(x)-R_0(\ell m)}\leq \norm{R_0}_1\,\sqrt{3}\ell\stackrel{(\ref{assumption tamaño v0})}{\leq} 2\ell\leq \frac{1}{8}\delta_{q+4}\,.
\end{equation}
Hence, to a good approximation, we may write $R_0$ using $\zeta_{0,m}$ and $\zeta_{1,m}$. 

In summary, the use of cutoffs allows us to define an adapted orthonormal basis in each open cube $U_m$ and we can use it to represent $R_0$. We have all the ingredients to start correcting the error. Unfortunately, we cannot perturb the field in all the cubes $U_m$ at once because the intersections would be problematic. What we will do is to decompose $\Lambda$ into 8 subsets $\Lambda_0, \dots, \Lambda_7$ such that for any $l=0, \dots, 7$ we have $U_m\cap U_{m'}=\varnothing$ for any two distinct $m,m'\in \Lambda_l$. 

Let us start correcting the field. For $m\in\Lambda$ we define 
\[\gamma_{m}\coloneqq \left[\rho_q(1-\rho_{q+1})\right]^{1/2}\sigma_m\abs{v_0(\ell m)}.\]
Using (\ref{cotas rhoq}), (\ref{cotas rhoq+1}), the fact that $\rho_q$ vanishes on $\Omega_q$ and the definition of $\sigma_m$, we have:
\begin{align*}
	\norm{\gamma_m}_0&\leq \delta_q^{1/2},  \\
	\norm{\gamma_m}_N&\leq C_N\delta_q^{1/2}\ell^{-N} \qquad \forall N\geq 0,
\end{align*}
where $C_N$ are universal constants. We fix the indices $J_0=(q,0,0)$ and we set $\Phi_{J_0}\equiv \Phi_q$ and $(v_{J_0},p_{J_0},R_{J_0})\equiv (v_q,\widetilde{p}_q, \widetilde{R}_q)$. Note that the inductive hypotheses (\ref{inductive support perturbation diffeo})-(\ref{inductive derivative}) for $q$ imply the inductive hypotheses (\ref{support perturbation diffeo J})-(\ref{derivative vqj}) for $J_0$, while (\ref{perturbation almost perpendicular}) is trivially satisfied. 

Starting with this field, we apply \cref{prop stages} iteratively, using the collections 
\[\{Q_m,\;U_m,\;\gamma_m,\; \zeta_{0,m}\,\}_{m\in\Lambda_{l(J)}}\]
as discussed in \cref{remark igual si varias disjuntas}. To do so, we must check that (\ref{stages tamaño ell})-(\ref{angulo vJ enunciado prop}) are satisfied. It follows from definition (\ref{def ell}) that condition (\ref{stages tamaño ell}) holds for $J_0$. Since the lower bound for $\ell$ decreases when we increase $J$, due to (\ref{aux rel parámetros}), we see that (\ref{stages tamaño ell}) is also satisfied for $J>J_0$. Then, (\ref{distance Q,U})-(\ref{cota gamma stages CN}) hold by construction of $Q_m,U_m$ and $\gamma_m$. In addition, (\ref{perpendicularidad o tamaño}) follows from (\ref{perpendicularidad con zetaim}) for sufficiently large $a>1$.

By (\ref{angulo con zetaim}), condition (\ref{angulo vJ enunciado prop}) is satisfied for $J_0$. We must check that it holds throughout the whole process:
\begin{lemma}
	\label{lema producto con zeta}
	Let $J\in\mathcal{J}$ such that $J_0\leq J\leq J_0+7$ and suppose that (\ref{angulo vJ enunciado prop}) holds for $J_0\leq I<J$ so that we can, indeed, construct $v_{J_0+1}, \dots , v_J$ satisfying (\ref{support perturbation diffeo J})-(\ref{perturbation almost perpendicular}). Then, if $a>1$ is sufficiently large, we have
	\begin{equation}
		\abs{v_J(x)\times \zeta_{i,m}}\geq 2^{-3}\delta_{q+2}^{1/2}
		\label{vJ times zetaim}
	\end{equation}
	for $i\in\{0,1\}$, any $m\in \Lambda$ and any $x\in U_m$.
\end{lemma}
\begin{proof}
	Fix $m\in\Lambda$. Consider the fields $\widetilde{\zeta}_{i,m}\in C^\infty(U_m)$ given by
	\[\widetilde{\zeta}_{0,m}(x)\coloneqq \frac{v_q(x)\times[\zeta_{0,m}\times v_q(x)]}{\abs{v_q(x)\times[\zeta_{0,m}\times v_q(x)]}}, \qquad \widetilde{\zeta}_{1,m}(x)\coloneqq \frac{v_q(x)}{\abs{v_q(x)}}\times \widetilde{\zeta}_{0,m}(x).\]
	We have
	\begin{align*}
		\|\widetilde{\zeta}_{0,m}\|_1&\lesssim\frac{\norm{v_q}_1}{\inf_{x\in U_m}\abs{v_q(x)\times \zeta_{i,m}}}\lesssim \delta_{q+2}^{-1/2}\delta_{q-1}^{1/2}\lambda_q, \\
		\|\widetilde{\zeta}_{1,m}\|_1&\leq \|\widetilde{\zeta}_{0,m}\|_1+\frac{\norm{v_q}_1}{\inf_{x\in U_m}\abs{v_q(x)}}\lesssim \delta_{q+2}^{-1/2}\delta_{q-1}^{1/2}\lambda_q.
	\end{align*}
	Since $\{v_q,\widetilde{\zeta}_{0,m},\widetilde{\zeta}_{1,m}\}$ is an orthogonal basis of $\RR^3$, we may write
	\[v_J(x)=c_1(x)v_q(x)+c_2(x) \widetilde{\zeta}_{0,m}+c_3(x)\widetilde{\zeta}_{1,m}\]
	for certain functions $c_1, c_2, c_3$. Note that
	\[v_J\times \widetilde{\zeta}_{0,m}=c_1 v_q\times \widetilde{\zeta}_{0,m}+c_3\,\widetilde{\zeta}_{1,m}\times \widetilde{\zeta}_{0,m}=c_1\abs{v_q}\widetilde{\zeta}_{1,m}-c_3\frac{v_q}{\abs{v_q}}.\]
	Since both terms are perpendicular, we have
	\[\abs{v_J(x)\times \widetilde{\zeta}_{0,m}(x)}\geq \abs{c_1(x)}\abs{v_q(x)}.\]
	It follows from (\ref{perturbation almost perpendicular}) that
	\[\abs{c_1(x)-1}\lesssim \delta_{q+3}^{1/2}.\]
	In particular, it can be made arbitrarily small by choosing $a>1$ sufficiently large. 
	
	Combining this with (\ref{tamaño vq en Omegaq+2}), we obtain
	\[\abs{v_J(x)\times \widetilde{\zeta}_{0,m}(x)}\geq 2^{-2}\delta_{q+2}^{1/2} \qquad \forall x\in U_m.\]
	On the other hand, for any $x\in U_m$ we have
	\[\abs{\widetilde{\zeta}_{i,m}(x)-\zeta_{i,m}}\leq \|\widetilde{\zeta}_{i,m}\|_1 2\ell \lesssim \delta_{q+2}^{-1/2}\delta_{q+4}\leq \delta_{q+3}^{1/2},\]
	because $\widetilde{\zeta}_{i,m}(\ell m)=\zeta_{i,m}$. Taking into account that
	\[\abs{v_J(x)\times \zeta_{0,m}}\geq \abs{v_J(x)\times \widetilde{\zeta}_{0,m}(x)}-\abs{v_J(x)}\abs{\widetilde{\zeta}_{0,m}(x)-\zeta_{0,m}},\]
	we conclude (\ref{vJ times zetaim}) for $i=0$ if $a>1$ is sufficiently large. The case $i=1$ is completely analogous.
\end{proof}

Therefore, the hypotheses of \cref{prop stages} will be satisfied. After applying it eight times, we obtain a subsolution $(v_{J_1},p_{J_1},R_{J_1})$ with $J_1=(q,1,0)$ and a volume-preserving diffeomorphism $\Phi_{J_1}$ such that (\ref{support perturbation diffeo J})-(\ref{perturbation almost perpendicular}) hold. 

\subsection{Second stage}
Starting with this field, we apply \cref{prop stages} iteratively, using the collections 
\[\{Q_m,\;U_m,\;\gamma_m,\; \zeta_{1,m}\,\}_{m\in\Lambda_{l(J)}}\]
as discussed in \cref{remark igual si varias disjuntas}. To do so, we must check that (\ref{stages tamaño ell})-(\ref{angulo vJ enunciado prop}) are satisfied througout the process. Conditions (\ref{stages tamaño ell})-(\ref{cota gamma stages CN}) hold by construction of $Q_m,U_m,\ell$, $\gamma_m$ and $\zeta_{0,m}$. Condition (\ref{perpendicularidad o tamaño}) follows from (\ref{perpendicularidad con zetaim}). Finally, arguing as in \cref{lema producto con zeta} yields (\ref{angulo vJ enunciado prop}).

Hence, the hypotheses of \cref{prop stages} will be satisfied. After applying it eight times, we obtain a subsolution $(v_{J_2},p_{J_2},R_{J_2})$ with $J_2=(q,2,0)$ and a volume-preserving diffeomorphism $\Phi_{J_2}$ such that (\ref{support perturbation diffeo J})-(\ref{perturbation almost perpendicular}) hold. By the definition of the matrices $E_{J+1}$ that appear in \cref{prop stages}, we have
\begin{align}
	\begin{split}
		\label{RJ2}
		R_{J_2}&=\widetilde{R}_q+\sum_{J=J_0}^{J_1-1}\left(E_{J+1}+\gamma_m^2\zeta_{0,m}\otimes\zeta_{0,m}\right)+\sum_{J=J_1}^{J_2-1}\left(E_{J+1}+\gamma_m^2\zeta_{1,m}\otimes\zeta_{1,m}\right) \\
		&= \widetilde{R}_q+\sum_{l=1}^7\sum_{m\in \Lambda_l}\gamma_m^2\left(\zeta_{0,m}\otimes\zeta_{0,m}+\zeta_{1,m}\otimes\zeta_{1,m}\right)+\sum_{J=J_0}^{J_2-1}E_{J+1} \\
		&= \widetilde{R}_q+\sum_{m\in\Lambda} \rho_q(1-\rho_{q+1})\sigma_m^2\abs{v_0(\ell m)}^2\left(\zeta_{0,m}\otimes\zeta_{0,m}+\zeta_{1,m}\otimes\zeta_{1,m}\right)+\sum_{J=J_0}^{J_2-1}E_{J+1} \\
		&= \rho_{q+1}R_0-2r\delta_{q+3}(1-\rho_{q+1})\left(\Id-\frac{S_q}{2r\delta_{q+3}}\right)\\&\hspace{10pt}+\rho_q(1-\rho_{q+1})\sum_{m\in\Lambda}\sigma_m^2\left[R_0-R_0(\ell m)\right]+\sum_{J=J_0}^{J_2-1}E_{J+1},
	\end{split}
\end{align}
where we have also used (\ref{R0 en l-1m}) and the fact that $\sum_{m\in \Lambda}\sigma_m^2=1$ on the support of the cutoff $(1-\rho_{q+1})$. By (\ref{diferencia R0 con valor cte}), the first term in the last line is of order $\delta_{q+4}$, which means that we have canceled the second term in the decomposition (\ref{descomposición Rtilde}). 

\subsection{Remaining stages}
We will now focus on the third term in the decomposition (\ref{descomposición Rtilde}). We will need an adapted orthonormal frame to decompose the Reynolds stress, so we proceed as before. We set
\begin{equation}
	\tilde{\ell}\coloneqq \floor*{\delta_q^{1/2}\delta_{q+4}^{-1}\lambda_{J_2}}^{-1}
	\label{def elltilde}
\end{equation}
and for $m\in \ZZ^3$ we define
\begin{align*}
	\tilde{\sigma}_m(x)&\coloneqq \sigma\big(\tilde{\ell}^{-1}x-m\big), \\ \widetilde{Q}_m&\coloneqq \tilde{\ell}m+\tilde{\ell}Q, \\
	\widetilde{U}_m&\coloneqq \tilde{\ell}m+\tilde{\ell}U_1,
\end{align*}
where $Q$ and $U_1$ were defined at the beginning of \cref{first stage}. By construction, $\widetilde{Q}_m$ is a cube and $\widetilde{U}_m$ is a convex open set such that
\[\supp \tilde{\sigma}_m\subset \widetilde{Q}_m\subset \widetilde{U}_m\subset \widetilde{Q}_m+B(0,\tilde{\ell}/4)\]
and $\dist(\widetilde{Q}_m,\partial \widetilde{U}_m)\geq \tilde{\ell}/8$.  Since $\tilde{\ell}\leq \delta_{q+4}\ll\delta_{q+1}^{1/2}$ for sufficiently large $a$, we can find a subset $\tilde{\Lambda}\subset \ZZ^3$ such that
\[\sum_{m\in \tilde{\Lambda}}\tilde{\sigma}_m^2\equiv 1\qquad \text{on }\supp (1-\rho_{q+1})\]
and the closure of $\widetilde{U}_m$ is contained in $\Omega_{q+2}$ for any $m\in \tilde{\Lambda}$. For each $m\in \tilde{\Lambda}$ we choose two unitary vectors $\zeta_{2,m}$ and $\zeta_{3,m}$ such that
\[\left\{v_{J_2}\big(\tilde{\ell}m\big)\,,\; \zeta_{2,m}\,,\; \zeta_{3,m}\right\}\]
is an orthogonal reference frame. 

By \cref{geometric lemma}, there exist unitary vectors $\zeta_{4,m}, \dots, \zeta_{7,m}$ and smooth functions $\Gamma_{4}, \dots, \Gamma_7:C^\infty(B(\Id,r))\to \RR$ such that
\begin{equation}
	\Id-\frac{S_q(x)}{2r\delta_{q+3}}=\sum_{j=3}^7\Gamma_j\left(\Id-\frac{S_q(x)}{2r\delta_{q+3}}\right)^2\zeta_{j,m}\otimes\zeta_{j,m}
	\label{descomposición Sq}
\end{equation}
for any $m\in \tilde{\Lambda}$ and $x\in \widetilde{U}_m$. As mentioned in \cref{geometric lemma}, the vectors $\zeta_{2,m}, \dots, \zeta_{7,m}$ may be assumed to make an angle greater or equal than $45^\circ$ with $v_{J_2}(\tilde{\ell}m)$. By (\ref{crecimiento vJ}) and (\ref{derivative vqj}) we have
\[\abs{\frac{v_{J_2}(x)-v_{J_2}(\tilde{\ell}m)}{v_{J_2}(x)}}\leq\frac{\norm{v_{J_2}}_1\sqrt{3}\,\tilde{\ell}}{\inf_{x\in U_m}\abs{v_{J_2}(x)}}\lesssim \frac{\delta_{q+4}}{\delta_{q+2}^{1/2}}\leq \delta_{q+4}^{1/2}, \] 
which can be made arbitrarily small by taking $a$ sufficiently large. Hence, we may assume that the vectors $\zeta_{2,m}, \dots, \zeta_{7,m}$ make an angle greater or equal than $40^\circ$ with $v_{J_2}(x)$ for $x\in U_m$. By (\ref{crecimiento vJ}), we see that (\ref{angulo vJ enunciado prop}) holds. 

Next, we define
\[\gamma_{j,m}\coloneqq [2r\delta_{q+3}(1-\rho_{q+1})]^{1/2}\,\tilde{\sigma}_m \;\Gamma_j\left(\Id-\frac{S_q(\tilde{\ell}m)}{2r\delta_{q+3}}\right),\]
which satisfies
\begin{equation}
	\norm{\gamma_{j,m}}_N\lesssim \delta_{q+3}^{1/2}\tilde{\ell}^{-N}
	\label{tamaño gamma pequeñas}
\end{equation}
for $N\geq 0$ and some implicit constants depending on $N$. 

Starting with the volume-preserving diffeomorphism $\Phi_{J_2}$ and the subsolution $(v_{J_2},p_{J_2},R_{J_2})$, we apply \cref{prop stages} iteratively for $J_2\leq J\leq J_2+47$. By our previous definitions, (\ref{stages tamaño ell})-(\ref{perpendicularidad o tamaño}) hold. As we have discussed, (\ref{angulo vJ enunciado prop}) is also satisfied for $J_2$. Concerning the following fields, it follows from (\ref{crecimiento vJ}), (\ref{cota mejor si gamma pequeña}), (\ref{tamaño gamma pequeñas}) and the fact that the perturbation is supported on $\Omega_{q+2}$ that
\[\abs{\frac{v_{J+1}-v_J}{v_J}}\lesssim \left(\frac{\delta_{q+3}}{\delta_{q+2}}\right)^{1/2}=a^{-\alpha b^{q+2}(b-1)}\leq a^{-\alpha(b-1)}.\] 
Since this can be made arbitraly small by taking $a>1$ sufficiently large, we can ensure that the angle between $v_J$ and the vectors $\zeta_{2,m}, \dots, \zeta_{7,m}$ remains greater than $30^\circ$ throughout the process.

Therefore, we can apply \cref{prop stages} iteratively for $J_2\leq J\leq J_2+47$. After 48 applications, we obtain a subsolution $(v_{J_3},p_{J_3},R_{J_3})$ with $J_3=(q+1,0,0)$ and a volume-preserving diffeomorphism $\Phi_{J_3}$ such that (\ref{support perturbation diffeo J})-(\ref{perturbation almost perpendicular}) hold. Making the identification $J_3\equiv q+1$, we see that they imply (\ref{inductive support perturbation diffeo})-(\ref{inductive derivative}). 

On the other hand, it is clear that (\ref{cota H-1 q}) follows from (\ref{cota H-1 J}). In addition, by (\ref{cambio iteración j}) we have
\begin{align*}
	\norm{v_{q+1}-v_q}_0+\lambda_q^{-1}\norm{v_{q+1}-v_q}_1&\leq \sum_{J=J_0}^{J_3-1}\left(\norm{v_{J+1}-v_J}_0+\lambda_{q+1}^{-1}\norm{v_{J+1}-v_J}_1\right)\\&\leq \sum_{J=J_0}^{J_3-1}\left(\norm{v_{J+1}-v_J}_0+\lambda_{J+1}^{-1}\norm{v_{J+1}-v_J}_1\right) \\&\leq 2^{6+6}\delta_q^{1/2}.
\end{align*}
We conclude (\ref{conclusión 1}). Similarly, bounds (\ref{conclusión 2}) and (\ref{conclusión 3}) are easily seen to be implied by (\ref{conclusión stages 2}) and (\ref{conclusión stages 3}). 

Regarding the Reynolds stress:
\begin{align*}
	R_{q+1}&=R_{J_2}+\sum_{J=J_2}^{J_3-1}\left(E_{J+1}+\sum_{m\in\tilde{\Lambda}_{l(J)}}\gamma_{l(J),m}^2\zeta_{l(J),m}\otimes\zeta_{l(J),m}\right) \\
	&\stackrel{(\ref{descomposición Sq})}{=}R_{J_2}+\sum_{m\in\tilde{\Lambda}}2r\delta_{q+3}(1-\rho_{q+1})\tilde{\sigma}_m^2\hspace{-4pt}\left[\sum_{j=2}^7\Gamma_j\hspace{-3pt}\left(\Id-\frac{S_q(\tilde{\ell}m)}{2r\delta_{q+3}}\right)^{\hspace{-3pt}2}\hspace{-3pt}(\zeta_{j,m}\otimes\zeta_{j,m})\right]+\sum_{J=J_2}^{J_3-1}E_{J+1}\\
	&\stackrel{(\ref{RJ2})}{=}\rho_{q+1}R_0+S_{q+1},
\end{align*}
where
\[S_{q+1}\coloneqq \rho_q(1-\rho_{q+1})\sum_{m\in\Lambda}\sigma_m^2\left[R_0-R_0(\ell m)\right]+(1-\rho_{q+1})\sum_{m\in\tilde{\Lambda}}\tilde{\sigma}_m^2\left[S_q-S_q(\tilde{\ell} m)\right]+\sum_{J=J_0}^{J_3-1}E_{J+1}.\]
By (\ref{inductive decomposition error}) and the definition (\ref{def elltilde}), for any $m\in\tilde{\Lambda}$ and any $x\in \widetilde{U}_m$, we have
\[\abs{S_q(x)-S_q(\tilde{\ell}m)}\leq \norm{S_q}_1\sqrt{3}\hspace{2pt}\tilde{\ell}\lesssim \delta_{q+3}\lambda_q \frac{\delta_{q+4}}{\delta_q^{1/2}\lambda_{J_2}}\leq \delta_{q+3}^{1/2}\delta_{q+4},\]
so it can be smaller than $\delta_{q+4}/8$ for sufficiently large $a$. Combining this with (\ref{error J+1}) and (\ref{diferencia R0 con valor cte}), we conclude that \[\norm{S_{q+1}}_0\leq \frac{1}{2}\delta_{q+4}.\]

Concerning the $C^1$-norm, it follows from (\ref{aux rel parámetros}),(\ref{relación mu eta}) and (\ref{stages tamaño ell}) that
\[\lambda_q+\ell^{-1}+\widetilde{\ell}^{-1}\leq \delta_{q+4}\lambda_{q+1}\]
for sufficiently large $a>1$. Combining this bound with (\ref{assumption tamaño v0}), (\ref{inductive decomposition error}) and (\ref{error J+1}) leads to
\[\norm{S_{q+1}}_1\leq \frac{1}{2}\delta_{q+4}\lambda_{q+1}.\]
Hence, (\ref{inductive decomposition error}) is satisfied for $q+1$. On the other hand, since $\rho_{q+1}$ is supported in $\Omega_{q+1}$, it follows from (\ref{assumption tamaño v0}), (\ref{crecimiento R0}) and (\ref{cotas rhoq+1}) that for sufficiently large $a>1$ we have
\[\norm{\rho_{q+1}R_0}_0+\lambda_{q+1}^{-1}\norm{\rho_{q+1}R_0}_1\leq \frac{1}{2}\delta_{q+1},\]
so (\ref{inductive Rq C0}) and (\ref{inductive Rq C1}) hold for $q+1$.

In conclusion, the diffeomorphism $\Phi_{q+1}$, the new subsolution $(v_{q+1},p_{q+1},R_{q+1})$, the cutoff $\rho_{q+1}$ and the matrix $S_{q+1}$ have all of the properties claimed in \cref{prop steps}.

\section{Proof of Theorem 1.1}
\label{S.proof}
We choose a parameter $\alpha>0$ sufficiently small so that \cref{prop steps} holds. Since $\tau\in (\sqrt{2/3},1)$, we can further reduce $\alpha$ so that
\begin{align}
	12800\alpha+2\tau^{-2}-3&\leq -\frac{1}{2}(3-2\tau^{-2}), \label{cota gamma por abajo} \\
	\frac{1}{2}(1-\tau)&\leq (1-\tau)-7680\alpha. \label{relación gamma alpha}
\end{align}

Let $\Phi_0\coloneqq \Id$ and define $p_0$ and $R_0$ as in (\ref{def p0}) and (\ref{def R0}). As discussed in \cref{subsect inductive hyp}, $(v_0,p_0,R_0)$ is a subsolution and it satisfies the inductive hypotheses (\ref{inductive support perturbation diffeo})-(\ref{inductive decomposition error}). Therefore, we can apply \cref{prop steps} iteratively, obtaining a sequence of volume-preserving diffeomorphisms $\{\Phi_q\}_{q=0}^\infty$ and a sequence of subsolutions $\{(v_q,p_q,R_q)\}_{q=0}^\infty$ satisfying (\ref{inductive support perturbation diffeo})-(\ref{inductive Rq C1}) and (\ref{cota H-1 q})-(\ref{conclusión 3}).

By interpolation of (\ref{conclusión 2}) and (\ref{conclusión 3}), we have
\begin{align*}
	\norm{\Phi_{q+1}-\Phi_q}_\tau&\lesssim \norm{\Phi_{q+1}-\Phi_q}_0^{1-\tau}\norm{\Phi_{q+1}-\Phi_q}_1^{\tau}\leq 2^6(\delta_{q+1}^{-10}\lambda_q^{-1})^{1-\tau}\delta_{q+1}^{-1280\tau}\\&\leq 2^6\delta_{q+1}^{-1280}\lambda_q^{-(1-\tau)}=2^6a^{5120\alpha(b^{q+1}-1)-(1-\tau)(b^q-1)} \\ &=2^6a^{-[(1-\tau)-7680\alpha]b^q+[(1-\tau)-5120\alpha]}\stackrel{(\ref{relación gamma alpha})}{\leq} 2^6a^{1-\frac{1}{2}(1-\tau)b^q},
\end{align*}
where we have used that $b=\frac{3}{2}$. Since the term on the right-hand side defines a convergent series, we conclude that $\Phi_q$ must converge in $C^\tau(\TT^3)$ to some map $\Phi:\TT^3\to\TT^3$. By an analogous argument, $\Phi_q^{-1}$ converges in $C^\tau(\TT^3)$ to some map $\Psi:\TT^3\to\TT^3$. Passing to the limit in the identities
\[\Phi_q\circ\Phi_q^{-1}=\Phi_q^{-1}\circ\Phi_q=\Id\]
yields $\Psi=\Phi^{-1}$. Hence, $\Phi$ is an homeomorphism of $\TT^3$. Since it is the uniform limit of volume-preserving diffeomorphism, it is volume preserving.

Meanwhile, the fields $v_q$ are the pushforward of $v_0$ by the corresponding diffeomorphism $\Phi_q$, by (\ref{inductive pushforward}). In addition, interpolation of (\ref{conclusión 1}) yields:
\begin{align*}
	\norm{v_{q+1}-v_q}_\alpha&\lesssim \norm{v_{q+1}-v_q}_0^{1-\alpha}\norm{v_{q+1}-v_q}_1^\alpha\leq 2^{12}\delta_q^{1/2}\lambda_{q+1}^\alpha \\ &=2^{12}a^{-2\alpha(b^q-1)+\alpha(b^{q+1}-1)}=2^{12}a^{\alpha-\frac{1}{2}b^q},
\end{align*}
where we have used that $b=\frac{3}{2}$. Since the term on the right-hand side defines a convergent series, we conclude that $v_q$ must converge in $C^\alpha(\TT^3)$ to some map $v\in C^\alpha(\TT^3)$. Since $v$ is the uniform limit of divergence-free fields, it is weakly divergence free. On the other hand, multiplying the definition of subsolution (\ref{def subsolution}) by a divergence-free test-function $w\in C^\infty(\TT^3,\RR^3)$ and applying the divergence theorem yields
\[\int_{\TT^3}(v_q)_i(v_q)_j\partial_iw_j=\int_{\TT^3}(R_q)_{ij}\partial_iw_j.\]
Since the Reynolds stress $R_q$ converges uniformly to~$0$, we see that $v$ is a weak steady Euler flow.

Next, if follows from (\ref{diffeo-Id}) that $\norm{\Phi-\Id}_0$ can be made arbitrarily small by choosing $a>1$ sufficiently large. Concerning the $H^{-1}$-norm of the perturbation of the velocity, by (\ref{cota H-1 q}) we have
\begin{align*}
	\norm{v-v_0}_{H^{-1}(\TT^3)}&\leq \sum_{q=0}^\infty \norm{v_{q+1}-v_q}_{H^{-1}(\TT^3)}\leq 2^6\sum_{q=0}^\infty\delta_{q+1}=2^6\sum_{q=0}^\infty a^{-4\alpha(b^{q+1}-1)}\\&\leq 2^6\sum_{q=0}^\infty a^{-4\alpha\ln(b)(q+1)}\leq a^{-4\alpha\ln(b)}\sum_{q=0}^\infty2^{6-q}=2^7a^{-4\alpha\ln(b)}
\end{align*}
if $a>1$ is sufficiently large. We conclude that we can also make $\norm{v-v_0}_{H^{-1}}$ arbitrarily small by further increasing $a$.

Let us study $X_t\coloneqq \Phi \circ X^0_t \circ \Phi^{-1}$, where $X^0_t$ is the (unique) flow of $v_0$. First of all, $X_t\in C^0(\RR,C^{\tau^2}(\TT^3,\TT^3))$ because $\Phi$ and $\Phi^{-1}$ are in $C^\tau(\TT^3,\TT^3)$. We consider the maps $X_t^q\coloneqq \Phi_q \circ X^0_t \circ \Phi^{-1}_q$. Since $v_q$ is the pushforward of $v_0$ by $\Phi_q$, we have
\begin{align*}
	\partial_t X_t^q&=D\phi_q\circ(X_t^0\circ \Phi_q^{-1})\,(v_0\circ X^0_t)\circ\Phi_q^{-1}=(D\Phi_q\,v_0)\circ (X^0_t\circ\Phi_q^{-1})\\&=(D\Phi_q\,v_0)\circ\Phi_q^{-1}\circ X^q_t=v_q\circ X^q_t.
\end{align*}
Therefore, $X^q_t$ converges uniformly to $X_t$ and $\partial_t X^q_t$ converges uniformly to $v\circ X_t$. Since $C^1(\RR,C^0(\TT^3,\TT^3))$ is closed, we conclude that $X_t\in C^1(\RR,C^0(\TT^3,\TT^3))$ and 
\[\begin{cases}
	\partial_t X_t=v\circ X_t, \\
	X_0=\Id.
\end{cases}\]
On the other hand, it is clear that $X_t$ satisfies the group property because so does $X^0_t$. Finally, $X_t$ is volume-preserving for any $t\in \RR$ because so are $\Phi$, $\Phi^{-1}$ and $X^0_t$. Hence, $X_t$ is a flow of $v$.

Let us check that $X_t$ is the only flow with regularity $C^0_{\rm loc}(\RR,C^{\tau^2}(\TT^3,\TT^3))\cap C^1_{\rm loc}(\RR,C^{0}(\TT^3,\TT^3))$. Let $Y_t$ be a flow of $v$ (as defined in \cref{def flujos}) with this regularity. We define the maps
\begin{align*}
	\widetilde{X}^q_t&\coloneqq X^0_t\circ \Phi_q^{-1}, \\
	\widetilde{Y}^q_t&\coloneqq \Phi_q^{-1}\circ Y_t,
\end{align*}
which converge uniformly to 
\begin{align*}
	\widetilde{X}_t&\coloneqq X^0_t\circ \Phi^{-1}, \\
	\widetilde{Y}_t&\coloneqq \Phi^{-1}\circ Y_t,
\end{align*}
respectively. 

Our goal is to show that $\widetilde{Y}_t=\widetilde{X}_t$, from which we deduce that $Y_t=X_t$. Note that they are equal at $t=0$ because $X_0=Y_0=\Id$. Since $v_0$ is the pushforward of $v_q$ by $\Phi_q^{-1}$, we have
\begin{align*}
	\partial_t\widetilde{Y}^q_t&=(D\Phi_q^{-1}\circ Y_t)\,\partial_t Y_t=(D\Phi_q^{-1}\circ Y_t)\,(v\circ Y_t)\\&=(D\Phi_q^{-1}\,v_q)\circ Y_t+[D\Phi_q^{-1}(v-v_q)]\circ Y_t \\ &=v_0\circ\Phi_q^{-1}\circ Y_t+[D\Phi_q^{-1}(v-v_q)]\circ Y_t \\&= v_0\circ \widetilde{Y}^q_t+[D\Phi_q^{-1}(v-v_q)]\circ Y_t.
\end{align*}
Note that we also have $\partial_t \widetilde{X}^q_t=v_0\circ \widetilde{X}^q_t$, so
\[\partial_t \widetilde{Y}^q_t-\partial_t \widetilde{X}^q_t=v_0\circ\widetilde{Y}^q_t-v_0\circ \widetilde{X}^q_t+[D\Phi_q^{-1}(v-v_q)]\circ Y_t.\]

If the last term were not present, we could apply Grönwall's inequality to the rest, obtaining $\widetilde{X}^q=\widetilde{Y}^q$. Unfortunately, the last term is present and it can be quite large, by (\ref{derivative Phiq}). We can, however, expect some cancellations when integrating. We fix a ball $B(x_0,r)\subset \TT^3$ and we integrate in space and time, obtaining
\begin{align}
	\nonumber
	\int_{B(x_0,r)}(\widetilde{Y}^q_t-\widetilde{X}^q_t)\,dx=&\int_0^t\int_{B(x_0,r)}(v_0\circ\widetilde{Y}^q_s-v_0\circ \widetilde{X}^q_s)\,dx\hspace{0.5pt}ds\\&+\int_0^t\int_{B(x_0,r)}[D\Phi_q^{-1}(v-v_q)]\circ Y_s\;dx\hspace{0.5pt}ds.
	\label{gronwall + integral}
\end{align}

We will show that the last term goes to $0$ as $q\to \infty$. First of all, since $Y_s$ is a volume-preserving homeomorphism, we have
\[\int_{B(x_0,r)}[D\Phi_q^{-1}(v-v_q)]\circ Y_s=\int_{Y_s^{-1}(B(x_0,r))} D\Phi_q^{-1}(v-v_q).\]
Indeed, we can approximate uniformly $Y_s$ by volume-preserving diffeomorphisms (see~\cite{Siko}). Passing to the limit in the change of variables formula for these diffeomorphisms yields the previous identity.

Instead of estimating an integral involving the composition with a rough map, we now have to estimate an integral in a rough domain. We will see that integrating in the interior of $Y_s^{-1}(B(x_0,r))$ is not an issue, while the region near the boundary is problematic. To simplify the notation, we define \[\Sigma_s\coloneqq Y_s^{-1}(\partial B(x_0,r)).\] 
Fix $q\geq 1$ sufficiently large so that $\lambda_q^{-1}<r^{\alpha}$. We fix a grid of size $\ceil{\lambda_q}^{-1}$ and we choose a differentiable partition of unity $\{\chi_m\}_{m=0}^\infty$ associated to this grid. Let us denote by $\Lambda_1$ the set of the indices of the cubes that intersect $\Sigma_s$ and by $\Lambda_2$ the set of the indices of the cubes that are contained in $Y_s^{-1}(B(x_0,r))$. We can then write the integral as
\begin{align}
	\begin{split}
		\int_{Y_s^{-1}(B(x_0,r))} D\Phi_q^{-1}(v-v_q)=&\sum_{m\in \Lambda_1}\int_{Y_s^{-1}(B(x_0,r))}\chi_m D\Phi_q^{-1}(v-v_q)\\&+\sum_{m\in \Lambda_2}\int_{\TT^3}\chi_m D\Phi_q^{-1}(v-v_q).
		\label{descomposición integral difeo}
	\end{split}
\end{align}
We can crudely estimate each term in the first sum by the supremum of the function times the volume of the corresponding cube, obtaining
\[\abs{\sum_{m\in \Lambda_1}\int_{Y_s^{-1}(B(x_0,r))}\chi_m D\Phi_q^{-1}(v-v_q)}\leq 2^{15}\delta_{q+1}^{-1280}\lambda_q^{-3}\abs{\Lambda_1},\]
where we have used (\ref{derivative Phiq}) and (\ref{uniform bound vq}). Since we are assuming that $q\geq 1$, we have
\[\lambda_{q+1}=a^{b^{q+1}-1}=a^{b(b^q-1)}a^{b-1}=\lambda_q^b\lambda_1\leq \lambda_q^{1+b}=\lambda_q^{5/2}.\]
Therefore, we can write
\[\abs{\sum_{m\in \Lambda_1}\int_{Y_s^{-1}(B(x_0,r))}\chi_m D\Phi_q^{-1}(v-v_q)}\leq 2^{15}\lambda_q^{12800\alpha-3}\abs{\Lambda_1}.\]

We need to estimate the number of cubes that intersect $\Sigma_s$. Let 
\[L\coloneqq 1+ \max_{s\in[0,t]}\norm{Y_s^{-1}}_{C^{\tau^2}(\TT^3)}=1+\max_{s\in[-t,0]}\norm{Y_s}_{C^{\tau^2}(\TT^3)}\]
because of the group property of $Y$. We then decompose $\partial B(x_0,r)$ as a union of sets $S_1, \dots, S_n$ such that their diameter is less than $(L\lambda_q)^{-\tau^{-2}}\leq r$. This can be done in such a way that \[n\leq C_r[(L\lambda_q)^{-\tau^{-2}})]^{-2}= C_rL^{2\tau^{-2}}\lambda_{q}^{2\tau^{-2}}\] for certain constant $C_r>0$ depending on $r$. We see that $\Sigma_s=\bigcup_{i=1}^nY_s^{-1}(S_i)$ and
\[\text{diam}(Y_s^{-1}(S_i))\leq L\,\text{diam}(S_i)^{\tau^2}\leq \lambda_q^{-1}.\]
Hence, each $Y_s^{-1}(S_i)$ intersects, at most, $8$ of the cubes of the grid. Therefore,
\[\abs{\Lambda_1}\leq 8n\leq 8C_rL^{2\tau^{-2}}\lambda_{q}^{2\tau^{-2}}.\]
We conclude that
\begin{align*}
	\abs{\sum_{m\in \Lambda_1}\int_{Y_s^{-1}(B(x_0,r))}\chi_m D\Phi_q^{-1}(v-v_q)}&\leq\left(2^{18}C_rL^{2\tau^{-2}}\right)\lambda_q^{12800\alpha+2\tau^{-2}-3} \\ &\stackrel{(\ref{cota gamma por abajo})}{\leq} \left(2^{18}C_rL^{2\tau^{-2}}\right)\lambda_q^{-\frac{1}{2}(3-2\tau^{-2})},
\end{align*}
which goes to $0$ as $q\to\infty$ because $\tau^2>2/3$. 

Concerning the second sum in (\ref{descomposición integral difeo}), we write
\[\sum_{m\in \Lambda_2}\int_{\TT^3}\chi_m D\Phi_q^{-1}(v-v_q)=\sum_{J=(q,0,0)}^\infty\sum_{m\in \Lambda_2}\int_{\TT^3}\chi_m D\Phi_q^{-1}(v_{J+1}-v_J).\]
For a fixed $J\geq (q,0,0)$, by (\ref{cotas A y B}) and (\ref{integral norma H-1}) we have
\begin{align*}
	\abs{\sum_{m\in \Lambda_2}\int_{\TT^3}\chi_m D\Phi_q^{-1}(v_{J+1}-v_J)}&\leq \lambda_q^{-3}\abs{\Lambda_2}\left(\delta_{q(J)+1}^{3000}\norm{\chi_mD\Phi_q}_0+\lambda_{J+1}^{-1}\delta_{q(J)+1}^{3000}\norm{\chi_mD\Phi_q}_1\right) \\&\lesssim \frac{\delta_{q(J)+1}}{\delta_{q+1}}\left(\delta_{q+1}^{3000}\norm{D\Phi_q}_0+\lambda_q^{-1}\delta_{q+1}^{3000}\norm{\chi_mD\Phi_q}_1\right).  
\end{align*}
Taking into account that the cutoffs $\chi_m$ can be chosen so that 
\[\norm{\chi_m}_N\leq C_N\lambda_q^{-N}\]
for some constants $C_N$ depending on $N\in \NN$, it follows from (\ref{derivative Phiq}) and (\ref{second derivative Phiq}) that
\[\delta_{q+1}^{3000}\norm{D\Phi_q}_0+\lambda_q^{-1}\delta_{q+1}^{3000}\norm{\chi_mD\Phi_q}_1\lesssim \delta_{q+1}^{400}.\]
On the other hand,
\[\frac{\delta_{q(J)+1}}{\delta_{q+1}}=a^{-4\alpha(b^{q(J)+1}-b^{q+1})}\leq a^{-4\alpha\ln(b)[q(J)-q]}\leq 2^{-[q(J)-q]}\]
for sufficiently large $a>1$. Since this defines a convergent series, we have
\[\abs{\sum_{m\in \Lambda_2}\int_{\TT^3}\chi_m D\Phi_q^{-1}(v-v_q)}\leq \sum_{J=(q,0,0)}^\infty 2^{-[q(J)-q]}\delta_{q+1}^{400}\lesssim \delta_{q+1}^{400},\]
where the implicit constant is independent of $q$. 

We conclude that both sums in (\ref{descomposición integral difeo}) converge to $0$ as $q\to\infty$. Therefore, passing to the limit $q\to\infty$ in (\ref{gronwall + integral}) leads to
\[\int_{B(x_0,r)}(\widetilde{Y}_t-\widetilde{X}_t)\,dx=\int_0^t\int_{B(x_0,r)}(v_0\circ\widetilde{Y}_s-v_0\circ \widetilde{X}_s)\,dx\hspace{0.5pt}ds.\]
Dividing by $\abs{B(x_0,r)}$ and taking the limit $r\to0$ yields
\[\widetilde{Y}_t(x_0)-\widetilde{X}_t(x_0)=\int_0^t[v_0(\widetilde{Y}_s(x_0))-v_0( \widetilde{X}_s(x_0))]\,ds.\]
Since $\norm{v_0}_1\leq 1$, we have
\[\abs{\widetilde{Y}_t(x_0)-\widetilde{X}_t(x_0)}\leq \int_0^t\abs{\widetilde{Y}_s(x_0)-\widetilde{X}_s(x_0)}\,ds.\]
Taking into account that $\widetilde{Y}_0=\widetilde{X}_0$ and applying Grönwall's inequality leads to
\[\widetilde{Y}_t(x_0)=\widetilde{X}_t(x_0).\]
Since $x_0\in \TT^3$ is arbitrary, we conclude $\widetilde{Y}=\widetilde{X}$, from which we deduce $Y=X$, as claimed. This completes the proof of Theorem~\ref{T.main}.

\section{Proof of Theorem 1.2}
\label{S.proofGrad}
Our approach is based on adapting \cref{T.main} to bounded domains. Applying this result to a solenoidal field $v_0$ with suitable geometric properties will yield the desired weak solution $v$.

To adapt \cref{T.main} to bounded domains, we will use an initial subsolution in which the Reynolds stress vanishes at $\partial\Omega$. This means that we can perturb only in the interior of the domain and essentially ignore the boundary throughout the iterative process. The following proposition is the key ingredient:
\begin{proposition}
	\label{divergencia nula con valor de frontera}
	Let $\Omega\subset \RR^3$ be a bounded domain with smooth boundary. Let $S_0\in C^\infty(\partial\Omega, \mathcal{S}^{ 3})$ be a matrix-valued field such that $S_0\nu=0$, where $\nu$ is the (outer) normal vector to $\partial\Omega$. Then, there exists $S\in C^\infty(\overline{\Omega},\mathcal{S}^{3})$ such that
	\[\begin{cases}
		\Div S=0, \\ S|_{\partial\Omega}=S_0.
	\end{cases}\]
\end{proposition}
\begin{proof}
	By the collar neighborhood theorem, there exists a neighborhood $U$ of $\partial\Omega$ in $\overline{\Omega}$ and a diffeomorphism $$\varphi:U\to\partial\Omega\times[0,1), \;x\mapsto(p(x),f(x))$$such that $\partial\Omega=f^{-1}(0)$. We may also assume that $p|_{\partial\Omega}=\text{Id}$.  Let $\{\psi_m:V_m\to \mathbb{R}^2\}_m$ be an atlas of $\partial\Omega$ and let $U_m\coloneqq \varphi^{-1}(V_m\times[0,1))$. We can then define the diffeomorphisms $$\Phi_m:U_m\to\mathbb{R}^2\times[0,1), \; x\mapsto(\psi_m(p(x)), f(x)).$$Note that
	\begin{align}
		\nonumber f(\Phi_m^{-1}(y_1,y_2,y_3))=y_3\qquad &\Rightarrow \qquad (\nabla f\circ\Phi_m^{-1})D\Phi_m^{-1}=e_3\\ \nonumber &\Rightarrow \qquad \nabla f\,(D\Phi_m)^{-1}=e_3 \\ &\Rightarrow \qquad (\nabla f)^t=(D\Phi_m)^t\,e_3^t, \label{aux divergencia}
	\end{align}
	where we are considering vectors to be $1\times 3$ matrices. For convenience, let us denote $R_m\coloneqq D\Phi_m\circ p$. We define a matrix $\widetilde{S}_0^m\in C^\infty(U_m,\mathcal{S}^3)$ given by$$\widetilde{S}_0^m\coloneqq R_m (S_0\circ p)R_m^t.$$Since $\nabla f(x_0)$ is parallel to $\nu(x_0)$ for any $x_0\in \partial\Omega$, it follows from (\ref{aux divergencia}) and the tangency condition for $S_0$ that$$\widetilde{S}^m_0\,e_3^t=0.$$
	
	We define an auxiliary tensor $(\widetilde{A}_m)^{ik}_{jl}$ by $$(\widetilde{A}_m)^{i\hspace{1pt}3}_{j\hspace{1pt}3}=-(\widetilde{A}_m)^{3\hspace{1pt}i}_{j,3}=-(\widetilde{A}_m)^{i\hspace{1pt}3}_{3\hspace{1pt}j}=-(\widetilde{A}_m)^{3\hspace{1pt}i}_{3\hspace{1pt}j}=(\widetilde{S}^m_0)_{ij}$$and let all other entries equal $0$. Since $(\widetilde{S}^m_0)_{ij}=0$ whenever one of the indices is $3$, $\widetilde{A}_m$ is well-defined. By construction, it is antisymmetric in $(i,k)$ and in $(j,l)$. In addition, it is symmetric in $(i,j)$ because so is $\widetilde{S}^m_0$. We use this auxiliary tensor to define a tensor $(A_m)^{ik}_{jl}$ given by$$(A_m)^{ik}_{jl}\coloneqq (\widetilde{A}_m)^{pr}_{qs}\,(R_m^{-1})_{ip}\,(R_m^{-1})_{jq}\,(R_m^{-1})_{kr}\,(R_m^{-1})_{ls}\hspace{0.5pt},$$which has the same symmetry properties as $\widetilde{A}_m$.
	
	Next, we fix cutoff functions $\chi_m\in C^\infty_c(U_m)$ such that $\sum_m\chi_m=1$ on $\partial\Omega$. We define $S\in C^\infty(U,\mathcal{S}^3)$ given by$$S_{ij}\coloneqq\sum_m\partial_{kl}\left[\frac{1}{2}\chi_m(A_m)^{ik}_{jl}\,f^2\right].$$Since $(A_m)^{ik}_{jl}$ is symmetric in $(i,j)$, $S$ is symmetric, as claimed. In addition, it is divergence-free because $(A_m)^{ik}_{jl}$ is antisymmetric in $(j,l)$. We will now check that $S(x_0)=S_0(x_0)$ for any $x_0\in \partial\Omega$. In what follows, all the functions will be evaluated at $x_0$, but we omit it in the notation for the sake of clarity. Since $f$ vanishes at $\partial \Omega$, we have$$S_{ij}=\sum_m\chi_m(A_m)^{ik}_{jl}\,\partial_kf\,\partial_lf.$$Taking into account that $\nabla f\,R_m^{-1}=e_3$, we see that$$\begin{aligned}(A_m)^{ik}_{jl}\,\partial_kf\,\partial_l f&=(\widetilde{A}_m)^{pr}_{qs}\,(R_m^{-1})_{ip}\,(R_m^{-1})_{jq}\,[(R_m^{-1})_{kr}\,\partial_kf]\,[(R_m^{-1})_{ls}\,\partial_lf] \\ &=(\widetilde{A}_m)^{p\hspace{1pt}3}_{q\hspace{1pt}3}\,(R_m^{-1})_{ip}\,(R_m^{-1})_{jq} =(\widetilde{S}^m_0)_{pq}\,(R_m^{-1})_{ip}\,(R_m^{-1})_{jq} \\ &=[R_m^{-1} \,\widetilde{S}_0^m\,(R_m^{-1})^t]_{ij}=(S_0)_{ij}\end{aligned}$$by the definition of $\widetilde{A}_m$ and $\widetilde{S}_0^m$. Since $\sum_m\chi_m=1$ on $\partial\Omega$, we conclude that $S(x_0)=S_0(x_0)$. Since $x_0\in\partial\Omega$ is arbitrary, we have $S|_{\partial\Omega}=S_0$, as claimed.
\end{proof}

We will now adapt \cref{T.main} to bounded domains. For the sake of simplicity, we have imposed that the field $v_0$ is non-vanishing on the boundary of the domain, which makes the application of \cref{T.main} much more direct. This assumption can probably be removed with some extra work, but for the proof of \cref{T.Grad} that refinement is not needed.

\begin{theorem}
	\label{T.dominios}
	Let $\Omega\subset \RR^3$ be a smooth toroidal domain. Let $v_0\in C^\infty(\overline{\Omega})$ be a solenoidal field, whose flow we denote by $X^0_t$. Assume that $v_0$ is tangent to $\partial\Omega$ and that it is non-vanishing on $\partial\Omega$. Fix some $\tau\in (\sqrt{2/3},1)$ and some $\varepsilon>0$. 
	
	Then, {for some $\alpha>0$ depending on $\tau$ but not on~$v_0$,} there exists a weak steady Euler flow $v\in C^\alpha(\Omega)$, a sequence of volume-preserving diffeomorphisms $\Phi_q\in \SDiff(\Omega)$ and a volume-preserving Hölder homeomorphism $\Phi\in \SHom^{\tau}(\Omega)$ such that:
	\begin{enumerate}
		\item The diffeomorphisms $\Phi_q$ converge to $\Phi$ in $C^\tau(\Omega)$ as $q\to\infty$.
		\item The sequence of smooth pushed forward vields $v_q\coloneqq (\Phi_q)_\ast v_0$ converges to $v$ in $C^\alpha(\Omega)$.
		\item $\|v-v_0\|_{H^{-1}(\overline{\Omega})}+ \|\Phi-\id\|_{C^0(\overline{\Omega})}<\varepsilon$.
		\item $X_t:= \Phi\circ X^0_t\circ \Phi^{-1}$ is the only flow of~$v$ of class $C^0_{\rm loc}(\RR,C^{\tau^2}(\TT^3,\TT^3))\cap C^1_{\rm loc}(\RR,C^{0}(\TT^3,\TT^3))$.
	\end{enumerate}
\end{theorem}
\begin{proof}[Sketch of the proof]
	The proof of this theorem is quite similar to the proof of \cref{T.main}, so we will merely point out the differences.
	
	By the collar neighborhood theorem, there exists a neighborhood $U$ of $\partial\Omega$ in $\overline{\Omega}$ and a diffeomorphism $$U\to\partial\Omega\times[0,1), \;x\mapsto(p(x),f(x))$$such that $\partial\Omega=f^{-1}(0)$. We may also assume that $p|_{\partial\Omega}=\text{Id}$. Since $v_0$ does not vanish on $\partial \Omega$, by further reducing $U$ we may assume that $v_0$ does not vanish on $\overline{U}$. We may also assume that $\partial U$ is smooth. Fix $\varphi\in C^\infty(\RR)$ that vanishes for $x\geq 2$ and equals $1$ for $x\leq 1$. For $\varepsilon>0$, we define the cutoff in $\overline{\Omega}$
	\[\varphi_\varepsilon(x)\coloneqq \varphi(\varepsilon^{-1}f(x)).\]
	Therefore, $\varphi_\varepsilon$ equals $1$ in a neighborhood of $\partial\Omega$ but it is supported in an arbitrarily small neighborhood of $\partial\Omega$ (by reducing $\varepsilon$). 
	
	By \cref{divergencia nula con valor de frontera}, there exists a divergence-free matrix $M_1 \in C^\infty(\overline{\Omega},\mathcal{S}^3)$ such that $M_1|_{\partial\Omega}=v_0\otimes v_0$. Let
	\[\rho_\varepsilon\coloneqq \Div(\varphi_\varepsilon M_1)=M_1\hspace{0.5pt}\nabla\varphi_\varepsilon\]
	because $M_1$ is divergence-free. Let $x\in \supp \rho_\varepsilon$. Since $M_1(p(x))=(v_0\otimes v_0)(p(x))$ and $\nabla f(p(x))$ is parallel to $\nu(p(x))$, which is perpendicular to $v_0(p(x))$, we have
	\begin{align*}
		\abs{\rho_\varepsilon(x)}&\leq \abs{[M_1(x)-M_1(p(x))]\varepsilon^{-1}\varphi'(\varepsilon^{-1}f(x))\nabla f(x)} \\&\hspace{10pt}+\abs{\varepsilon^{-1}\varphi'(\varepsilon^{-1}f(x))M_1(p(x))[\nabla f(x)-\nabla f(p(x))]} \\ &\leq (\norm{M_1}_1\norm{\varphi}_1\norm{f}_1+\norm{M_1}_0\norm{\varphi}_1\norm{f}_2)\,\varepsilon^{-1}\abs{x-p(x)}=O(1)
	\end{align*}
	because $\rho_\varepsilon$ is supported in a $(2\varepsilon)$-neighborghood of $\partial\Omega$. By \cref{frecuencias bajas de dominio pequeño}, we have
	\[\norm{\rho_\varepsilon}_{B^{-1/2}_{\infty,\infty}}\lesssim \varepsilon^{1/2},\]
	where the implicit constants are allowed to depend on anything but $\varepsilon$. On the other hand, since $M_1\hspace{0.5pt}\nu=0$, it follows from (\ref{identidad de Green matrices}) that
	\[\int_\Omega \rho_\varepsilon\cdot\xi=0 \qquad \forall \xi\in \ker D^s.\]
	By \cref{invertir divergencia matrices}, there exists $M_2\in C^\infty(\overline{\Omega},\mathcal{S}^3)$ such that $\Div M_2=\rho_\varepsilon$ and whose support is contained in the interior of $U$. Furthermore, we have
	\[\norm{M_2}_0\lesssim \varepsilon^{1/2}.\]
	
	We have all the ingredients that we need to define the initial subsolution. Let
	\begin{align*}
		p_0&\coloneqq -(1-\varphi_\varepsilon)\abs{v_0}^2, \\
		R_0&\coloneqq (1-\varphi_\varepsilon)(v_0\otimes v_0-\abs{v_0}^2\Id)+\varphi_\varepsilon(M_1-v_0\otimes v_0)-M_2.
	\end{align*}
	By construction of $M_1$ and $M_2$, we see that $(v_0,p_0,R_0)$ is a subsolution. Dividing by an appropriate constant, we may assume that
	\begin{equation}
		\norm{v_0}_1+\norm{R_0}_1\leq 1.
		\label{cota subs inicial dominios}
	\end{equation}
	Note that the second term in the definition of $R_0$ vanishes on $\partial\Omega$, while the other two terms vanish in a neighborhood of $\partial\Omega$. We define
	\[\Omega_q \coloneqq \left\{x\in\Omega:\abs{v_0(x)}>\frac{1}{2}\delta_q^{1/2}\;\&\;\dist(x,\partial\Omega)>\frac{1}{4}\delta_{q}\right\}\]
	and $\Sigma_q=\partial\Omega_q$. It follows from this definition and (\ref{cota subs inicial dominios}) that
	\[\abs{R_0(x)}\leq \frac{1}{4}\delta_q \qquad \forall x\in \overline{\Omega}\backslash \Omega_q.\]
	Furthermore, by taking $a>1$ sufficiently large we may assume that (\ref{inductive growth}) holds, since $v_0$ does not vanish on $U$.
	
	The inductive hypotheses for the sequence of subsolutions are the same as on the torus. In fact, the inductive process is exactly the same outside $U$ because $R_0$ has the familiar expression (\ref{def R0}). However, in $U$ the initial Reynolds stress does not have the convenient expression (\ref{def R0}), which means that it must be canceled in the stages $l\in\{2,\dots, 7\}$ instead of $l\in\{0,1\}$. This may seem worrying because $R_0$ does not satisfy estimates analogous to (\ref{inductive decomposition error}). However, the purpose of these estimates and various other considerations was to ensure that the perturbation was not too big. More precisely, we required that the angle between the velocity and several vectors remain above certain threshold. Since, $v_0$ does not vanish on $U$, we have
	\[\sup_{x\in U}\frac{\abs{R_0(x)}^{1/2}}{\abs{v_0(x)}}\lesssim \varepsilon^{1/4},\]
	so the ratio between the initial Reynolds stress and the initial velocity can be made arbitrarily small by reducing $\varepsilon$. By (\ref{cota mejor si gamma pequeña}), the perturbations will be arbitrarily small compared to $v_0$. While it is true that we will have to sum infinitely many of these terms, the size of $R_0$ outside $\Omega_q$ decreases at least exponentially with $q$. Hence, we can sum all the terms and the sum will still be arbitrarily small.
	
	In conclusion, the perturbations on $U$ can be made so small with respect to the initial velocity that they do not mess the geometry and we can carry out the iterative process exactly as on the torus. Note that at each step of the process, we do not perturb in a neighborhood of $\partial\Omega$, so the final field and homeomorphism satisfy $\Phi|_{\partial\Omega}=\Id$ and  $v|_{\partial\Omega}=v_0$. In particular, $v$ is tangent to $\partial\Omega$. 
	
	The rest of the proof is essentially the same as on $\TT^3$ except for the uniqueness of the flow, which requires some care. Note that it suffices to prove that $Y_t=X_t$ on $\Omega$, since equality on $\partial\Omega$ will follow by continuity. Given $x_0\in \Omega$, we will have $B(x_0,r)\subset \Omega$ if $r>0$ is sufficiently small. Since $Y^{-1}_s$ is a homeomorphism, we will also have $Y_s^{-1}(B(x_0,r))\subset \Omega$. Furthermore, for $s\in[0,t]$ it will remain at a distance $d>0$ from $\partial\Omega$. Choosing $q\geq 0$ sufficienty large so that $\sqrt{3}\hspace{0.5pt}\lambda_q^{-1}>d$, we can then proceed exactly as on the torus.
\end{proof}

Once we have \cref{T.dominios}, the proof of \cref{T.Grad} essentially reduces to finding an initial divergence-free field $v_0$ with the required geometric properties:
\begin{proof}[Proof of \cref{T.Grad}]
	Since $\Omega$ is toroidal, there exists a diffeomorphism $\Psi:\overline{\mathbb D}\times \mathbb S^1\to \overline{\Omega}$, where $\mathbb D\subset \RR^2$ is the unit disk centered at the origin. Let us endow $\mathbb D\times \mathbb S^1$ with coordinates $y=(y_1,y_2)\in \mathbb D$ and $\theta\in \mathbb S^1:=\mathbb R/\mathbb Z$. In these coordinates, the Euclidean volume reads as $\rho(y,\theta)dy\,d\theta$ for some positive factor $\rho$.
	
	Let us construct a vector field $u_0\in C^\infty(\mathbb D\times\mathbb S^1,\RR^3)$ with the desired geometric properties and then pushforward this field to $\Omega$ using $\Psi$. To this end, it is convenient to introduce polar coordinates $(y_1,y_2)=(r\cos\beta,r\sin\beta)$. Consider the vector field $u_0\in C^\infty(\mathbb D\times \mathbb S^1,\RR^3)$ given by
	\[u_0\coloneqq \frac{1}{\rho(r,\beta,\theta)}\Big[h_1(r^2)\partial_\beta+h_2(r^2)\partial_\theta\Big]\,,\]
	where $h_k\in C^\infty(\mathbb R)$ for $k=1,2$ and assume that $h_2>0$ on $\mathbb R$. It is easy to check that $u_0$ is divergence-free, tangent to $\partial(\mathbb D\times \mathbb S^1)$ and non-vanishing on $\overline {\mathbb D}\times \mathbb S^1$. Moreover, the function $$f:=y_1^2+y_2^2$$ is a first integral of $u_0$ whose level sets define a foliation by tori as in the statement of the theorem. Then, labeling by pairs $(h_1,h_2)$ as above we obtain a family of vector fields $u_0$ satisfying all the desired geometric properties. We want to emphasize that pairs $(h_1,h_2), (\tilde h_1,\tilde h_2)$ such that 
	\[
	\frac{h_1}{h_2}\neq \frac{\tilde h_1}{\tilde h_2}
	\]
	yield vector fields $u_0$ and $\tilde u_0$ in the family that are not topologically equivalent. The reason is that the Poincar\'e map defined by the flow of $u_0$ on the section $\{\theta=0\}$ is of the form 
	$$
	(r,\beta)\mapsto (r,\beta+h_1(r)/h_2(r))\,,
	$$
	and hence the function $h_1/h_2$ is the rotation number of the circle diffeomorphism $\beta+h_1(r)/h_2(r)$ for each fixed $r>0$. The claim follows using that the rotation number is invariant under topological conjugation of homeomorphisms.
	
	Next, consider the field $v_0\in C^\infty(\overline{\Omega},\RR^3)$ given by $v_0=\Psi_\ast\hspace{0.5pt} u_0$. By construction, $v_0$ is solenoidal. Furthermore, it is tangent to $\partial\Omega$ and non-vanishing on $\partial\Omega$. Defining $g\coloneqq f\circ \Psi^{-1}$, we obviously have that $g$ is a first integral of $v_0$, i.e.,
	\[v_0\cdot\nabla g=(D\Psi\,u_0)\circ\Psi^{-1}\cdot(\nabla f\circ\Psi^{-1})D\Psi=(D\Psi\,u_0\cdot\nabla f)\circ\Psi^{-1}\,D\Psi=0.\]
	Therefore, the flow $X^0_t$ of $v_0$ satisfies $g\circ X^0_t=g$. We apply \cref{T.dominios} to $v_0$, obtaining a homeomorphism $\Phi:\overline{\Omega}\to\overline{\Omega}$ and a weak steady Euler flow $v\in C^\alpha(\Omega)$ for some $\alpha>0$. The map $X_t\coloneqq \Phi\circ X^0_t\circ \Phi^{-1}$ is a flow of $v$ and it is unique within the regularity class $C^{\tau^2}$. Defining $\tilde{f}\coloneqq f\circ \Phi^{-1}$, we have
	\[\tilde{f}\circ X_t=(f\circ\Phi^{-1})\circ (\Phi\circ X^0_t\circ\Phi^{-1})=(f\circ X^0_t)\circ \Phi^{-1}=f\circ\Phi^{-1}=\tilde{f}.\]
	Hence, the level sets of $\tilde{f}$ are invariant under the flow $X_t$. Since they are homeomorphic to the level sets of $g$, they provide a foliation of $\Omega$ by tori, singular along a closed embedded circle~$\mathcal{C}\subset\Omega$. Moreover, since different fields in the family defined by $u_0$ (and hence $v_0$) are not topologically equivalent, the uniqueness of the flow map $X_t$ in its regularity class implies that we obtain a (nontrivial) family of different steady states $v$.  This concludes the proof. 
\end{proof}

\section*{Acknowledgements}
The authors are indebted to Massimo Sorella for valuable discussions. This work has received funding from the European Research Council (ERC) under the European Union's Horizon 2020 research and innovation programme through the grant agreement~862342 (A.E.). It is partially supported by the grants CEX2023-001347-S, RED2022-134301-T and PID2022-136795NB-I00 (A.E. and D.P.-S.) funded by MCIN/AEI/10.13039/501100011033, and Ayudas Fundaci\'on BBVA a Proyectos de Investigaci\'on Cient\'ifica 2021 (D.P.-S.). J.P.-T. was partially supported by an FPI grant CEX2019-000904-S-21-4 funded by MICIU/AEI/10.13039/501100011033 and by FSE+.

\appendix 
\section{Some auxiliary results}
We collect here some auxiliary results that are needed in our construction. First of all, we recall the definition of Hölder norms. Let us denote the supremum norm as \[\norm{f}_0\coloneqq \sup_{x\in\TT^3}\abs{f(x)},\]
where $\abs{\cdot}$ denotes the absolute value when working with real-valued maps, the modulus when working with vector-valued maps and the operator norm when working with matrix-valued maps, that is,
\[\abs{M}\coloneqq \sup_{\abs{v}\leq 1}\abs{Mv}.\]
This is the norm that we will use whenever we are working with matrices. It is not very important, however, since all norms on a finite-dimensional vector space are equivalent.

Given an integer $k\geq 0$ and $\alpha \in (0,1)$, we define the Hölder seminorms as
\begin{align*}
	[f]_{k}&\coloneqq \max_{\abs{\beta}=k}\norm{\partial^\beta f}_0, \\
	[f]_{k+\alpha}&\coloneqq \max_{\abs{\beta}=k}\sup_{x\neq y}\frac{\abs{\partial^\beta f(x)-\partial^\beta f(y)}}{\abs{x-y}^\alpha},
\end{align*}
where $\beta$ is a multi-index. The Hölder norms are then defined as
\begin{align*}
	\norm{f}_k&\coloneqq \sum_{j\leq k}[f]_j, \\
	\norm{f}_{k+\alpha}&\coloneqq \norm{f}_k+[f]_{k+\alpha}.
\end{align*}
They can, indeed, be shown to be seminorms and norms, respectively. The Hölder space $C^{k,\alpha}(\TT^3)$ is defined as the set of $C^k$ functions such that $[\cdot]_{k+\alpha}$ is finite. It is a Banach space when equipped with the norm $\norm{\cdot}_{k+\alpha}$. It is easy to check that the following inequality holds for any $0\leq r\leq 1$:
\[[fg]_r\leq C([f]_r\norm{g}_0+\norm{f}_0[g]_r).\]
In addition, we will often use the interpolation inequality
\[[f]_s\leq C\norm{f}_0^{1-\frac{s}{r}}[f]_r^{\frac{s}{r}},\]
where $r\geq s\geq 0$.

Let us also consider a Littlewood--Payley decomposition, e.g.\ as in~\cite[Section 2.2]{BCD}. For this, we take smooth radial functions $\chi,\varphi:\RR^3\to[0,1]$, whose supports are contained in the ball $B(0,\frac43)$ and in the annulus $\{\frac34 <|\xi|<\frac83\}$ respectively, with the property that
\[
\chi(\xi)+\sum_{N=0}^\infty\varphi(2^{-N}\xi)=1
\]
for all~$\xi\in\RR^3$. In terms of the Fourier multipliers $P_{<}:=\chi(D)$ and $P_N:=\varphi(2^{-N}D)$, the Besov norm $B^s_{\infty,\infty}$ (which is equivalent to the H\"older norm $C^s$ if $s\in\RR^+\backslash\NN$, and strictly weaker if $s\in\NN$) can be written as
\begin{equation}\label{E.BesovF}
	\|h\|_{B^{s}_{\infty,\infty}}:=\|P_{<}h\|_0+\sup_{N\geq0}2^{Ns}\|P_Nh\|_0\,.
\end{equation}
In the rest of the Appendix, we use the standard notation $D:=-i\nabla$.

The next result we will need, which is standard, is about the existence of cut-off functions with good bounds for their derivatives. For a proof, see e.g.~\cite[Lemma B.1]{Extension}.
\begin{lemma}
	\label{lemma cutoff}
	Let $A\subset \mathbb{R}^3$ be a measurable set and let $r>0$. There exists a cutoff function $\chi_r\in C^\infty(\mathbb{R}^3,[0,1])$ whose support is contained in $A+B(0,r)$ and such that $\chi_r\equiv 1$ in a neighborhood of $A$. Furthermore, for any $N\geq 0$ we have \[\left\|\chi_r\right\|_N\leq C_N\,r^{-N}\]
	for some universal constants $C_N$ depending on $N$ but not on $A$ or $r$. We may also assume that the square root of $\chi_r$ and $(1-\chi_r)$ is smooth and that it satisfies analogous estimates.
\end{lemma}

Our next result is used to decompose the Reynolds stress into simpler matrices, which can be canceled in a single iteration of the perturbative process.
\begin{lemma} \label{geometric lemma}
	Let $\{\zeta_1,\zeta_2,v\}$ be an orthonormal basis of $\RR^3$. There exist $r>0$, unitary vectors $\zeta_3,\dots \zeta_6$ and smooth maps $\Gamma_1,\dots, \Gamma_6$ defined in the ball $B(\Id,r)\subset\mathcal{S}^3$ such that
	\[S=\sum_{j=1}^6\Gamma_j(S)^2\zeta_j\otimes\zeta_j \qquad \forall S\in B(\Id,c_0).\]
	In addition, the angle described by $v$ and any of the vectors $\zeta_j$ is between $45^{\circ}$ and $90^{\circ}$.
\end{lemma}
\begin{proof}
	We define
	\begin{align*}
		\zeta_3&\coloneqq \frac{1}{\sqrt{2}}v+\frac{1}{2}(\zeta_1+\zeta_2), \qquad \zeta_4\coloneqq \frac{1}{\sqrt{2}}v-\frac{1}{2}(\zeta_1+\zeta_2),\\
		\zeta_5&\coloneqq \frac{1}{\sqrt{2}}v+\frac{1}{2}(\zeta_1-\zeta_2),\qquad		\zeta_6\coloneqq \frac{1}{\sqrt{2}}v-\frac{1}{2}(\zeta_1-\zeta_2).
	\end{align*}
	It is clear that the vectors $\zeta_j, \;j=1,\dots 6$ are unitary and the angle described by $v$ and any of them is between $45^{\circ}$ and $90^{\circ}$. In addition, $\{\zeta_j\otimes\zeta_j:j=1,\dots, 6\}$ is a basis of $\mathcal{S}^3$. Thus, the map  $\RR^6\to\mathcal{S}^3, (x_1,\dots, x_6)\mapsto \sum_{j=1}^6x_j\zeta_j\otimes\zeta_j$ is a linear isomorphism between both vector spaces. Since
	\[\Id=\sum_{j=1}^6\frac{1}{2}\zeta_j\otimes\zeta_j,\]
	the coordinates of the inverse map must be positive in a sufficiently small neighborhood of $\Id$, so we can take the square root and the claim follows.
\end{proof}

In the following lemma, which is now standard, we obtain bounds for the Besov norms of functions involving rapidly oscillating exponentials. Since we want to apply it to a particular class of functions, for completeness we provide a proof of the estimate.

\begin{lemma}
	\label{stationary phase lemma}
	Let $s>-1$, $J\in \NN$ and $\lambda\geq 1$. For $j=1, \dots J$, let $c_j\in C^\infty_c(\TT^3)$ and $k_j\in \RR^3$ such that $C_0^{-1}\leq \abs{k_j}\leq C_0$ for some $C_0>0$. Suppose that the supports of at most $J_*$ of the functions $\{c_j\}_{j=1}^J$ have a nonempty intersection, so that the product of any $J_*+1$ of the functions~$c_j$ is identically~$0$. Let
	\[f(x)\coloneqq \sum_{j=1}^J c_j(x)e^{i\lambda k_j\cdot x}.\]
	Then, for any nonnegative integer $m$ we have
	\[\norm{f}_{B^{s}_{\infty,\infty}}\leq C_*\left(\lambda^s\mathcal C_0+\lambda^{-m}\mathcal C_{m+s_+}\right),\]
	where $s_+:=\max\{s,0\}$ and $\mathcal C_\sigma:=\max_{j\leq J}\|c_j\|_\sigma$. The constant~$C_*$ depends on $s,m, J_*$ and $C_0$ but not on~$J$.
\end{lemma}

\begin{proof}
	Let us start with the case $s\geq0$. Write $s=l+\alpha$ with $l\in\NN$ and $\alpha\in[0,1)$. The condition that at most $J_*$ of the supports of the functions~$c_j$ intersect any point $x$ ensures that
	\[
	|\nabla^l f(x)|\leq \sum_{j=1}^J|\nabla^l(c_j(x)e^{i\lambda k_j\cdot x})|\lesssim J_*(\lambda^l\mathcal{C}_0+\mathcal{C}_l)\,.
	\]
	Here we have used that
	\[
	|\nabla^l(c_j(x)e^{i\lambda k_j\cdot x})|\lesssim\sum_{n=0}^l|\nabla^nc_j(x)|\lambda^{l-n}\lesssim\|c_j\|_0\lambda^l+[c_j]_l\leq\lambda^l\mathcal{C}_0+\mathcal{C}_l\,,
	\]
	and the implicit constants depend on~$l$ and~$C_0$.
	Likewise, for any $x,y$,
	\[
	\frac{|\nabla^lf(x)-\nabla f(y)|}{|x-y|^\alpha}\lesssim 2J_*(\lambda^{l+\alpha}\mathcal{C}_0+\mathcal{C}_{l+\alpha})\,.
	\]
	Therefore, for $s\geq0$,
	\[
	\|f\|_{B^s_{\infty,\infty}}\lesssim\|f\|_{s}\lesssim J_*(\lambda^s\mathcal{C}_0+\mathcal{C}_s)\lesssim J_*(\lambda^s\mathcal C_0+\lambda^{-m}\mathcal C_{m+s})\,.
	\]
	To pass to the last line, we have used the interpolation inequality 
	\[
	\|h\|_s\lesssim\|h\|_0^{\frac m{s+m}}\|h\|_{s+m}^{\frac s{s+m}}\lesssim \lambda^s\|h\|_0+ \lambda^{-m}\|h\|_{s+m}\,.
	\]
	
	Let us now consider the case $s\in(-1,0)$. The case $m=0$ is immediate because $\|P_<h\|_0+\|P_Nh\|_0\lesssim\|h\|_0$ for any $N\geq0$ and any~$h$, so that in this case
	\[
	\|f\|_{B^s_{\infty,\infty}}\leq\|P_{<}f\|_0+\sup_{N\geq0}\|P_NFf|_0\lesssim\|f\|_0\leq J_*\mathcal C_0\,.
	\]
	
	Let us now assume that $m\geq1$. One then integrates by parts to write
	\begin{align*}
		P_<f(x)&=\int_{\RR^3}\widehat\chi(x-y)\, f(y)\, dy\\
		&=\int_{\RR^3}\widehat\chi(x-y)\sum_{j=1}^Jc_j(y)\left( \frac{k_j\cdot\nabla_y}{i\lambda|k_j|^2}\right)^m e^{i\lambda k_j\cdot y}\, dy\\
		&=\int_{\RR^3}\sum_{j=1}^Je^{i\lambda k_j\cdot y}\left(\frac{ik_j\cdot\nabla_y}{\lambda|k_j|^2}\right)^m\left[\widehat\chi(x-y)\,c_j(y) \right] \, dy\,.
	\end{align*}
	Using the hypothesis on the supports of~$c_j$ and on the size of~$k_j$, this immediately results in the bound
	\[
	\|P_<f\|_0\lesssim J_*\lambda^{-m}\mathcal{C}_m\,.
	\]
	Analogously, with $N\geq0$,
	\begin{align*}
		P_Nf(x)&=2^{3N}\int_{\RR^3}\widehat\varphi(2^N(x-y))\, f(y)\, dy\\
		&=2^{3N}\int_{\RR^3}\sum_{j=1}^Je^{i\lambda k_j\cdot y}\left(\frac{ik_j\cdot\nabla_y}{\lambda|k_j|^2}\right)^m\left[\widehat\varphi(2^N(x-y))\,c_j(y) \right] \, dy\,,
	\end{align*}
	which yields
	\[
	\|P_Nf\|_0\lesssim J_*\lambda^{-m}(2^{mN}\mathcal{C}_0+\mathcal{C}_m)\,.    \]
	This is better than the direct estimate $\|P_Nf\|\lesssim J_*\mathcal C_0$ when $2^N<\lambda$.   Therefore, for $s\in(-1,0)$,
	\begin{align*}
		\|f\|_{B^s_{\infty,\infty}}&= \|P_<f\|_0+\sup_{N\geq0}2^{Ns}\|P_Nf\|_0\\
		&\lesssim J_*\sup_{N\geq0}2^{Ns}\min\{\mathcal{C}_0,\lambda^{-m}(2^{mN}\mathcal{C}_0+\mathcal{C}_m)\}\\
		&\leq J_*\sup_{1\leq 2^N\leq \lambda}2^{Ns}\lambda^{-m}(2^{mN}\mathcal{C}_0+\mathcal{C}_m) + J_*\sup_{ 2^N\geq \lambda}2^{Ns}\mathcal C_0\\
		&\lesssim J_*(\lambda^s \mathcal{C}_0+\lambda^{-m}\mathcal{C}_m)\,.
	\end{align*}
	The lemma then follows.
\end{proof}

If the function is supported in a small region, we can still prove good estimates for the $B^{-1+\alpha}_{\infty,\infty}$-norm, even in the absence of oscillations. For a proof, see~\cite[Lemma B.4]{Extension}:
\begin{lemma}
	\label{frecuencias bajas de dominio pequeño}
	Let $\Omega\subset \RR^3$ be a bounded domain with smooth boundary. Let $\alpha\in (0,1)$ and let $r>0$ be sufficiently small (depending on $\Omega$). Consider a function $f\in C_c^\infty(\RR^3)$ supported in $\{x\in\RR^3:0<\dist(x,\Omega)<r\}$. We have
	\[\norm{f}_{B^{-1+\alpha}_{\infty,\infty}}\leq C(\Omega,\alpha)\,r^{1-\alpha}\norm{f}_0.\]
\end{lemma}	

Finally, we recall some tools concerning the divergence equation for vector fields and matrix-valued fields. 
Let us start with the case of matrices.
We denote by $\mathcal{S}^3$ the set of symmetric $3\times 3$ matrices. Let $\mathcal U$ be a bounded subset of $\RR^3$ (or  or $\TT^3$) with smooth boundary. We introduce the following operators:
\begin{align*}
	\Div: C^\infty(\overline{\mathcal U},\mathcal{S}^3)&\to C^\infty(\overline{\mathcal U},\RR^3), \quad S\mapsto v_i\coloneqq\partial_j S_{ij}, \\
	D^s: C^\infty(\overline{\mathcal U},\RR^3)&\to C^\infty(\overline{\mathcal U},\mathcal{S}^3), \hspace{12pt} v\mapsto S_{ij}\coloneqq\frac{1}{2}(\partial_i v_j+\partial_j v_i).
\end{align*}
Using the divergence theorem, it is easy to see that the following identity holds:
\begin{equation}
	\int_{\mathcal U}\xi\cdot \Div S+\int_{\mathcal U} D^s\xi:S=\int_{\partial \mathcal U}\xi^t S\,\nu.
	\label{identidad de Green matrices}
\end{equation}
With this notation, we introduce the last result we need, which concerns the divergence equation for symmetric matrices (for a proof, see~\cite[Lemma 2.9]{Extension}):
\begin{lemma}
	\label{invertir divergencia matrices}
	Let $\mathcal U\subset \RR^3$ be a bounded domain with smooth boundary and let $\rho\in C^\infty_c(\mathcal U,\RR^3)$. There exists $S\in C^\infty_c(\mathcal U,\mathcal{S}^3)$ such that $\Div S=\rho$ if and only if
	\[\int_{\mathcal U} \rho\cdot\xi=0 \qquad \forall \xi\in \ker D^S.\]
	In that case, $S$ may be chosen so that for any $N\geq 0$ and $\alpha>0$ we have
	\[\norm{S}_{N}\lesssim \norm{\rho}_{B^{N-1+\alpha}_{\infty,\infty}},\]
	where the implicit constants depend on $N$, on $\alpha$ and on $\mathcal U$.
\end{lemma}

In the case of vector fields, the aforementioned result (and the proof) remain valid. In the following lemma, we prove that, additionally, the implicit constants remain  uniformly bounded when the set is small. Specifically, one has the following bound (which of course also holds for the case of matrix-valued fields):

\begin{lemma}
	\label{invertir divergencia vectores}
	Given a bounded domain with smooth boundary $U\subset \RR^3$, consider the rescaled domain $\mathcal U:=s\, U$ with some $s\in(0,1]$ and a function $f\in C^\infty_c(\mathcal U)$. There exists $z\in C^\infty_c(\mathcal U,\RR^3)$ such that $\Div z=f$ if and only if
	\[\int_{\mathcal U} f=0.\]
	In that case, $z$ may be chosen so that for any $N\geq 0$ and any $\alpha>0$ we have
	\begin{equation}\label{E.myscaling}
		\norm{z}_{N}\lesssim s^{\alpha} \norm{f}_{B^{N-1+\alpha}_{\infty,\infty}},\end{equation}
	where the implicit constants depend on $N$, on $\alpha$ and on $U$, but not on $s\in(0,1]$.   
\end{lemma}

\begin{proof}
	Consider the function $F(x):=f(s x)\in C^\infty_c(U)$, which satisfies $\int_U F=0$. By~\cite[Lemma 2.9]{Extension}, there exists $Z\in C^\infty_c(U,\RR^3)$
	satisfying\footnote{\cite[Lemma 2.9]{Extension} concerns the equation $\Div Z=F$ in the harder case where $Z$ takes values in the space of symmetric matrices. As mentioned above, the result and its proof remain valid in this case.}
	\[
	\Div Z=s F
	\]
	that is bounded as
	\begin{equation}\label{E.ZnsF}
		\|Z\|_N\lesssim s \|F\|_{B^{N-1+\alpha}_{\infty,\infty}}\,,
	\end{equation}
	with a constant independent of~$s$. Here $\alpha$ is any positive constant.
	Since $z(x):= Z(x/s)\in C^\infty_c(\mathcal U,\RR^3)$ then satisfies $\Div z=f$, all we have to do is to estimate~$z$ in terms of~$Z$ and analyze the dependence on~$s$.
	
	When $N\geq1$, this is straightforward. In this case, $B^{N-1+\alpha}_{\infty,\infty}$ coincides with the H\"older space $C^{N-1+\alpha}$, and any function $F\in C^\infty_c(U)$ satisfies
	\begin{equation}\label{E.Fn-1alpha}
		\|F\|_{N-1+\alpha}\lesssim[F]_{N-1+\alpha}
	\end{equation}
	with a constant depending on the diameter of~$U$. Indeed, if $N=1$, for any $x\in U$ and $X_*\in\partial U$, one has
	\[
	|F(x)|= |F(x)-F(x_*)|\leq [F]_\alpha|x-x_*|^\alpha\,.
	\]
	This implies $\|F\|_0\leq [F]_\alpha$, thereby showing~\eqref{E.Fn-1alpha} for $N=1$. If $N\geq2$, one can similarly use Taylor's formula to write
	\[
	|F(x)|=\frac1{(N-2)!}\left|\int_0^1\partial_t^{N-1}F((1-t)x_*+tx)\, t^{N-2}\, dt\right|\lesssim [F]_{N-1}\,,
	\]
	with a constant depending on $|x-x_*|$. Since $[F]_{N-1}\lesssim[F]_{N-1+\alpha}$ by arguing as in the case $N=1$, we arrive at~\eqref{E.Fn-1alpha}. Therefore, for all $N\geq1$ one can replace~\eqref{E.ZnsF} by
	\[
	\|Z\|_N\lesssim s [F]_{N-1+\alpha}\,.
	\]
	Thus \eqref{E.myscaling} follows directly from this bound using the scaling of the norms.
	
	The case $N=0$ requires a different proof. This proof also applies to the case of $N\geq1$ with very minor modifications, but we have chosen to keep the above elementary proof for the benefit of the reader. Our goal now is to analyze the bound
	$\|z\|_0\lesssim s\|F\|_{B^{\alpha-1}_{\infty,\infty}}$.
	Since $F(x):=f(sx)$, we have
	$P_<F(x)=[\chi(sD)f](sx)$,
	so 
	$$
	\|P_<F\|_0= \|\chi(sD)f\|_0\,.
	$$
	For this we will use that for $s\in(0,1]$, the smoothing operator $\chi(sD)$ is uniformly bounded in the symbol class~$S^0$ because
	\[
	|\partial_\xi^\beta[\chi(s\xi)]|= s^{|\beta|}|(\partial^\beta\chi)(s\xi)|\leq s^{|\beta|}\|\chi\|_{|\beta|}\,.
	\]
	Hence (see e.g.~\cite[Proposition 2.78]{BCD}), $\chi(sD):B^s_{p,q}\to B^s_{p,q}$ is bounded by a constant which depends on the Besov space parameters $s\in \RR$ and $1\leq p,q\leq \infty$ but not on~$s\in(0,1]$. 
	
	Let us write $s=:2^{-K}\tilde s$ with an integer $K\geq0$ and $\tilde s\in(\frac12,1]$. Then 
	$$
	\supp \chi(s\cdot)\subset B(0, 2^{K+1})\,, 
	$$
	and the function
	\[
	\chi(\xi) +\sum_{N=0}^{K+2}\varphi(2^{-N}\xi)
	\]
	equals~1 in this region and~0 outside $B(0, 2^{K+3})$. Therefore, we have the identity
	\[
	\chi(s\xi)=\chi(s\xi)\left[\chi(\xi) +\sum_{N=0}^{K+2}\varphi(2^{-N}\xi)\right]\,,
	\]
	which translates into the operator identity
	\[
	\chi(sD)=\left[P_<+\sum_{N=0}^{K+2}P_N\right]\chi(sD)\,.
	\]
	As a consequence of this, the first term in $s\|F\|_{B^{\alpha-1}_{\infty,\infty}}$ (see \eqref{E.BesovF}) can be estimated as
	\begin{align*}
		s\|P_<F\|_0 &\leq s \left\|P_<\chi(sD)f+\sum_{N=0}^{K+2}P_N\chi(sD)f\right\|_{0} \\[1mm]&\leq s \left[\|P_<\chi(sD)f\|_0+\sum_{N=0}^{K+2}\|P_N\chi(sD)f\|_0\right]\\[1mm]
		&\leq s\|\chi(sD)f\|_{{B^{\alpha-1}_{\infty,\infty}}}\left[ 1+\sum_{N=0}^{K+2}2^{(1-\alpha) N} \right]\\[1mm]
		&\lesssim s\|f\|_{{B^{\alpha-1}_{\infty,\infty}}}\cdot 2^{(1-\alpha)K}\\
		&\lesssim s^{\alpha}\|f\|_{{B^{\alpha-1}_{\infty,\infty}}}\,.
	\end{align*}
	
	To estimate the remaining terms, we start by noting that
	\[
	\varphi(2^{-N}D)F(x)=[\varphi(2^{-N}sD)f](sx)\,,
	\]
	so
	$
	\|P_NF\|_0= \|\varphi(2^{-N}sD)f\|_0
	$.
	As before, 
	$\varphi(2^{-N}sD):B^s_{p,q}\to B^s_{p,q}$ is also uniformly bounded on any given Besov space. 
	
	In the analysis, we distinguish two cases. First, if $2^{-N}s<\frac1{10}$ (there is no loss of generality assuming $s>10$), the support of $\varphi(2^{-N}s\,\cdot)$ is contained in $B(0,\frac15)$. As $\chi=1$ on this ball, we have \[
	\varphi(2^{-N}sD)= P_<\, \varphi(2^{-N}sD)\,.
	\]
	This implies that
	\begin{align*}
		\sup_{s2^N<\frac1{10}}2^{N(\alpha-1)}s\|\varphi(2^{-N}sD)f\|_{0}&= \sup_{s2^N<\frac1{10}}2^{N(\alpha-1)}s\|P_<\varphi(2^{-N}sD)f\|_{0}\\
		&\lesssim \|\varphi(2^{-N}sD)f\|_{B^{\alpha-1}_{\infty,\infty}}\sup_{s2^N<\frac1{10}}2^{N(\alpha-1)}s\\
		&\lesssim s^\alpha\|f\|_{B^{\alpha-1}_{\infty,\infty}}\,.
	\end{align*}
	
	Let us now assume that $2^{-N}s\geq\frac1{10}$ and write $s=:2^{-K}\tilde s$ as before. Then 
	$$
	\supp \varphi(2^{-N}s\,\cdot)\subset\{ 2^{N+K-2}<|\xi|< 2^{N+K+2}\}\,,
	$$
	so
	$$
	\varphi(2^{-N}sD)=\sum_{J=-2}^2 P_{N+K+J}\, \varphi(2^{-N}sD)\,.
	$$
	Therefore,
	\begin{align*}
		\sup_{s2^N\geq\frac1{10}}2^{N(\alpha-1)} &s\|\varphi(2^{-N}sD)f\|_{0}\leq\sup_{s2^N\geq\frac1{10}}2^{N(\alpha-1)}s\sum_{J=-2}^2\|P_{N+K+J}\varphi(2^{-N}sD)f\|_{0}\\
		& \leq\sup_{M\geq0}\sup_{s2^N\geq\frac1{10}}2^{N(\alpha-1)}s\sum_{J=-2}^2\|P_{N+K+J}\varphi(2^{-M}sD)f\|_{0}\\
		&\lesssim\sup_{M\geq0}\|\varphi(2^{-M}sD)f\|_{B^{\alpha-1}_{\infty,\infty}}\sup_{s2^N\geq\frac1{10}}2^{N(\alpha-1)}s\cdot2^{(N+K)(1-\alpha)}\\
		&\lesssim \|f\|_{B^{\alpha-1}_{\infty,\infty}} s^\alpha \,.
	\end{align*}
	
	Putting everything together, for any given $\alpha>0$, we then conclude
	\[
	s\|F\|_{B^{\alpha-1}_{\infty,\infty}}\lesssim s^{\alpha}\|f\|_{B^{\alpha-1}_{\infty,\infty}}\,.
	\]
	This completes the proof of the uniform estimate.
\end{proof}

\bibliographystyle{amsplain}

\end{document}